\documentclass[a4paper, 12pt]{article}

\usepackage{amsmath} 
\usepackage{theorem}
\usepackage{amssymb}
\usepackage{amsfonts}
\usepackage{latexsym}
\usepackage{amscd}
\usepackage{diagramb}
\input GrCalc4.sty
\usepackage{graphicx}
\usepackage{enumerate}

\usepackage{pxfonts}

\usepackage{stmaryrd}

\DeclareMathAlphabet{\mathpzc}{OT1}{pzc}{m}{it}

\usepackage{xspace,colortbl}
\usepackage{color}

\definecolor{verde}{rgb}{0.,0.7,0.}
\definecolor{indigo}{rgb}{.18, .34, .78}
\definecolor{indigo1}{rgb}{.18, .24, .78}
\definecolor{indigo2}{rgb}{.18, .14, .78}
\definecolor{indigo3}{rgb}{.18, 0., .78}
\definecolor{rojo}{rgb}{1,0,0}
\definecolor{negro}{rgb}{0,0,0}
\definecolor{lila}{rgb}{.46, .16, .78}
\definecolor{lila1}{rgb}{.46, .16, .86}
\definecolor{lila2}{rgb}{.56, .16, .86}
	\definecolor{lila3}{rgb}{.63, .16, .78}
\definecolor{lila4}{rgb}{.7, .16, .78}
\definecolor{lila5}{rgb}{.78, .26, .78}
\definecolor{lila6}{rgb}{.6, 0., .78}

\theoremstyle{plain}
\theoremheaderfont{\normalfont\bfseries}

\newtheorem{thm}{Theorem}[section]

\newtheorem{lma}[thm]{Lemma}
\newtheorem{cor}[thm]{Corollary}
\newtheorem{defn}[thm]{Definition}

\newtheorem{propdefn}[thm]{Proposition and Definition}

\theorembodyfont{\rmfamily}

\newtheorem{rem}[thm]{Remark}
\newtheorem{prop}[thm]{Proposition}
\newtheorem{ex}[thm]{Example}
\newcommand{\qed}{\hfill\quad\fbox{\rule[0mm]{0,0cm}{0,0mm}}  \par\bigskip}

\newcommand{\x}{\mbox{-}}

\newcommand{\rtr}{\hspace{-0,08cm}\triangleright}
\newcommand{\ltr}{\hspace{-0,08cm}\triangleleft}

\newcommand{\bEM}{{\rm bEM}}
\newcommand{\EM}{{\rm EM}}

\newcommand{\Bimnd}{{\rm Bimnd}}

\newcommand{\ot}{\otimes}

\newcommand{\C}{{\mathcal C}}

\newcommand{\A}{{\mathcal A}}
\newcommand{\B}{{\mathcal B}}
\newcommand{\E}{{\mathcal E}}

\newcommand{\crta}{\overline}

\newcommand{\Id}{\operatorname {Id}}

\newcommand{\ro}{\rho}

\newcommand{\Epsilon}{\varepsilon}

\def\K{{\mathcal K}}  

\newcommand{\Mod}{\operatorname{Mod}}

\newcommand{\cref}[1]{C.~\ref{c:#1}}

\newcommand{\exlabel}[1]{\label{ex:#1}}
\newcommand{\exref}[1]{Example~\ref{ex:#1}}
\newcommand{\lelabel}[1]{\label{le:#1}}

\newcommand{\eqlabel}[1]{\label{eq:#1}}
\newcommand{\equref}[1]{(\ref{eq:#1})}
\newcommand{\thlabel}[1]{\label{th:#1}}
\newcommand{\thref}[1]{Theorem~\ref{th:#1}}
\newcommand{\delabel}[1]{\label{de:#1}}
\newcommand{\deref}[1]{Definition~\ref{de:#1}}
\newcommand{\prlabel}[1]{\label{pr:#1}}
\newcommand{\prref}[1]{Proposition~\ref{pr:#1}}
\newcommand{\colabel}[1]{\label{co:#1}}
\newcommand{\coref}[1]{Corollary~\ref{co:#1}}

\newcommand{\rmlabel}[1]{\label{rm:#1}}
\newcommand{\rmref}[1]{Remark~\ref{rm:#1}}
\newcommand{\selabel}[1]{\label{se:#1}}
\newcommand{\seref}[1]{Section~\ref{se:#1}}
\newcommand{\sslabel}[1]{\label{ss:#1}}

\begin{document}

\date{}

\title{Paired wreaths: \\ towards a 2-categorical atlas \\ of cross products}
\author{Bojana Femi\'c \vspace{6pt} \\
{\small Facultad de Ingenier\'ia, \vspace{-2pt}}\\
{\small  Universidad de la Rep\'ublica} \vspace{-2pt}\\
{\small  Julio Herrera y Reissig 565,} \vspace{-2pt}\\
{\small  11 300 Montevideo, Uruguay}}

\maketitle

\begin{abstract}
After we introduced biwreaths and biwreath-like objects in our previous paper, in the present one we define paired wreaths. 
In a paired wreath there is a monad $B$ and a comonad $F$ over the same 
0-cell in a 2-category $\K$, so that $F$ is a left wreath around $B$ and $B$ is a right cowreath around $F$, and moreover, $FB$ is a bimonad in $\K$. 
The corresponding 1-cell $FB$ in the setting of a biwreath and a biwreath-like object was not necessarilly a bimonad. 
We obtain a 2-categorical version of the Radford biproduct and Sweedler's crossed (co)product, that are on one hand, both a biwreath and a biwreath-like object, respectively, 
and on the other hand, they are also paired wreaths. We show that many known crossed (bi)products in the literature are special cases of paired wreaths, including 
cocycle cross product bialgebras of Bespalov and Drabant in braided monoidal categories. This is a part of a project of constructing a kind of a 2-categorical atlas of 
all the known crossed (bi)products. We introduce a Hopf datum in $\K$ which contains part of the structure of a paired wreath. We define Yang-Baxter equations and 
naturality conditions of certain distributive laws in $\K$ and study when Hopf data are paired wreaths. From the notion of a Hopf datum new definitions of 
Yetter-Drinfel`d modules, (co)module (co)monads, 2-(co)cycles and (co)cycle twisted (co)actions in a 2-categorical setting are suggested. 

\bigbreak

\end{abstract}

\section{Introduction}

Lack and Street introduced wreaths in \cite{LS} as monads in the Eilenberg-Moore category $\EM^M(\K)$ of monads in a 2-category $\K$. 
This short and elegant definition of a wreath 
when unpacked provides the necessary 
data to recover crossed products in different algebraic settings, as it has been showed {\em e.g.} in \cite{LS, BC}.  With the motivation, on the one hand, to understand why 
certain structures in the examples obey the laws of a wreath, and on the other to investigate an analogue of a wreath in terms of bimonads, 
we introduced biwreaths and biwreath-like objects in \cite{Femic5}. For this aim we first defined bimonads in $\K$ following the line of \cite{Wisb} 
(following the line of \cite{McC, Moe}  
would require that the 0-cells of $\K$ posses a monoidal structure, which would be a restriction for our purposes). Then we defined biwreaths as bimonads in the 
Eilenberg-Moore category $\bEM(\K)$ of bimonads in $\K$. Biwreaths consist of 1-cells $B,F:\A\to\A$ both over the same 0-cell $\A$ in $\K$, together with some 2-cells 
satisfying certain axioms. Biwreaths are in particular wreaths, cowreaths, mixed wreaths and mixed cowreaths. 
With the notion of biwreath it becomes clear why the structures in some afore-mentioned  examples have the form they have. 
Though, in the biwreaths such as we defined them, 
some structure 2-cells necessarily turn out to be trivial (some (co)module structures and (co)cycles). 
In order to include the non-trivial structures in the study, we introduced biwreath-like objects and mixed biwreath-like objects. 
However, from all these constructions only in the {\em canonical biwreath} 
(in which the (co)wreath structure 2-cells are images of the embedding 2-functor $\Bimnd(\K)\hookrightarrow\bEM(\K)$ from the 2-category of bimonads in $\K$) 
the 1-cell $FB$ turns out to be a bimonad in $\K$. This is the Radford biproduct in 2-categorical terms. A biwreath-like object turns out to be the Sweedler's crossed (co)product, 
while a mixed biwreath-like object turns out to be 3-cocycle twisted comodule monad and 3-cycle twisted module comonad, all in the 2-categorical setting. 

In the present paper we study structures similar to the previous ones but in which $FB$ is a bimonad. 
While in a (left) biwreath $F$ is a left wreath and left cowreath around $B$, once we fix $F$ to be a left wreath around $B$, the only resting possibility for $FB$ to have a bimonad 
structure is to consider $B$ as a right cowreath around $F$. For this reason we call our new objects {\em paired wreaths}. 
2-categorical Radford biproduct and Sweedler's crossed (co)product, that are a biwreath and a biwreath-like object, respectively, fit also to the setting of a paired wreath. 
Paired wreaths represent a 2-categorical ``origin'' of many crossed (bi)products known in algebra. In particular, when $\K=\hat\C$ is the 2-category induced by a braided monoidal category $\C$, 
our paired wreaths and the consequent notion of a Hopf datum recover the cross product bialgebras and Hopf data from \cite{BD}. The latter, in turn, generalize 
to the setting of braided monoidal categories all cross products known to the date, \cite{Maj5, Maj6, Sch3, BD1}. Broadening of this ``atlas'' of crossed products in a 2-categorical setting 
in order to collect some other structures that appeared in the literature in the last years, is a topic of our future research. 

As we mentioned, every paired wreath is a Hopf datum. We introduce Yang-Baxter type equations and a sort of naturality for certain distributive laws in 2-categories in order to study 
under which conditions a Hopf datum is a paired wreath, \thref{sufficient for bimonad}. Along the way we analyze when a {\em monad Hopf datum} is a wreath, \prref{psi_1,2}, \prref{psi_4}. 

Paired wreaths and Hopf data give rise to new definitions of Yetter-Drinfel`d modules, (co)module (co)monads, 2-(co)cycles and (co)cycle twisted (co)actions in a 2-categorical setting. 
These new interpretations of the known notions may be used to study classical results in this new setting and in a convenient 2-category $\K$ ({\em e.g.} the one induced by the braided 
monoidal category of modules over a commutative ring), being thus a source for new research. 

\bigskip

We organized the paper as follows. In Section 2 we set the notation used, we recall some definitions and results from \cite{Femic5} and we introduce $\tau$-bimonads in 2-categories. 
Section 3 is devoted to the definition of paired wreaths and to the analysis of which structures underlie them. In Section 4 we define Hopf data collecting some consequent structures of 
paired wreaths and study when a Hopf datum is a paired wreath. Here is where we define Yang-Baxter type equations and the mentioned naturality conditions in 2-categories. We also give a 
list of examples of the known crossed (bi)products that come out as special cases of paired wreaths. Finally, we indicate the new definitions in a 2-categorical setting of the classically 
known objects, as we listed above.

\section{Preliminaries: notation and recalling some basic notions}

We assume the reader is familiar with braided monoidal categories and string diagram notation, to have basic knowledge of 2-categories, (co)monads and (co)wreaths in 
2-categories. For reference we recommend \cite{K, Be, Bo, St1, LS}. 
Throughout $\K$ will denote a 2-category, $\C$ will denote a (braided) monoidal category with braiding 
$\gbeg{2}{1}
\gbr \gnl
\gend$. By $\hat\C$ we will mean the 2-category induced by $\C$. 1-cells in $\hat\C$ are objects of $\C$ and 2-cells in $\hat\C$ are morphisms in $\C$. 
From here it is clear that a composition $YX$ of 1-cells $X: \A\to\B$ and $Y: \B\to\E$, 
where $\A, \B, \E$ are 0-cells in $\hat\C$, translates to the tensor product $X\ot Y$ in $\C$, that is, the order of 1-cells in the composition is opposite when seen from $\hat\C$ 
versus when seen from $\C$. A fortiori the composition of 2-cells is opposite, too, this means that the diagrams written in $\hat\C$ describing identities with 2-cells  
are left-right symmetric to the corresponding ones in $\C$. To simplify the interpretation of identities we will be working with we are not going to turn the diagrams 
from $\hat\C$ to read them in $\C$. Observe that the composition of composable 1-cells in $\K$ gives them a monoidal structure, thus it is justified to use string diagrams to write 
identities between 2-cells in $\K$ (which act on composable 1-cells). 
Multiplication and unit of a monoid, commultiplication and counit of a comonoid (both in $\C$ and $\K$) we write respectively: 
\begin{center}
\begin{tabular}{ccp{0.5cm}cccp{0.5cm}c}
{\footnotesize multiplication} & {\footnotesize unit} & & {\footnotesize comultiplication} & {\footnotesize counit} & & \\
 $\gbeg{2}{2}
\got{1}{} \gnl
\gmu \gnl
\gend$ & $\gbeg{1}{2}
\got{1}{} \gnl
\gu{1} \gnl
\gend$ & & $
\gbeg{2}{2}
\got{2}{} \gnl
\gcmu \gnl
\gend$ &  $
\gbeg{1}{2}
\got{1}{} \gnl
\gcu{1} \gnl
\gend$
\end{tabular}
\end{center}

\medskip

In \cite[Definition 2.3]{Femic5} we defined modules over a monad, comodules over a comonad and bimonads in a 2-category $\K$. In the same paper we used monadic and comonadic 
distributive laws in $\K$. 
We recall the respective definitions here.

\begin{defn} \delabel{distr}
Let $(\A, T, \mu, \eta)$ be a monad, $(\A, D, \Delta, \Epsilon)$ a comonad and $F:\A\to\A$ a 1-cell in $\K$.
\begin{enumerate}[(a)]
\item A 2-cell $\psi:TF\to FT$ in $\K$ is called a left monadic distributive law if identities \equref{psi laws} hold.
\item A 2-cell $\phi:FD\to DF$ in $\K$ is called a left comonadic distributive law if identities \equref{phi laws} hold.
\end{enumerate}
\vspace{-1,4cm}
\begin{center} \hspace{-0,2cm}
\begin{tabular}{p{7.2cm}p{0cm}p{8cm}}
\begin{equation}\eqlabel{psi laws}
\gbeg{3}{5}
\got{1}{T}\got{1}{T}\got{1}{F}\gnl
\gcl{1} \glmpt \gnot{\hspace{-0,34cm}\psi} \grmptb \gnl
\glmptb \gnot{\hspace{-0,34cm}\psi} \grmptb \gcl{1} \gnl
\gcl{1} \gmu \gnl
\gob{1}{F} \gob{2}{T}
\gend=
\gbeg{3}{5}
\got{1}{T}\got{1}{T}\got{1}{F}\gnl
\gmu \gcn{1}{1}{1}{0} \gnl
\gvac{1} \hspace{-0,34cm} \glmptb \gnot{\hspace{-0,34cm}\psi} \grmptb  \gnl
\gvac{1} \gcl{1} \gcl{1} \gnl
\gvac{1} \gob{1}{F} \gob{1}{T}
\gend;
\quad
\gbeg{2}{5}
\got{3}{F} \gnl
\gu{1} \gcl{1} \gnl
\glmptb \gnot{\hspace{-0,34cm}\psi} \grmptb \gnl
\gcl{1} \gcl{1} \gnl
\gob{1}{F} \gob{1}{T}
\gend=
\gbeg{3}{5}
\got{1}{F} \gnl
\gcl{1} \gu{1} \gnl
\gcl{2} \gcl{2} \gnl
\gob{1}{F} \gob{1}{T}
\gend
\end{equation} & &
\begin{equation}\eqlabel{phi laws}
\gbeg{3}{5}
\got{1}{F} \got{2}{D}\gnl
\gcl{1} \gcmu \gnl
\glmptb \gnot{\hspace{-0,34cm}\phi} \grmptb \gcl{1} \gnl
\gcl{1} \glmptb \gnot{\hspace{-0,34cm}\phi} \grmptb \gnl
\gob{1}{D} \gob{1}{D} \gob{1}{F}
\gend=
\gbeg{3}{5}
\got{2}{F} \got{1}{\hspace{-0,2cm}D}\gnl
\gcn{1}{1}{2}{2} \gcn{1}{1}{2}{2} \gnl
\gvac{1} \hspace{-0,34cm} \glmpt \gnot{\hspace{-0,34cm}\phi} \grmptb \gnl
\gvac{1} \hspace{-0,2cm} \gcmu \gcn{1}{1}{0}{1} \gnl
\gvac{1} \gob{1}{D} \gob{1}{D} \gob{1}{F}
\gend;
\quad
\gbeg{3}{5}
\got{1}{F} \got{1}{D} \gnl
\gcl{1} \gcl{1} \gnl
\glmptb \gnot{\hspace{-0,34cm}\phi} \grmptb \gnl
\gcu{1} \gcl{1} \gnl
\gob{3}{F}
\gend=
\gbeg{3}{5}
\got{1}{F} \got{1}{D} \gnl
\gcl{1} \gcl{1} \gnl
\gcl{2}  \gcu{1} \gnl
\gob{1}{F}
\gend
\end{equation}
\end{tabular}
\end{center} 
\end{defn}

\begin{defn} \delabel{bimonad} \cite[Definition 2.2]{Femic5}
A 1-cell $B:\A\to\A$ in $\K$ is called a {\em bimonad} if $(\A, B, \mu, \eta)$ is a monad, $(\A, B, \Delta, \Epsilon)$ a comonad and there is a left monadic and comonadic 
distributive law $\lambda: BB\to BB$ in $\K$ such that the following conditions are fulfilled:
$$
\gbeg{2}{3}
\got{1}{B} \got{1}{B} \gnl
\gcl{1} \gcl{1} \gnl
\gcu{1}  \gcu{1} \gnl
\gend=
\gbeg{2}{3}
\got{1}{B} \got{1}{B} \gnl
\gmu \gnl
\gvac{1} \hspace{-0,2cm} \gcu{1} \gnl
\gob{1}{}
\gend
\hspace{1,5cm}
\gbeg{2}{3}
\gu{1}  \gu{1} \gnl
\gcl{1} \gcl{1} \gnl
\gob{1}{B} \gob{1}{B}
\gend=
\gbeg{2}{3}
\gu{1} \gnl
\hspace{-0,34cm} \gcmu \gnl
\gob{1}{B} \gob{1}{B}
\gend
\hspace{1,5cm}
\gbeg{1}{2}
\gu{1} \gnl
\gcu{1} \gnl
\gob{1}{}
\gend=
\Id_{id_{\A}}
\hspace{1,5cm}
\gbeg{3}{5}
\got{1}{B} \got{3}{B} \gnl
\gwmu{3} \gnl
\gvac{1} \gcl{1} \gnl
\gwcm{3} \gnl
\gob{1}{B}\gvac{1}\gob{1}{B}
\gend=
\gbeg{4}{5}
\got{1}{B} \got{3}{B} \gnl
\gcl{1} \gwcm{3} \gnl
\glmptb \gnot{\hspace{-0,34cm}\lambda} \grmptb \gvac{1} \gcl{1} \gnl
\gcl{1} \gwmu{3} \gnl
\gob{1}{B} \gvac{1} \gob{2}{\hspace{-0,34cm}B.}
\gend
$$
\end{defn}

\begin{defn} \delabel{(co)modules} \cite[Definition 2.3]{Femic5}
Let $(\A, T, \mu, \eta)$ be a monad and $(\A, D, \Delta, \Epsilon)$ a comonad in $\K$.
\begin{enumerate}[(a)]
\item A 1-cell $F:\B\to\A$ in $\K$ is called a {\em left $T$-module} if there is a 2-cell $\nu: TF\to F$ such that $\nu(\mu\times\Id_F)=\nu(\Id_T\times\nu)$
and $\nu(\eta\times\Id_F)=\Id_F$ holds.
\item A 1-cell $F:\A\to\B$ in $\K$ is called a {\em right $T$-module} if there is a 2-cell $\nu: FT\to F$ such that $\nu(\Id_F\times\mu)=\nu(\nu\times\Id_T)$
and $\nu(\Id_F\times\eta)=\Id_F$ holds.
\item A 1-cell $F:\B\to\A$ in $\K$ is called a {\em left $D$-comodule} if there is a 2-cell $\lambda: F\to DF$ such that $(\Delta\times\Id_F)\lambda=(\Id_D\times\lambda)\lambda$
and $(\Epsilon\times\Id_F)\lambda=\Id_F$ holds.
\item A 1-cell $F:\A\to\B$ in $\K$ is called a {\em right $D$-comodule} if there is a 2-cell $\rho: F\to FD$ such that $(\Id_F\times\Delta)\rho=(\rho\times\Id_D)\rho$
and $(\Id_F\times\Epsilon)\rho=\Id_F$ holds.
\end{enumerate}
\end{defn}

\medskip

In \cite[Proposition 2.4]{Femic5} we proved:

\begin{prop}  \prlabel{distr-actions}

Let $F:\A\to\A$ be a 1-cell in $\K$.
\begin{enumerate}[(a)]
\item
Given a monad  $(\A, B, \mu, \eta)$ with a left monadic distributive law $\psi:BF\to FB$ and a 2-cell
$\Epsilon=
\gbeg{2}{1}
\got{1}{B} \gnl
\gcu{1} \gnl
\gend$ such that
$\gbeg{2}{2}
\got{1}{B} \got{1}{B} \gnl
\gcu{1}  \gcu{1} \gnl
\gend=
\gbeg{2}{3}
\got{1}{B} \got{1}{B} \gnl
\gmu \gnl
\gvac{1} \hspace{-0,2cm} \gcu{1} \gnl
\gend$ and
$\gbeg{1}{2}
\gu{1} \gnl
\gcu{1} \gnl
\gob{1}{}
\gend=
\Id_{id_{\A}}$
hold, then the 2-cell:
\begin{equation*} 
\gbeg{2}{3}
\got{1}{B} \got{1}{F} \gnl
\glm \gnl
\gob{3}{F}
\gend=
\gbeg{2}{4}
\got{1}{B} \got{1}{F} \gnl
\glmptb \gnot{\hspace{-0,34cm}\psi} \grmptb \gnl
\gcl{1} \gcu{1} \gnl
\gob{1}{F}
\gend
\end{equation*}
makes $F$ a left $B$-module.
\item
Given a comonad $(\A, B, \Delta, \Epsilon)$ with a left comonadic distributive law $\phi:FB\to BF$ and a 2-cell
$\eta=
\gbeg{2}{2}
\gu{1} \gnl
\gob{1}{B}
\gend$ such that
$\gbeg{2}{2}
\gu{1} \gu{1} \gnl
\gob{1}{B} \gob{1}{B}
\gend=
\gbeg{2}{3}
\gu{1} \gnl
\hspace{-0,34cm} \gcmu \gnl
\gob{1}{B} \gob{1}{B}
\gend$
and
$\gbeg{1}{2}
\gu{1} \gnl
\gcu{1} \gnl
\gend=
\Id_{id_{\A}}$ hold, then the 2-cell:
\begin{equation*} 
\gbeg{2}{3}
\got{3}{F} \gnl
\glcm \gnl
\gob{1}{B} \gob{1}{F}
\gend=
\gbeg{2}{4}
\got{1}{F} \gnl
\gcl{1} \gu{1} \gnl
\glmptb \gnot{\hspace{-0,34cm}\phi} \grmptb \gnl
\gob{1}{B} \gob{1}{F}
\gend
\end{equation*}
makes $F$ a left $B$-comodule.
\item In particular, given a bimonad $(\A, B, \mu, \eta, \Delta, \Epsilon, \lambda)$ and distributive laws $\psi:BF\to FB$ and $\phi:FB\to BF$ as in (a) and (b), the 2-cells
$\gbeg{2}{3}
\got{1}{B} \got{1}{F} \gnl
\glm \gnl
\gob{3}{F}
\gend$ and 
$\gbeg{2}{3}
\got{3}{F} \gnl
\glcm \gnl
\gob{1}{B} \gob{1}{F}
\gend$ \hspace{0,2cm} 
make $F$ a left $B$-module and a left $B$-comodule.
\end{enumerate}
\end{prop}

\subsection{Biwreaths versus biwreath-like objects} \selabel{bw vs bl}


In \cite{Femic5} we defined biwreaths and (mixed) biwreath-like objects, we resume them here. 

Firstly, we introduced the 2-category $\bEM(\K)$. It is the Eilenberg-Moore category of bimonads in $\K$. Without entering into details, let us say that 
2-cells in $\bEM(\K)$ are pairs $(\rho_M, \rho_C)$, where $\rho_M:F\to GB$ and $\rho_C: FB\to G$ are 2-cells in $\K$ (here $B$ is a 0-cell and $F$ and $G$ come from the involved 1-cells). 
We say that $\rho_M$ and $\rho_C$ are {\em canonical} if \equref{canonical} holds. Their {\em canonical restrictions} are given by \equref{restr}. \vspace{-0,4cm}
\begin{center} 
\begin{tabular}{p{7,6cm}p{0,6cm}p{6cm}}
\begin{equation} \eqlabel{canonical} 
\rho_M=
\gbeg{2}{5}
\got{1}{F} \gnl
\glmptb \gnot{\hspace{-0,34cm}\rho_M} \grmpb \gnl
\gcl{1} \gcu{1} \gnl
\gcl{1} \gu{1} \gnl
\gob{1}{G} \gob{1}{B} \gnl
\gend,\hspace{0,5cm}
\rho_C=
\gbeg{2}{5}
\got{1}{F} \got{1}{B} \gnl
\gcl{1} \gcu{1} \gnl
\gcl{1} \gu{1} \gnl
\glmptb \gnot{\hspace{-0,34cm}\rho_C} \grmp \gnl
\gob{1}{G} \gnl
\gend
\end{equation} &  & 
\begin{equation}\eqlabel{restr}
\gbeg{2}{4}
\got{1}{F} \gnl
\glmptb \gnot{\hspace{-0,34cm}\rho_M} \grmpb \gnl
\gcl{1} \gcu{1} \gnl
\gob{1}{G}  \gnl
\gend,\hspace{0,5cm}
\gbeg{2}{4}
\got{1}{F} \gnl
\gcl{1} \gu{1} \gnl
\glmptb \gnot{\hspace{-0,34cm}\rho_C} \grmp \gnl
\gob{1}{G} \gnl
\gend
\end{equation}
\end{tabular}
\end{center} 

A biwreath consists of a pair of 1-cells $B$ and $F$ in $\K$ where $B$ is a bimonad in $\K$ and
$F$ is a bimonad in the Eilenberg-Moore category $\bEM^M(\K)$ over the 0-cell $B$ therein. This means in particular that 
$F$ is a wreath (with 2-cells $\mu_M, \eta_M$), cowreath (with 2-cells $\Delta_C, \Epsilon_C$) -- we refer to the latter two as to ``straight structures'' --, 
mixed wreath (with 2-cells $\Delta_M, \Epsilon_M$), mixed cowreath (with 2-cells $\mu_C, \eta_C$) -- we refer to the latter two as to ``mixed structures'' -- all around $B$. 
Moreover, there is an additional 2-cell in $\bEM(\K)$ (given by 2-cells $\lambda_M, \lambda_C$ in $\K$) 
governing the ``monad-comonad'' compatibilities of $F$. 
We set the following notation:
$$
\gbeg{2}{3}
\got{1}{F} \got{1}{F} \gnl
\gmu \gnl
\gob{2}{F} \gnl
\gend:=
\gbeg{2}{4}
\got{1}{F} \got{1}{F} \gnl
\glmptb \gnot{\hspace{-0,34cm}\mu_M} \grmptb \gnl
\gcl{1} \gcu{1} \gnl
\gob{1}{F} \gnl
\gend \hspace{1,3cm}
\gbeg{1}{4}
\got{1}{} \gnl
\gu{1} \gnl
\gcl{1} \gnl
\gob{1}{F} \gnl
\gend:=
\gbeg{1}{4}
\got{1}{} \gnl
\glmpb \gnot{\hspace{-0,34cm}\eta_M} \grmpb \gnl
\gcl{1} \gcu{1} \gnl
\gob{1}{F} \gnl
\gend \hspace{1,3cm}
\gbeg{2}{3}
\got{2}{F} \gnl
\gcmu \gnl
\gob{1}{F} \gob{1}{F} \gnl
\gend:=
\gbeg{3}{4}
\got{3}{F} \gnl
\glmpb \gnot{\Delta_M} \gcmptb \grmpb \gnl
\gcl{1} \gcl{1} \gcu{1} \gnl
\gob{1}{F} \gob{1}{F} \gnl
\gend \hspace{1,3cm}
\gbeg{1}{3}
\got{1}{F} \gnl
\gcl{1} \gnl
\gcu{1} \gnl
\gob{1}{} \gnl
\gend:=
\gbeg{1}{3}
\got{1}{F} \gnl
\gbmp{\Epsilon_{M}} \gnl
\gcu{1} \gnl
\gob{1}{} \gnl
\gend \hspace{1,3cm}
\sigma:=
\gbeg{2}{4}
\got{1}{F} \got{1}{F} \gnl
\glmptb \gnot{\hspace{-0,34cm}\mu_M} \grmptb \gnl
\gcu{1} \gcl{1} \gnl
\gob{3}{B} \gnl
\gend 
$$
and similarly for the comonadic structures.

From the structure of a biwreath it follows that $F$ is a proper left $B$-module and comodule 
in $\K$. Moreover, via the (co)unit 2-cells in the Eilenberg-Moore 2-categories (``pre-(co)units'') and the distributive laws $\psi$ and $\phi$ coming from 
the wreath and the cowreath structure, we define left and right $F$-action and coaction on $B$ in a broader sense. Namely, the left $F$-(co)module structures on $B$ are 
twisted by a so-called 3-(co)cycle (coming from ``mixed structures''), while its right $F$-(co)module structures are twisted by a 2-(co)cycle (coming from ``straight structures''). 

From the multiplication-comultiplication compatibility governed by $\lambda$'s (we recall the one for $\lambda_M$ in \equref{8 lambda_M}, this is identity (50) in \cite{Femic5}) we get the following: 
when the ``mixed (co)multiplication 2-cells'' $\mu_C$ and $\Delta_M$ are canonical we have 
precise non-canonical forms of the ``straight (co)multiplication 2-cells'' $\mu_M$ and $\Delta_C$, respectively 
(see \equref{mu_M with sigma}, where we also assumed that $\lambda_M$ is canonical). 
Also the other way around: when the straight (co)multiplication 2-cells are canonical, we obtain non-canonical forms of the mixed (co)multiplication 2-cells. 
\vspace{-0,6cm}
\begin{center} 
\begin{tabular} {p{5.5cm}p{0cm}p{6cm}} 
\begin{equation}\eqlabel{8 lambda_M}
\gbeg{4}{6}
\gvac{2} \got{1}{F} \got{1}{F} \gnl
\gvac{2} \glmptb \gnot{\hspace{-0,34cm}\mu_M} \grmptb \gnl
\glmpb \gnot{\Delta_M} \gcmpb \grmptb \gcl{1} \gnl
\gcl{2} \gcl{2} \gmu \gnl
\gvac{2} \gcn{1}{1}{2}{2} \gnl
\gob{1}{F} \gob{1}{F} \gob{2}{B} \gnl
\gend=
\gbeg{5}{8}
\gvac{1} \got{1}{F} \got{1}{F} \gnl
\gvac{1} \gcl{1} \glmptb \gnot{\Delta_M} \gcmpb \grmpb \gnl
\glmpb \gnot{\lambda_M} \gcmptb \grmptb \gcl{1} \gcl{2} \gnl
\gcl{1} \gcl{1} \glmptb \gnot{\hspace{-0,34cm}\psi} \grmptb \gnl
\gcl{2} \glmptb \gnot{\hspace{-0,34cm}\mu_M} \grmptb \gmu \gnl
\gcl{1} \gcl{1} \gcl{1} \gcn{1}{1}{2}{1} \gnl
\gcl{1} \gcl{1}  \gmu \gnl
\gob{1}{F} \gob{1}{F} \gob{2}{B} \gnl
\gend
\end{equation} & & \vspace{0,8cm}
\begin{equation} \eqlabel{mu_M with sigma}
\mu_M=
\gbeg{3}{5}
\got{1}{F} \got{2}{F} \gnl
\gcl{1} \gcmu \gnl
\glmptb \gnot{\hspace{-0,34cm}\lambda_F} \grmptb \gcl{1} \gnl
\gcl{1} \glmpt \gnot{\hspace{-0,34cm}\sigma} \grmptb \gnl
\gob{1}{F} \gob{3}{B} \gnl
\gend\end{equation}
\end{tabular}
\end{center}
\vspace{-0,6cm}
Now, since we consider the (co)unit 2-cells in the Eilenberg-Moore 2-categories canonical, by the above definitions the left and right $F$-(co)module structures 
on $B$ must be trivial, and the 2- and 3-(co)cycles are trivial, too. 
This means that the proper biwreath will have only left $B$-(co)module structures on $F$. Moreover, since the 2--(co)cycles are trivial, if the ``mixed (co)multiplication 2-cells'' 
are canonical, so are the ``straight'' ones (see \equref{mu_M with sigma}), and vice versa. Consequently, when all the four latter structure 2-cells of $F$ are canonical, $F$ is a proper monad and comonad. 
If $\lambda$'s are canonical, too, $F$ is indeed a bimonad in $\K$. 
This follows from the $\Delta_M-\mu_M$ compatibility \equref{8 lambda_M}, which, in the case that all the 2-cells are canonical, becomes: 
$$
\gbeg{2}{4}
\got{1}{F} \got{1}{F} \gnl
\gmu \gnl
\gcmu  \gnl
\gob{1}{F}\gob{1}{F} 
\gend=
\gbeg{3}{5}
\got{1}{F} \got{2}{F} \gnl
\gcl{1} \gcmu \gnl
\glmptb \gnot{\hspace{-0,34cm}\lambda_F} \grmptb  \gcl{1} \gnl
\gcl{1} \gmu  \gnl
\gob{1}{F}  \gob{2}{F.} 
\gend
$$
Furthermore, the left monadic distributive law $\psi: BF\to FB$ and the left comonadic distributive law $\phi: FB\to BF$ (both with respect to $B$), turn out to be both monadic 
and comonadic (at an appropriate side) with respect to $F$. Besides, the pairs of 2-cells $(\psi,\lambda_M)$ and $(\phi,\lambda_C)$ satisfy Yang-Baxter type equations. 
Applying (co)unit of $B$ to the afore-mentioned (co)monadic distributive laws of $\psi$ and $\phi$ with respect to $F$, we obtained identities (63--72) in \cite{Femic5}, 
stating that $F$ is a left $B$-(co)module (co)monad in $\K$. 

It is an interesting question which biwreaths one can get if none of these (co)multiplication 2-cells is canonical, we have not investigated yet such examples of biwreaths. 

From the above said we may state: 

\begin{propdefn}
A biwreath $(B,F)$ in $\K$ for which the following holds: 
\begin{itemize}
\item either the mixed or the straight structures are canonical;
\item the canonical restrictions of monadic and comonadic components of the 2-cells coincide; 
\end{itemize}
consists of the following data: 
\begin{enumerate}
\item left bimonads $B$ and $F$ in $\K$;
\item a left monadic distributive law $\psi: BF\to FB$ and a left comonadic distributive law $\phi: FB\to BF$ (both with respect to $B$), which moreover are monadic  
and comonadic (at an appropriate side) with respect to $F$, such that the identities: 
$$
\gbeg{3}{5}
\got{1}{B} \got{1}{F} \got{1}{B} \gnl
\glmptb \gnot{\hspace{-0,34cm}\psi} \grmptb \gcl{1} \gnl
\gcl{1} \glmptb \gnot{\hspace{-0,34cm}\lambda_B} \grmptb \gnl
\glmptb \gnot{\hspace{-0,34cm}\phi} \grmptb \gcl{1} \gnl
\gob{1}{B} \gob{1}{F} \gob{1}{B}
\gend=
\gbeg{3}{5}
\got{1}{B} \got{1}{F} \got{1}{B} \gnl
\gcl{1} \glmptb \gnot{\hspace{-0,34cm}\phi} \grmptb \gnl
\glmptb \gnot{\hspace{-0,34cm}\lambda_B} \grmptb \gcl{1} \gnl
\gcl{1} \glmptb \gnot{\hspace{-0,34cm}\psi} \grmptb \gnl
\gob{1}{B} \gob{1}{F} \gob{1}{B}
\gend\hspace{1,4cm}
\gbeg{3}{5}
\got{1}{B} \got{1}{F} \got{1}{F} \gnl
\glmptb \gnot{\hspace{-0,34cm}\psi} \grmptb \gcl{1} \gnl
\gcl{1} \glmptb \gnot{\hspace{-0,34cm}\psi} \grmptb \gnl
\glmptb \gnot{\hspace{-0,34cm}\lambda_F} \grmptb \gcl{1} \gnl
\gob{1}{F} \gob{1}{F} \gob{1}{B}
\gend=
\gbeg{3}{5}
\got{1}{B} \got{1}{F} \got{1}{F} \gnl
\gcl{1} \glmptb \gnot{\hspace{-0,34cm}\lambda_F} \grmptb \gnl
\glmptb \gnot{\hspace{-0,34cm}\psi} \grmptb \gcl{1} \gnl
\gcl{1} \glmptb \gnot{\hspace{-0,34cm}\psi} \grmptb \gnl
\gob{1}{F} \gob{1}{F} \gob{1}{B}
\gend\hspace{1,4cm}
\gbeg{3}{5}
\got{1}{F} \got{1}{F} \got{1}{B} \gnl
\gcl{1} \glmptb \gnot{\hspace{-0,34cm}\phi} \grmptb \gnl
\glmptb \gnot{\hspace{-0,34cm}\phi} \grmptb \gcl{1} \gnl
\gcl{1} \glmptb \gnot{\hspace{-0,34cm}\lambda_F} \grmptb \gnl
\gob{1}{B} \gob{1}{F} \gob{1}{F}
\gend=
\gbeg{3}{5}
\got{1}{F} \got{1}{F} \got{1}{B} \gnl
\glmptb \gnot{\hspace{-0,34cm}\lambda_F} \grmptb \gcl{1} \gnl
\gcl{1} \glmptb \gnot{\hspace{-0,34cm}\phi} \grmptb \gnl
\glmptb \gnot{\hspace{-0,34cm}\phi} \grmptb \gcl{1} \gnl
\gob{1}{B} \gob{1}{F} \gob{1}{F}
\gend
$$
hold, where $\lambda_B:BB\to BB$ and $\lambda_F:FF\to FF$ are left monadic and comonadic distributive laws making $B$ and $F$ left bimonads. 
\end{enumerate}

In particular, $F$ is: 
\begin{itemize}
\item a left Yetter-Drinfel`d module;
\item a left $B$-module monad;
\item a left $B$-module comonad;
\item a left $B$-comodule monad and
\item a left $B$-comodule comonad 
\end{itemize}
in $\K$. 

A biwreath described above we will call {\em canonical biwreath}. 
\end{propdefn}

The novel notions in a 2-category listed at the end of the above Definition came out as a natural consequence of the notion of a biwreath.

\begin{ex} \exlabel{Example 1}
In \cite[Section 5.1]{Femic5} we proved that a canonical biwreath in $\K=\hat\C$ on a pair of bimonads $B,F$ in $\K$ is given by the 2-cells, {\em i.e.} morphisms in $\C$: 
\begin{equation} \eqlabel{lambda Radford}
\psi=
\gbeg{3}{5}
\got{2}{B} \got{1}{F} \gnl
\gcmu \gcl{1} \gnl
\gcl{1} \gbr \gnl
\glm \gcl{1} \gnl
\gob{1}{} \gob{1}{F} \gob{1}{B} 
\gend 
\hspace{1cm}
\phi= 
\gbeg{3}{5}
\got{1}{} \got{1}{F} \got{1}{B} \gnl
\glcm \gcl{1} \gnl
\gcl{1} \gbr \gnl
\gmu \gcl{1} \gnl
\gob{2}{B} \gob{1}{F} 
\gend
\hspace{1cm}
\lambda_{Rad}=
\gbeg{4}{7}
\got{1}{} \got{1}{F} \got{3}{F} \gnl
\gwcm{3} \gcl{2} \gnl
\gcl{1} \glcm \gnl
\gcl{1} \gcl{1} \gbr \gnl
\gcl{1} \glm \gcl{2} \gnl
\gwmu{3} \gnl
\gvac{1} \gob{1}{F} \gob{3}{F.}
\gend
\end{equation}
This canonical biwreath recovers the Radford biproduct $B\ot F$ in $\C$. 
\end{ex}

As we exposed above, given that in a biwreath the rest of (co)module structures between $B$ and $F$ are trivial, as well as the 2- and 3-(co)cycles, 
we introduced biwreath-like objects (consisting of straight structures and an appropriate 2-cell $\lambda_F$) and 
mixed biwreath-like objects (consisting of mixed structures and an appropriate 2-cell $\lambda_F$) in \cite{Femic5}. In these new objects the above-mentioned structures are non-trivial. 
We recall here the definition of a biwreath-like object ($\EM^M(\K)$ and $\EM^C(\K)$ denote the Eilenberg-Moore categories for monads and comonads, respectively): 

\begin{defn} \cite[Definition 6.1]{Femic5} \delabel{bl object}
A biwreath-like object in $\K$ is a monad $(F, \mu_M, \eta_M)$ in $\EM^M(\K)$ and a comonad $(F, \Delta_C, \Epsilon_C)$ in $\EM^C(\K)$ over the same bimonad $B$ in $\K$ 
with the canonical restrictions: 
\begin{equation} \eqlabel{mon comon cell}
\gbeg{2}{3}
\got{1}{F} \got{1}{F} \gnl
\gmu \gnl
\gob{2}{F} \gnl
\gend:=
\gbeg{2}{4}
\got{1}{F} \got{1}{F} \gnl
\glmptb \gnot{\hspace{-0,34cm}\mu_M} \grmptb \gnl
\gcl{1} \gcu{1} \gnl
\gob{1}{F} \gnl
\gend \hspace{2cm}
\gbeg{1}{4}
\got{1}{} \gnl
\gu{1} \gnl
\gcl{1} \gnl
\gob{1}{F} \gnl
\gend:=
\gbeg{1}{4}
\got{1}{} \gnl
\glmpb \gnot{\hspace{-0,34cm}\eta_M} \grmpb \gnl
\gcl{1} \gcu{1} \gnl
\gob{1}{F} \gnl
\gend \hspace{2cm}
\gbeg{2}{3}
\got{2}{F} \gnl
\gcmu \gnl
\gob{1}{F} \gob{1}{F} \gnl
\gend:=
\gbeg{2}{4}
\got{1}{F} \gnl
\gcl{1} \gu{1} \gnl
\glmptb \gnot{\hspace{-0,34cm}\Delta_C} \grmptb \gnl
\gob{1}{F} \gob{1}{F} \gnl
\gend \hspace{2cm}
\gbeg{1}{3}
\got{1}{F} \gnl
\gcl{1} \gnl
\gcu{1} \gnl
\gob{1}{} \gnl
\gend:=
\gbeg{1}{3}
\got{1}{F} \gnl
\gcl{1} \gu{1} \gnl
\glmpt \gnot{\hspace{-0,34cm}\Epsilon_C} \grmpt \gnl
\gob{1}{} \gnl
\gend
\end{equation}
equipped with a left monadic and comonadic distributive law $\lambda_F:FF\to FF$ with respect to the structure 2-cells \equref{mon comon cell} so that 
the following compatibility conditions are fulfilled: 
\begin{equation} \eqlabel{lambda mixed 1-3}
\gbeg{2}{3}
\got{1}{F} \got{1}{F} \gnl
\gcl{1} \gcl{1} \gnl
\gcu{1} \gcu{1} \gnl
\gend=
\gbeg{2}{3}
\got{1}{F} \got{1}{F} \gnl
\gmu \gnl
\gvac{1} \hspace{-0,22cm} \gcu{1} \gnl
\gend \hspace{2cm}
\gbeg{1}{4}
\got{2}{}  \gnl
 \gu{1} \gnl
\hspace{-0,34cm} \gcmu \gnl
\gob{1}{F} \gob{1}{F} \gnl
\gend=
\gbeg{2}{4}
\got{1}{} \gnl
\gu{1} \gu{1} \gnl
\gcl{1} \gcl{1} \gnl
\gob{1}{F} \gob{1}{F}
\gend  \hspace{2cm}
\gbeg{1}{2}
\gu{1} \gnl
\gcu{1} \gnl
\gob{2}{} \gnl
\gend=
\Id_{id_{\A}}
\end{equation}
and
\begin{center} \hspace{-0,2cm}
\begin{tabular}{p{5.6cm}p{0cm}p{6cm}}
\begin{equation} \eqlabel{lambda_M 8F}
\gbeg{2}{5}
\got{1}{F} \got{1}{F} \gnl
\gcl{1} \gcl{1} \gnl
\glmptb \gnot{\hspace{-0,34cm}\mu_M} \grmptb \gnl
\hspace{-0,22cm} \gcmu \gcn{1}{1}{0}{1} \gnl
\gob{1}{F} \gob{1}{F} \gob{1}{B} \gnl
\gend=
\gbeg{3}{5}
\got{1}{F} \got{2}{F} \gnl
\gcl{1} \gcmu \gnl
\glmptb \gnot{\hspace{-0,34cm}\lambda_F} \grmptb \gcl{1} \gnl
\gcl{1} \glmptb \gnot{\hspace{-0,34cm}\mu_M} \grmptb \gnl
\gob{1}{F} \gob{1}{F} \gob{1}{B} \gnl
\gend
\end{equation} & & 
\begin{equation}\eqlabel{lambda_C 8F}
\gbeg{3}{4}
\got{1}{F} \got{1}{F} \got{1}{B} \gnl
\gmu \gcn{1}{1}{1}{0} \gnl
\gvac{1} \hspace{-0,34cm} \glmptb \gnot{\hspace{-0,34cm}\Delta_C} \grmptb \gnl
\gvac{1} \gob{1}{F} \gob{1}{F} \gnl
\gend=
\gbeg{3}{5}
\got{1}{F} \got{1}{F} \got{1}{B} \gnl
\gcl{1} \glmptb \gnot{\hspace{-0,34cm}\Delta_C} \grmptb \gnl
\glmptb \gnot{\hspace{-0,34cm}\lambda_F} \grmptb \gcl{1} \gnl
\gcl{1} \gmu \gnl
\gob{1}{F} \gob{2}{F.} \gnl
\gend
\end{equation}
\end{tabular}
\end{center}
\end{defn}

Observe that the identity \equref{lambda_M 8F} can be obtained from \equref{8 lambda_M} assuming that $\Delta_M, \lambda_M$ are canonical.

Applying $\Epsilon_B$ to \equref{lambda_M 8F} (or applying $\eta_B$ to \equref{lambda_C 8F}), one gets \equref{bimonad in K F}, 
then $F$ is a bimonad in $\K$ with possibly non-(co)associative (co)multiplication. On the other hand, applying 
$F
\gbeg{1}{2}
\got{1}{F}\gnl
\gcu{1} \gnl
\gend B$ to \equref{lambda_M 8F} and 
$F
\gbeg{1}{2}
\gu{1} \gnl
\gob{1}{F}\gnl
\gend B$ to \equref{lambda_C 8F}, we obtain $\mu_M$ and $\Delta_C$ below:  \vspace{-0,7cm}
\begin{center} \hspace{-1,3cm} 
\begin{tabular}{p{4cm}p{0cm}p{4cm}p{0cm}p{4cm}}
\begin{equation} \eqlabel{bimonad in K F}
\gbeg{2}{4}
\got{1}{F} \got{1}{F} \gnl
\gmu \gnl
\gcmu \gnl
\gob{1}{F}\gob{1}{F} 
\gend=
\gbeg{3}{5}
\got{1}{F} \got{2}{F} \gnl
\gcl{1} \gcmu \gnl
\glmptb \gnot{\hspace{-0,34cm}\lambda_F} \grmptb  \gcl{1} \gnl
\gcl{1} \gmu \gnl
\gob{1}{F}  \gob{2}{F}
\gend
\end{equation} & & 
\begin{equation} \eqlabel{mu_M bl}
\mu_M=
\gbeg{3}{5}
\got{1}{F} \got{2}{F} \gnl
\gcl{1} \gcmu \gnl
\glmptb \gnot{\hspace{-0,34cm}\lambda_F} \grmptb \gcl{1} \gnl
\gcl{1} \glmpt \gnot{\hspace{-0,34cm}\sigma} \grmptb \gnl
\gob{1}{F} \gob{3}{B} \gnl
\gend
\end{equation} & & 
\begin{equation} \eqlabel{Delta_C bl}
\Delta_C=
\gbeg{3}{5}
\got{1}{F} \got{3}{B} \gnl
\gcl{1} \glmpb \gnot{\hspace{-0,34cm}\rho} \grmptb \gnl
\glmptb \gnot{\hspace{-0,34cm}\lambda_F} \grmptb \gcl{1} \gnl
\gcl{1} \gmu \gnl \gnl
\gob{1}{F} \gob{2}{F} \gnl
\gend
\end{equation}
\end{tabular}
\end{center}
where
\begin{center} \hspace{1,2cm}
\begin{tabular}{p{4cm}p{1cm}p{4cm}}
\gbeg{2}{4}
\got{1}{F} \got{1}{F} \gnl
\glmpt \gnot{\hspace{-0,34cm}\sigma} \grmptb \gnl
\gvac{1} \gcl{1} \gnl
\gob{3}{B}
\gend:=
\gbeg{2}{4}
\got{1}{F} \got{1}{F} \gnl
\glmptb \gnot{\hspace{-0,34cm}\mu_M} \grmptb \gnl
\gcu{1} \gcl{1} \gnl
\gob{3}{B}
\gend
& &
\gbeg{2}{4}
\got{3}{B} \gnl
\gvac{1} \gcl{1} \gnl
\glmpb \gnot{\hspace{-0,34cm}\rho} \grmptb \gnl
\gob{1}{F}\gob{1}{F}
\gend:=
\gbeg{2}{4}
\got{3}{B} \gnl
\gu{1} \gcl{1} \gnl
\glmptb \gnot{\hspace{-0,34cm}\Delta_C} \grmptb \gnl
\gob{1}{F}\gob{1}{F.}
\gend
\end{tabular}
\end{center}

\medskip

Now the following is obvious:

\begin{cor}
A canonical biwreath is a biwreath-like object.
\end{cor}

\begin{ex} \exlabel{Example 2}
In \cite[Section 6.1]{Femic5} we proved that in $\K=\hat\C$ a biwreath-like object on a pair of bimonads $B,F$ in $\K$ is given by the 2-cells, {\em i.e.} morphisms in $\C$: \vspace{-0,7cm} 
$$\psi=
\gbeg{3}{5}
\got{1}{B} \got{2}{F} \gnl
\gcl{1} \gcmu \gnl
\gbr \gcl{1} \gnl
\gcl{1} \grm \gnl
\gob{1}{F} \gob{1}{B} 
\gend
\hspace{1cm}
\phi= 
\gbeg{3}{5}
\got{1}{F} \got{1}{B} \gnl
\gcl{1} \grcm \gnl
\gbr \gcl{1} \gnl
\gcl{1} \gmu \gnl
\gob{1}{B} \gob{2}{F} 
\gend
\hspace{1cm}
\lambda=
\gbeg{3}{5}
\got{2}{F} \got{1}{F} \gnl
\gcmu \gcl{1} \gnl  
\gcl{1} \gbr  \gnl
\gmu \gcl{1} \gnl
\gob{2}{F} \gob{1}{F}
\gend \hspace{0,8cm} \Rightarrow \hspace{0,8cm}
\mu_M=
\gbeg{3}{5}
\got{2}{F} \got{2}{F} \gnl
\gcmu  \gcmu \gnl
\gcl{1} \gbr \gcl{1} \gnl
\gmu \glmpt \gnot{\hspace{-0,34cm}\sigma} \grmptb  \gnl
\gob{2}{F} \gob{3}{B} 
\gend
$$For $\C=R\x\Mod$, the category of modules over a commutative ring $R$, we recover Sweedler's normalized 2-cocycle $\sigma: F\ot F\to B$, twisted $F$-action on $B$ (``$F$ measures $B$'') 
and Sweedler's crossed product algebra $F\ot B$. 
\end{ex}

\subsection{$\tau$-bimonads}

In \deref{bimonad} we defined left bimonads in $\K$ - the definition involves left monadic and comonadic distributive law $\lambda$. A right bimonad in $\K$ is defined in the 
obvious way. 

\begin{defn} \delabel{tau-bim}
A monad and a comonad $B$ in $\K$ we call a {\em $\tau$-bimonad} if there is a 2-cell $\tau_{B,B}:BB\to BB$ which is a left and right monadic and comonadic distributive law 
such that the following compatibility conditions hold:
$$\gbeg{2}{3}
\got{1}{B} \got{1}{B} \gnl
\gcl{1} \gcl{1} \gnl
\gcu{1}  \gcu{1} \gnl
\gend=
\gbeg{2}{3}
\got{1}{B} \got{1}{B} \gnl
\gmu \gnl
\gvac{1} \hspace{-0,2cm} \gcu{1} \gnl
\gob{1}{}
\gend \hspace{1,6cm}
\gbeg{2}{3}
\gu{1}  \gu{1} \gnl
\gcl{1} \gcl{1} \gnl
\gob{1}{B} \gob{1}{B}
\gend=
\gbeg{2}{3}
\gu{1} \gnl
\hspace{-0,34cm} \gcmu \gnl
\gob{1}{B} \gob{1}{B}
\gend \hspace{1,6cm}
\gbeg{1}{2}
\gu{1} \gnl
\gcu{1} \gnl
\gob{1}{}
\gend=
\Id_{id_{\A}}  \hspace{1,6cm}
\gbeg{3}{5}
\got{1}{B} \got{3}{B} \gnl
\gwmu{3} \gnl
\gvac{1} \gcl{1} \gnl
\gwcm{3} \gnl
\gob{1}{B}\gvac{1}\gob{1}{B}
\gend=
\gbeg{2}{5}
\got{2}{B} \got{2}{B} \gnl
\gcmu \gcmu \gnl
\gcl{1} \glmptb \gnot{\hspace{-0,34cm}\tau_{B,B}} \grmptb \gcl{1} \gnl
\gmu \gmu \gnl
\gob{2}{B} \gob{2}{B.} \gnl
\gend
$$
\end{defn}

It is straightforward to show that a $\tau$-bimonad $B$ is both a left and a right bimonad with the corresponding 2-cells $\lambda_l$ and $\lambda_r$ being given by: 
$$\lambda_l=
\gbeg{4}{5}
\gvac{1} \got{1}{B} \got{3}{B} \gnl
\gwcm{3} \gcl{1} \gnl
\gcl{1} \gvac{1} \glmptb \gnot{\hspace{-0,34cm}\tau_{B,B}} \grmptb \gnl
\gwmu{3} \gcl{1} \gnl
\gvac{1} \gob{1}{B} \gvac{1} \gob{2}{\hspace{-0,34cm}B}
\gend \hspace{1cm}\textnormal{and}\hspace{1cm}
\lambda_r=
\gbeg{4}{5}
\got{1}{B} \got{3}{B} \gnl
\gcl{1} \gwcm{3} \gnl
\glmptb \gnot{\hspace{-0,34cm}\tau_{B,B}} \grmptb \gvac{1} \gcl{1} \gnl
\gcl{1} \gwmu{3} \gnl
\gob{1}{B} \gvac{1} \gob{2}{\hspace{-0,34cm}B.}
\gend 
$$
respectively.

\subsection{When a biwreath-like object is a $\tau$-bimonad?} \sslabel{when is bimonad}

In a biwreath-like object $(B,F)$ the 1-cell $FB$ is a monad and a comonad by the wreath product and the cowreath coproduct structures, these are given by:
\begin{equation}  \eqlabel{wreath (co)product - old}
\nabla_{FB}=
\gbeg{3}{5}
\got{1}{F} \got{1}{B} \got{1}{F} \got{3}{B}  \gnl
\gcl{1}  \glmptb \gnot{\hspace{-0,34cm}\psi} \grmptb \gvac{1} \gcl{1} \gnl
\glmptb \gnot{\hspace{-0,34cm}\mu_M} \grmptb \gwmu{3} \gnl
\gcl{1} \gwmu{3} \gnl
\gob{1}{F} \gvac{1} \gob{1}{B}
\gend\hspace{1,5cm}
\eta_{FB}=
\gbeg{2}{3}
\glmpb \gnot{\hspace{-0,34cm}\eta_M} \grmpb \gnl
\gcl{1} \gcl{1} \gnl
\gob{1}{F} \gob{1}{B}
\gend; \hspace{1,5cm}
\Delta_{FB}=
\gbeg{3}{5}
\got{1}{F} \got{3}{B} \gnl
\gcl{1} \gwcm{3} \gnl
\glmptb \gnot{\hspace{-0,34cm}\Delta_C} \grmptb \gwcm{3} \gnl
\gcl{1}  \glmptb \gnot{\hspace{-0,34cm}\phi} \grmptb \gvac{1} \gcl{1} \gnl
\gob{1}{F} \gob{1}{B} \got{1}{F} \got{3}{B}
\gend\hspace{1,5cm}
\Epsilon_{FB}=
\gbeg{2}{4}
\got{1}{F} \got{1}{B} \gnl
\gcl{1} \gcl{1} \gnl
\glmpt \gnot{\hspace{-0,34cm}\Epsilon_C} \grmpt \gnl
\gend
\end{equation}
(see \cite{LS}). Let $\tau_{B,F}, \tau_{F,F}, \tau_{F,F}, \tau_{F,B}$ be left and right monadic and comonadic distributive laws. Here the term ``(co)monadic'' with respect to $F$ 
is meant with respect to the canonical restrictions of the structure 2-cells of $F$, recall \equref{mon comon cell}.  
It is straightforward to show that 
\begin{equation}  \eqlabel{tau FB}
\tau_{FB,FB}:=
\gbeg{3}{5}
\got{1}{F} \got{1}{B} \got{1}{F} \got{1}{B} \gnl
\gcl{1} \glmptb \gnot{\hspace{-0,34cm}\tau_{B,F}} \grmptb  \gcl{1} \gnl
\glmptb \gnot{\hspace{-0,34cm}\tau_{F,F}} \grmptb  \glmptb \gnot{\hspace{-0,34cm}\tau_{B,B}} \grmptb  \gnl
\gcl{1} \glmptb \gnot{\hspace{-0,34cm}\tau_{F,B}} \grmptb  \gcl{1} \gnl
\gob{1}{F} \gob{1}{B} \gob{1}{F} \gob{1}{B}
\gend
\end{equation}
is a left and right monadic and comonadic distributive law for $FB$. 

In \cite[Theorem 5.3]{Femic5} we proved that in a canonical biwreath in $\K=\hat\C$ (with suitable $\psi, \phi$ and $\lambda_F$) $F\ot B$ is a Radford biproduct 
if and only if $F\ot B$ is a bialgebra in $\C$. For the rest of biwreaths and for biwreath-like objects we could not prove that the respective 1-cells $FB$ are bimonads. 
In the next section we introduce another concept, which is an alternative definition of a biwreath-like object and it has a structure of a bimonad in $\K$.

\section{Paired wreaths}  \selabel{Paired}

To the (mixed) (co)wreaths that we were dealing with in \cite{Femic5} and were speaking about in \seref{bw vs bl}, we refer to as to {\em left} ones. 
Before we proceed, let us list the notation of the structure 2-cells that we are using for left and right (co)wreaths. 

\medskip

In a left wreath $F$ around a monad $B$ in $\K$ we have 2-cells $\psi: BF\to FB, \mu_M:FF\to FB, \eta_M:\Id_{\A}\to FB$ so that $\psi$ is a left monadic distributive law and 
further 5 axioms are fulfilled (all these will be detailed later on). Then $F$ is a right wreath around a monad $B$ in $\K$ if it is equipped with 2-cells 
$\psi': FB\to BF, \mu_M':FF\to BF, \eta_M':\Id_{\A}\to BF$ so that $\psi'$ is a right monadic distributive law and further 5 axioms, left-right symmetric to the latter 5 axioms, are fulfilled. 

A left cowreath $F$ around a comonad $B$ in $\K$ is given by 2-cells $\phi: FB\to BF, \Delta_C: FB\to FF, \Epsilon_C: FB\to\Id_{\A}$ where $\phi$ is a left comonadic 
distributive law and further 5 axioms are fulfilled. These axioms are up-down symmetric to the axioms of a left wreath around a monad. Lastly, a right 
cowreath $F$ around a comonad $B$ in $\K$ is given by 2-cells $\phi': BF\to FB, \Delta_C': BF\to FF, \Epsilon_C': BF\to\Id_{\A}$ so that $\phi'$ is right comonadic distributive law 
and 5 further axioms are fulfilled, which are left-right symmetric those of a left cowreath. 

\medskip

Our next objective is to study biwreath-like structures in $\K$ that would cover the most general family of the known crossed products and simultaneously are such that the 
composit 1-cell $FB$ is a bimonad in $\K$. For the latter, we will consider $FB$ a monad as in \equref{wreath (co)product - old}, that is, $F$ is a left wreath around $B$, 
but for the comonad structure we take the only resting possibility different than in \equref{wreath (co)product - old}: that $B$ is a right cowreath around $F$. 
Notice that this means that the comonadic diagrams we had in a biwreath-like object will turn into their left-right 
symmetric versions and moreover the r\^oles of $B$ and $F$ will be interchanged. Thus, for example, the 2-cell $\Delta_C: FB\to FF$ will turn into $\Delta_C': FB\to BB$, and so on. 
This type of symmetry we will call {\em $\alpha$-symmetry} (we will use this term again from \prref{psi_4} on). The symmetry obtained by composing up-down symmetry (rotating 
string diagrams by $\pi$) with the $\alpha$-symmetry we will call {\em $\pi$-symmetry}, in accordance with \cite{BD}. 

As for generalizing the known crossed products, in particular, for $\K=\hat\C$ we want to cover most general forms of the morphisms $\psi, \phi, \lambda$ of those studied in 
\exref{Example 1}, \exref{Example 2}. 
In order to prove in \cite[Lemma 5.2]{Femic5} that $\lambda_{Rad}$ from \equref{lambda Radford} satisfies the necessary left monadic and comonadic distributive laws,  
we needed to assume that $\mu_M$ is canonical or that $F$ is a trivial left $B$-module,  
and that $F$ is a proper left $B$-comodule. In a general biwreath-like structure none of the latter three conditions need be satisfied. For this reason we will drop the 
condition that $\lambda$ should be a proper distributive law from the coming definition. Moreover, the condition \equref{8 lambda_M} on $\lambda$ 
we will substitute by its consequence \equref{mu_M bl} and another identity, which on one hand will yield \equref{bimonad in K F}, but also a concrete form of the 2-cell $\lambda$, on the other hand. 
Observe that the following definition is auto $\pi$-symmetric. 

\begin{defn} \delabel{paired bl object} 
A paired wreath in $\K$ consists of the following data: 
\begin{enumerate}
\item a monad $B$ with a 2-cell $\Epsilon_B=
\gbeg{1}{2}
\got{1}{B} \gnl
\gcu{1} \gnl
\gend$ in $\K$ and a comonad $F$ with a 2-cell $\eta_F=
\gbeg{1}{2}
\gu{1} \gnl
\gob{1}{F} \gnl
\gend$ in $\K$, both over the same 0-cell $\A$ in $\K$, satisfying the compatibility conditions:
\begin{equation} \eqlabel{(co)unital}
\gbeg{2}{3}
\got{1}{B} \got{1}{B} \gnl
\gmu \gnl
\gvac{1} \hspace{-0,22cm} \gcu{1} \gnl
\gend=
\gbeg{2}{3}
\got{1}{B} \got{1}{B} \gnl
\gcl{1} \gcl{1} \gnl
\gcu{1} \gcu{1} \gnl
\gend \hspace{1,5cm}
\gbeg{1}{2}
\gu{1} \gnl
\gcu{1} \gnl
\gob{2}{} \gnl
\gend B=
\Id_{id_{\A}}
\hspace{2cm}
\gbeg{1}{3}
 \gu{1} \gnl
\hspace{-0,34cm} \gcmu \gnl
\gob{1}{F} \gob{1}{F} \gnl
\gend=
\gbeg{2}{3}
\gu{1} \gu{1} \gnl
\gcl{1} \gcl{1} \gnl
\gob{1}{F} \gob{1}{F}
\gend  \hspace{1,5cm}
\gbeg{1}{2}
\gu{1} \gnl
\gcu{1} \gnl
\gob{2}{} \gnl
\gend F=
\Id_{id_{\A}}
\end{equation}
\item so that $(F, \psi: BF\to FB, \mu_M, \eta_M)$ is a left wreath around $B$ and $(B, \phi': FB\to BF, \Delta_C', \Epsilon_C')$ is a right cowreath around $F$  
with the canonical restrictions: 
\begin{equation} \eqlabel{pre-(co)mult}
\gbeg{2}{3}
\got{1}{F} \got{1}{F} \gnl
\gmu \gnl
\gob{2}{F} \gnl
\gend:=
\gbeg{2}{4}
\got{1}{F} \got{1}{F} \gnl
\glmptb \gnot{\hspace{-0,34cm}\mu_M} \grmptb \gnl
\gcl{1} \gcu{1} \gnl
\gob{1}{F} \gnl
\gend \hspace{2cm}
\gbeg{1}{4}
\got{1}{} \gnl
\gu{1} \gnl
\gcl{1} \gnl
\gob{1}{F} \gnl
\gend=
\gbeg{1}{4}
\got{1}{} \gnl
\glmpb \gnot{\hspace{-0,34cm}\eta_M} \grmpb \gnl
\gcl{1} \gcu{1} \gnl
\gob{1}{F} \gnl
\gend \hspace{2cm}
\gbeg{2}{3}
\got{2}{B} \gnl
\gcmu \gnl
\gob{1}{B} \gob{1}{B} \gnl
\gend:=
\gbeg{2}{4}
\got{3}{B} \gnl
\gu{1} \gcl{1} \gnl
\glmptb \gnot{\hspace{-0,34cm}\Delta_C'} \grmptb \gnl
\gob{1}{B} \gob{1}{B} \gnl
\gend \hspace{2cm}
\gbeg{1}{3}
\got{1}{B} \gnl
\gcl{1} \gnl
\gcu{1} \gnl
\gob{1}{} \gnl
\gend=
\gbeg{1}{3}
\got{3}{B} \gnl
\gu{1} \gcl{1} \gnl
\glmpt \gnot{\hspace{-0,34cm}\Epsilon_C'} \grmpt \gnl
\gob{1}{} \gnl
\gend
\end{equation}
where $\gbeg{1}{2}
\got{1}{B} \gnl
\gcu{1} \gnl
\gend$ and $\gbeg{1}{2}
\gu{1} \gnl
\gob{1}{F} \gnl
\gend$ are from the point 1. The above four 2-cells we will call {\em pre-multiplication} and {\em pre-unit} on $F$ and {\em pre-comultiplication} and {\em pre-counit} on $B$. 
We set the following notations: 
\vspace{-0,6cm}
\begin{center} 
\begin{tabular}{p{6cm}p{0cm}p{6cm}}
\begin{equation} \eqlabel{F left B-mod}
\gbeg{2}{3}
\got{1}{B} \got{1}{F} \gnl
\glm \gnl
\gob{3}{F}
\gend:=
\gbeg{2}{4}
\got{1}{B} \got{1}{F} \gnl
\glmptb \gnot{\hspace{-0,34cm}\psi} \grmptb \gnl
\gcl{1} \gcu{1} \gnl
\gob{1}{F}
\gend
\end{equation} & &
\begin{equation} \eqlabel{B right F-comod}
\gbeg{2}{3}
\got{1}{B} \gnl
\grcm \gnl
\gob{1}{B} \gob{1}{F}
\gend:=
\gbeg{2}{4}
\got{3}{B} \gnl
\gu{1} \gcl{1} \gnl
\glmptb \gnot{\hspace{-0,34cm}\phi'} \grmptb \gnl
\gob{1}{B} \gob{1}{F}
\gend
\end{equation}
\end{tabular}
\end{center} \vspace{-0,6cm}

 \vspace{-0,6cm}
\begin{center} 
\begin{tabular}{p{6cm}p{0cm}p{6cm}}  
\begin{equation} \eqlabel{right F-mod}
\gbeg{2}{3}
\got{1}{B} \got{1}{F} \gnl
\grmo \gcl{1} \gnl 
\gob{1}{B}
\gend:=
\gbeg{2}{4}
\got{1}{B} \got{1}{F} \gnl
\glmptb \gnot{\hspace{-0,34cm}\psi} \grmptb \gnl
\gcu{1} \gcl{1} \gnl
\gob{3}{B} 
\gend
\end{equation} & &
\begin{equation} \eqlabel{left B-comod}
\gbeg{2}{3}
\got{3}{F} \gnl
\grmo \gvac{1} \gcl{1} \gnl
\gob{1}{B} \gob{1}{F}
\gend:=
\gbeg{2}{4}
\got{1}{F} \gnl
\gcl{1} \gu{1} \gnl
\glmptb \gnot{\hspace{-0,34cm}\phi'} \grmptb \gnl
\gob{1}{B} \gob{1}{F}
\gend
\end{equation} 
\end{tabular}
\end{center}

\begin{center} \hspace{1,2cm}
\begin{tabular}{p{4cm}p{1cm}p{4cm}}
\gbeg{2}{4}
\got{1}{F} \got{1}{F} \gnl
\glmpt \gnot{\hspace{-0,34cm}\sigma} \grmptb \gnl
\gvac{1} \gcl{1} \gnl
\gob{3}{B}
\gend:=
\gbeg{2}{4}
\got{1}{F} \got{1}{F} \gnl
\glmptb \gnot{\hspace{-0,34cm}\mu_M} \grmptb \gnl
\gcu{1} \gcl{1} \gnl
\gob{3}{B}
\gend
& &
\gbeg{2}{4}
\got{3}{F} \gnl
\gvac{1} \gcl{1} \gnl
\glmpb \gnot{\hspace{-0,34cm}\rho'} \grmptb \gnl
\gob{1}{B}\gob{1}{B}
\gend:=
\gbeg{2}{4}
\got{3}{F} \gnl
\gu{1} \gcl{1} \gnl
\glmptb \gnot{\hspace{-0,34cm}\Delta_C'} \grmptb \gnl
\gob{1}{B}\gob{1}{B.}
\gend
\end{tabular}
\end{center}

\item The following compatibilities between the (co)units and pre-(co)multiplications hold: 
\begin{equation} \eqlabel{(co)unital mixed}
\gbeg{1}{3}
 \gu{1} \gnl
\hspace{-0,34cm} \gcmu \gnl
\gob{1}{B} \gob{1}{B} \gnl
\gend=
\gbeg{2}{3}
\gu{1} \gu{1} \gnl
\gcl{1} \gcl{1} \gnl
\gob{1}{B} \gob{1}{B} \gnl
\gend  \hspace{2,5cm}
\gbeg{2}{3}
\got{1}{F} \got{1}{F} \gnl
\gmu \gnl
\gvac{1} \hspace{-0,22cm} \gcu{1} \gnl
\gend=
\gbeg{2}{3}
\got{1}{F} \got{1}{F} \gnl
\gcl{1} \gcl{1} \gnl
\gcu{1} \gcu{1} \gnl
\gend 
\end{equation}
\item The distributive laws $\psi: BF\to FB$ and $\phi': FB\to BF$ from the point 2. should additionally fulfill:
\begin{equation} \eqlabel{psi unital add}
\gbeg{2}{3}
\got{1}{B} \got{1}{F} \gnl
\glmptb \gnot{\hspace{-0,34cm}\psi} \grmptb \gnl
\gcu{1} \gcu{1} \gnl
\gend=
\gbeg{2}{3}
\got{1}{B} \got{1}{F} \gnl
\gcl{1} \gcl{1} \gnl
\gcu{1} \gcu{1} \gnl
\gend
\hspace{3cm}
\gbeg{2}{3}
\gu{1} \gu{1} \gnl
\glmptb \gnot{\hspace{-0,34cm}\phi'} \grmptb \gnl
\gob{1}{B} \gob{1}{F}
\gend=
\gbeg{2}{3}
\gu{1} \gu{1} \gnl
\gcl{1} \gcl{1} \gnl
\gob{1}{B} \gob{1}{F}
\gend
\end{equation}
and
\begin{equation} \eqlabel{psi and phi concrete}
\psi=
\gbeg{3}{5}
\got{2}{B} \got{2}{F} \gnl
\gcmu \gcmu \gnl
\gcl{1} \glmptb \gnot{\hspace{-0,34cm}\tau_{B,F}} \grmptb \gcl{1} \gnl
\glm \grmo \gcl{1} \gnl
\gob{1}{} \gob{1}{F} \gob{1}{B} 
\gend 
\hspace{3cm}
\phi'= 
\gbeg{3}{5}
\got{1}{} \got{1}{F} \got{1}{B} \gnl
\grmo \gvac{1} \gcl{1} \grcm \gnl
\gcl{1} \glmptb \gnot{\hspace{-0,34cm}\tau_{F,B}} \grmptb \gcl{1} \gnl
\gmu \gmu \gnl
\gob{2}{B} \gob{2}{F} 
\gend
\end{equation}
where $\tau_{B,F}, \tau_{F,F}, \tau_{F,F}, \tau_{F,B}$ are left and right monadic and comonadic distributive laws, where the adjectives ``comonadic'' corresponding to $B$ 
and ``monadic'' corresponding to $F$ are meant with respect to the 2-cells from the point 2.
\item 2-cells $\lambda_B:BB\to BB, \lambda_F:FF\to FF$ in $\K$ 
so that the following compatibility conditions are fulfilled: 
\vspace{-0,7cm} 
\begin{center} \hspace{-1,3cm} 
\begin{tabular}{p{7cm}p{0cm}p{7cm}}
\begin{equation} \eqlabel{proj B}
\gbeg{3}{5}
\got{2}{B} \got{1}{B} \gnl
\gcmu \gcl{1} \gnl
\gcl{1} \glmptb \gnot{\hspace{-0,34cm}\lambda_B} \grmptb \gnl
\gcl{1} \gcl{1} \gcl{1} \gnl
\gob{1}{B} \gob{1}{B}  \gob{1}{B}
\gend
=
\gbeg{5}{9}
\gvac{1} \got{1}{B} \got{5}{B} \gnl
\gcn{1}{1}{3}{3} \gvac{2} \gwcm{3} \gnl
\gwcm{3} \grcm \gcl{4} \gnl
\grcm \glmptb \gnot{\hspace{-0,34cm}\tau_{B,B}} \grmptb \gcl{1} \gnl
\gcl{1} \glmptb \gnot{\hspace{-0,34cm}\tau_{F,B}} \grmptb  \glmptb \gnot{\hspace{-0,34cm}\psi} \grmptb \gnl %
\gcl{3} \gcl{3}  \glmpt \gnot{\hspace{-0,34cm}\sigma} \grmptb \gcl{1} \gnl
\gvac{3} \gmu \gcn{1}{1}{1}{0} \gnl
\gcl{1} \gvac{3} \hspace{-0.32cm} \gmu \gnl
\gvac{1} \hspace{-0.24cm} \gob{1}{B} \gob{1}{B} \gvac{2} \gob{1}{B} 
\gend
\end{equation} & & 
\begin{equation} \eqlabel{proj F}
\gbeg{3}{5}
\got{1}{F} \got{1}{F} \got{1}{F} \gnl
\gcl{1} \gcl{1} \gcl{1} \gnl
\glmptb \gnot{\hspace{-0,34cm}\lambda_F} \grmptb  \gcl{1} \gnl
\gcl{1} \gmu \gnl
\gob{1}{F}  \gob{2}{F}
\gend
=
\gbeg{6}{9}
\gvac{1} \got{2}{F} \gvac{1} \got{2}{F} \got{1}{\hspace{-0,34cm}F} \gnl
\gvac{1} \gcmu \gvac{1} \gcn{1}{3}{2}{2} \gcn{1}{3}{2}{2} \gnl
\gvac{1} \hspace{-0.32cm} \gcn{1}{1}{2}{1} \gcmu \gnl
\gvac{1} \gcl{4} \gcl{1} \glmptb \gnot{\hspace{-0,34cm}\ro'} \grmpb \gnl
\gvac{2} \glmptb \gnot{\hspace{-0,34cm}\phi'} \grmptb \glmptb \gnot{\hspace{-0,34cm}\tau_{B,F}} \grmptb \gcl{1}  \gnl %
\gvac{2} \gcl{1} \glmptb \gnot{\hspace{-0,34cm}\tau_{F,F}} \grmptb \glm \gnl
\gvac{2} \glm \gwmu{3} \gnl
\gvac{1} \gwmu{3} \gcn{1}{1}{3}{3} \gnl
\gvac{2} \gob{1}{F} \gvac{2} \gob{1}{F} 
\gend 
\end{equation} 
\end{tabular}
\end{center} \vspace{-0,7cm} 
and
\vspace{-0,7cm} 
\begin{center} \hspace{-1,3cm} 
\begin{tabular}{p{6cm}p{0cm}p{6cm}}
\begin{equation} \eqlabel{mu_M paired}
\mu_M=
\gbeg{3}{5}
\got{1}{F} \got{2}{F} \gnl
\gcl{1} \gcmu \gnl
\glmptb \gnot{\hspace{-0,34cm}\lambda_F} \grmptb \gcl{1} \gnl
\gcl{1} \glmpt \gnot{\hspace{-0,34cm}\sigma} \grmptb \gnl
\gob{1}{F} \gob{3}{B} \gnl
\gend
\end{equation} & & 
\begin{equation} \eqlabel{Delta_C' paired}
\Delta_C'=
\gbeg{3}{5}
\got{1}{F} \got{3}{B} \gnl
\glmpb \gnot{\hspace{-0,34cm}\rho'} \grmptb \gcl{1} \gnl
\gcl{1} \glmptb \gnot{\hspace{-0,34cm}\lambda_B} \grmptb \gnl
\gmu \gcl{1} \gnl
\gob{2}{B} \gob{1}{B} \gnl
\gend
\end{equation}
\end{tabular}
\end{center} \vspace{-0,7cm} 
\item $FB$ is a $\tau$-bimonad where $FB$ has the monad and comonad structure from the wreath product and the cowreath coproduct (see point 2.), that is: 
\begin{equation}  \eqlabel{wreath (co)product}
\nabla_{FB}=
\gbeg{3}{5}
\got{1}{F} \got{1}{B} \got{1}{F} \got{3}{B}  \gnl
\gcl{1}  \glmptb \gnot{\hspace{-0,34cm}\psi} \grmptb \gvac{1} \gcl{1} \gnl
\glmptb \gnot{\hspace{-0,34cm}\mu_M} \grmptb \gwmu{3} \gnl
\gcl{1} \gwmu{3} \gnl
\gob{1}{F} \gvac{1} \gob{1}{B}
\gend\hspace{1,5cm}
\eta_{FB}=
\gbeg{2}{3}
\glmpb \gnot{\hspace{-0,34cm}\eta_M} \grmpb \gnl
\gcl{1} \gcl{1} \gnl
\gob{1}{F} \gob{1}{B}
\gend; \hspace{1,5cm}
\Delta_{FB}=
\gbeg{3}{5}
\gvac{2} \got{1}{F} \got{3}{B} \gnl
\gvac{1} \gwcm{3} \gcl{1} \gnl
\gwcm{3} \glmptb \gnot{\hspace{-0,34cm}\Delta_C'} \grmptb \gnl
\gcl{1} \gvac{1} \glmptb \gnot{\hspace{-0,34cm}\phi'} \grmptb \gcl{1} \gnl
\gob{1}{F} \gvac{1} \gob{1}{B} \got{1}{F} \got{1}{B}
\gend\hspace{1,5cm}
\Epsilon_{FB}=
\gbeg{2}{4}
\got{1}{F} \got{1}{B} \gnl
\gcl{1} \gcl{1} \gnl
\glmpt \gnot{\hspace{-0,34cm}\Epsilon_C'} \grmpt \gnl
\gend
\end{equation}
and $\tau_{FB,FB}$ is given by 
$$\tau_{FB,FB}=
\gbeg{3}{5}
\got{1}{F} \got{1}{B} \got{1}{F} \got{1}{B} \gnl
\gcl{1} \glmptb \gnot{\hspace{-0,34cm}\tau_{B,F}} \grmptb  \gcl{1} \gnl
\glmptb \gnot{\hspace{-0,34cm}\tau_{F,F}} \grmptb  \glmptb \gnot{\hspace{-0,34cm}\tau_{B,B}} \grmptb  \gnl
\gcl{1} \glmptb \gnot{\hspace{-0,34cm}\tau_{F,B}} \grmptb  \gcl{1} \gnl
\gob{1}{F} \gob{1}{B} \gob{1}{F} \gob{1}{B.}
\gend
$$
\end{enumerate} 
\end{defn}


\begin{cor} \colabel{(co)actions}
Let $(B,F)$ be a paired wreath in $\K$. Then the 2-cells \equref{F left B-mod} and \equref{B right F-comod}
make $F$ a left $B$-module and $B$ a right $F$-comodule. Moreover, it holds: 
\begin{equation} \eqlabel{(co)mod (co)alg (co)unital}
\gbeg{2}{3}
\got{1}{B} \got{1}{F} \gnl
\glm \gnl
\gvac{1} \gcu{1}
\gend=
\gbeg{2}{3}
\got{1}{B} \got{1}{F} \gnl
\gcl{1} \gcl{1} \gnl
\gcu{1} \gcu{1} \gnl
\gend
\hspace{2,5cm}
\gbeg{2}{3}
\gu{1} \gnl
\grcm \gnl
\gob{1}{B} \gob{1}{F}
\gend=
\gbeg{2}{3}
\gu{1} \gu{1} \gnl
\gcl{1} \gcl{1} \gnl
\gob{1}{B} \gob{1}{F}
\gend
\end{equation}
and 
\begin{equation} \eqlabel{right F-mod counits}
\gbeg{2}{3}
\got{1}{B} \got{1}{F} \gnl
\grmo \gcl{1} \gnl 
\gcu{1}
\gend=
\gbeg{2}{3}
\got{1}{B} \got{1}{F} \gnl
\gcl{1} \gcl{1} \gnl
\gcu{1} \gcu{1} \gnl
\gend
\hspace{2,5cm}
\gbeg{2}{3}
\gvac{1} \gu{1} \gnl
\grmo \gvac{1} \gcl{1} \gnl
\gob{1}{B} \gob{1}{F}
\gend=
\gbeg{2}{3}
\gu{1} \gu{1} \gnl
\gcl{1} \gcl{1} \gnl
\gob{1}{B} \gob{1}{F}
\gend
\end{equation} 
\end{cor}

\begin{proof}
The first claim follows by the points 1. and 2. in \deref{paired bl object} and by the part a) and the right hand-side version of the part b) of \prref{distr-actions}. 
(The axioms for $\psi$ and $\phi'$ are listed below in \equref{psi laws for B} and \equref{phi' laws for B}, respectively.) The second assertion follows then by \equref{psi unital add}, 
and the last one follows from \equref{right F-mod} and \equref{left B-comod}, by \equref{psi unital add}. 
\qed\end{proof}

We now list the axioms from the point 2. in \deref{paired bl object}, saying that $(F, \psi: BF\to FB, \mu_M, \eta_M)$ is a left wreath around $B$ and that 
$(B, \phi': FB\to BF, \Delta_C', \Epsilon_C')$ is a right cowreath around $F$. We have the 2-cells in $\K$:
$$\psi: BF\to FB, \quad \phi': FB\to BF$$
$$\mu_M:FF\to FB, \quad \eta_M:\Id_{\A}\to FB, \quad \Delta_C': FB\to BB, \quad \Epsilon_C': FB\to\Id_{\A}$$
which obey the $\psi$ and $\phi'$ axioms:
\begin{center} \hspace{-0,8cm}
\begin{tabular}{p{7.4cm}p{0cm}p{8cm}}
\begin{equation}\eqlabel{psi laws for B}
\gbeg{3}{5}
\got{1}{B}\got{1}{B}\got{1}{F}\gnl
\gcl{1} \glmpt \gnot{\hspace{-0,34cm}\psi} \grmptb \gnl
\glmptb \gnot{\hspace{-0,34cm}\psi} \grmptb \gcl{1} \gnl
\gcl{1} \gmu \gnl
\gob{1}{F} \gob{2}{B}
\gend=
\gbeg{3}{5}
\got{1}{B}\got{1}{B}\got{1}{F}\gnl
\gmu \gcn{1}{1}{1}{0} \gnl
\gvac{1} \hspace{-0,34cm} \glmptb \gnot{\hspace{-0,34cm}\psi} \grmptb  \gnl
\gvac{1} \gcl{1} \gcl{1} \gnl
\gvac{1} \gob{1}{F} \gob{1}{B}
\gend;
\quad
\gbeg{2}{5}
\got{3}{F} \gnl
\gu{1} \gcl{1} \gnl
\glmptb \gnot{\hspace{-0,34cm}\psi} \grmptb \gnl
\gcl{1} \gcl{1} \gnl
\gob{1}{F} \gob{1}{B}
\gend=
\gbeg{2}{5}
\got{1}{F} \gnl
\gcl{1} \gu{1} \gnl
\gcl{2} \gcl{2} \gnl
\gob{1}{F} \gob{1}{B}
\gend
\end{equation} & &
\begin{equation}\eqlabel{phi' laws for B}
\gbeg{3}{5}
\got{2}{F} \got{1}{B}\gnl
\gcmu \gcl{1} \gnl
\gcl{1} \glmptb \gnot{\hspace{-0,34cm}\phi'} \grmptb \gnl
\glmptb \gnot{\hspace{-0,34cm}\phi'} \grmptb \gcl{1} \gnl
\gob{1}{B} \gob{1}{F} \gob{1}{F}
\gend=
\gbeg{3}{5}
\got{2}{F} \got{1}{\hspace{-0,2cm}B}\gnl
\gcn{1}{1}{2}{2} \gcn{1}{1}{2}{2} \gnl
\gvac{1} \hspace{-0,34cm} \glmptb \gnot{\hspace{-0,34cm}\phi'} \grmptb \gnl
\gvac{1} \gcn{1}{1}{1}{0} \hspace{-0,2cm} \gcmu \gnl
\gvac{1} \gob{1}{B} \gob{1}{F} \gob{1}{F}
\gend;
\quad
\gbeg{2}{5}
\got{1}{F} \got{1}{B} \gnl
\gcl{1} \gcl{1} \gnl
\glmptb \gnot{\hspace{-0,34cm}\phi'} \grmptb \gnl
\gcl{1} \gcu{1} \gnl
\gob{1}{B}
\gend=
\gbeg{2}{5}
\got{1}{F} \got{1}{B} \gnl
\gcl{1} \gcl{3} \gnl
\gcu{1} \gnl
\gob{3}{B}
\gend
\end{equation}
\end{tabular}
\end{center} 
the 2-cell conditions: 
\begin{center} \hspace{-1,4cm}
\begin{tabular}{p{6.2cm}p{1cm}p{7cm}}
\begin{equation} \eqlabel{2-cell mu_M}
\gbeg{3}{6}
\got{1}{B} \got{1}{F} \got{1}{F} \gnl
\glmptb \gnot{\hspace{-0,34cm}\psi} \grmptb \gcl{1} \gnl
\gcl{1} \glmptb \gnot{\hspace{-0,34cm}\psi} \grmptb \gnl
\glmptb \gnot{\hspace{-0,34cm}\mu_M} \grmptb \gcl{1} \gnl
\gcl{1} \gmu \gnl
\gob{1}{F} \gob{2}{B}
\gend=
\gbeg{3}{5}
\got{1}{B} \got{1}{F} \got{1}{F}\gnl
\gcl{1} \glmptb \gnot{\hspace{-0,34cm}\mu_M} \grmptb \gnl
\glmptb \gnot{\hspace{-0,34cm}\psi} \grmptb \gcl{1} \gnl
\gcl{1} \gmu \gnl
\gob{1}{F} \gob{2}{B}
\gend
\end{equation} &   &
\begin{equation}\eqlabel{2-cell eta_M}
\gbeg{3}{4}
\got{1}{} \got{3}{B}\gnl
\glmpb \gnot{\hspace{-0,34cm}\eta_M} \grmpb \gcl{1} \gnl
\gcl{1} \gmu \gnl
\gob{1}{F} \gob{2}{B}
\gend=
\gbeg{3}{5}
\got{1}{B}\gnl
\gcl{1} \glmpb \gnot{\hspace{-0,34cm}\eta_M} \grmpb \gnl
\glmptb \gnot{\hspace{-0,34cm}\psi} \grmptb \gcl{1} \gnl
\gcl{1} \gmu \gnl
\gob{1}{F} \gob{2}{B}
\gend
\end{equation} 
\end{tabular}
\end{center}

\begin{center} \hspace{-2,6cm}
\begin{tabular}{p{7.2cm}p{2cm}p{5cm}}
\begin{equation}\eqlabel{2-cell Delta_C'}
\gbeg{3}{5}
\got{2}{F} \got{1}{B} \gnl
\gcmu \gcl{1} \gnl
\gcl{1} \glmptb \gnot{\hspace{-0,34cm}\phi'} \grmptb \gnl
\glmptb \gnot{\hspace{-0,34cm}\Delta_C'} \grmptb \gcl{1} \gnl
\gob{1}{B} \gob{1}{B} \gob{1}{F}
\gend=
\gbeg{3}{6}
\got{2}{F} \got{1}{B} \gnl
\gcmu \gcl{1} \gnl
\gcl{1} \glmptb \gnot{\hspace{-0,34cm}\Delta_C'} \grmptb \gnl
\glmptb \gnot{\hspace{-0,34cm}\phi'} \grmptb \gcl{1} \gnl
\gcl{1} \glmptb \gnot{\hspace{-0,34cm}\phi'} \grmptb \gnl
\gob{1}{B} \gob{1}{B} \gob{1}{F}
\gend
\end{equation} & &
\begin{equation}\eqlabel{2-cell Epsilon_C'}
\gbeg{3}{5}
\got{2}{F} \got{1}{B}\gnl
\gcmu \gcl{1} \gnl
\gcl{1} \glmpb \gnot{\hspace{-0,34cm}\phi'} \grmpb \gnl
\glmpt \gnot{\hspace{-0,34cm}\Epsilon_C'} \grmpt \gcl{1} \gnl
\gob{5}{F}
\gend=
\gbeg{3}{5}
\got{2}{F} \got{1}{B}\gnl
\gcmu \gcl{1} \gnl
\gcl{2} \glmpt \gnot{\hspace{-0,34cm}\Epsilon_C'} \grmpt \gnl
\gob{1}{F}
\gend
\end{equation}
\end{tabular}
\end{center}
and the monad and comonad compatibility laws in the corresponding Eilenberg-Moore categories, which translate into the following conditions in $\K$: 
\vspace{-0,4cm}
\begin{center} \hspace{0,4cm}
\begin{tabular}{p{5.2cm}p{2cm}p{6.8cm}}
\begin{equation} \eqlabel{monad law mu_M}
\gbeg{3}{5}
\got{1}{F} \got{1}{F} \got{1}{F}\gnl
\gcl{1} \glmptb \gnot{\hspace{-0,34cm}\mu_M} \grmptb \gnl
\glmptb \gnot{\hspace{-0,34cm}\mu_M} \grmptb \gcl{1} \gnl
\gcl{1} \gmu \gnl
\gob{1}{F} \gob{2}{B}
\gend=
\gbeg{3}{6}
\got{1}{F} \got{1}{F} \got{1}{F} \gnl
\glmptb \gnot{\hspace{-0,34cm}\mu_M} \grmptb \gcl{1} \gnl
\gcl{1} \glmptb \gnot{\hspace{-0,34cm}\psi} \grmptb \gnl
\glmptb \gnot{\hspace{-0,34cm}\mu_M} \grmptb \gcl{1} \gnl
\gcl{1} \gmu \gnl
\gob{1}{F} \gob{2}{B}
\gend
\end{equation} &  & \vspace{0,1cm}
\begin{equation}\eqlabel{monad law eta_M}
\gbeg{3}{6}
\got{1}{} \got{3}{F}\gnl
\glmpb \gnot{\hspace{-0,34cm}\eta_M} \grmpb \gcl{1} \gnl
\gcl{1} \glmptb \gnot{\hspace{-0,34cm}\psi} \grmptb \gnl
\glmptb \gnot{\hspace{-0,34cm}\mu_M} \grmptb \gcl{1} \gnl
\gcl{1} \gmu \gnl
\gob{1}{F} \gob{2}{B}
\gend=
\gbeg{2}{4}
\got{1}{F}\gnl
\gcl{1} \gu{1} \gnl
\gcl{1} \gcl{1} \gnl
\gob{1}{F} \gob{1}{B}
\gend=
\gbeg{3}{5}
\got{1}{F}\gnl
\gcl{1} \glmpb \gnot{\hspace{-0,34cm}\eta_M} \grmpb \gnl
\glmptb \gnot{\hspace{-0,34cm}\mu_M} \grmptb \gcl{1} \gnl
\gcl{1} \gmu \gnl
\gob{1}{F} \gob{2}{B}
\gend
\end{equation} 
\end{tabular}
\end{center} \vspace{-0,4cm}
$$ \textnormal{ \hspace{-1cm} \footnotesize monad law for $\mu_M$}  \hspace{6,5cm}  \textnormal{\footnotesize monad law for $\eta_M$} $$ \vspace{-0,7cm}


\begin{center} 
\begin{tabular}{p{6cm}p{2cm}p{6.8cm}}
\begin{equation} \eqlabel{comonad law Delta_C'}
\gbeg{3}{6}
\got{2}{F} \got{1}{B}\gnl
\gcmu \gcl{1} \gnl
\gcl{1} \glmptb \gnot{\hspace{-0,34cm}\Delta_C'} \grmptb \gnl
\glmptb \gnot{\hspace{-0,34cm}\phi'} \grmptb \gcl{1} \gnl
\gcl{1} \glmptb \gnot{\hspace{-0,34cm}\Delta_C'} \grmptb \gnl
\gob{1}{B} \gob{1}{B}  \gob{1}{B}
\gend=
\gbeg{3}{5}
\got{2}{F} \got{1}{B}\gnl
\gcmu \gcl{1} \gnl
\gcl{1} \glmptb \gnot{\hspace{-0,34cm}\Delta_C'} \grmptb \gnl
\glmptb \gnot{\hspace{-0,34cm}\Delta_C'} \grmptb \gcl{1} \gnl
\gob{1}{B} \gob{1}{B}  \gob{1}{B}
\gend
\end{equation} &  &
\begin{equation}\eqlabel{comonad law Epsilon_C'}
\gbeg{3}{6}
\got{2}{F} \got{1}{B}\gnl
\gcmu \gcl{1} \gnl
\gcl{1} \glmptb \gnot{\hspace{-0,34cm}\Delta_C'} \grmptb \gnl
\glmptb \gnot{\hspace{-0,34cm}\phi'} \grmptb \gcl{1} \gnl
\gcl{1} \glmpt \gnot{\hspace{-0,34cm}\Epsilon_C'} \grmpt \gnl
\gob{1}{B}
\gend=
\gbeg{2}{4}
\got{1}{F} \got{1}{B}\gnl
\gcl{1} \gcl{1} \gnl
\gcu{1} \gcl{1} \gnl
\gob{3}{B}
\gend=
\gbeg{3}{5}
\got{2}{F} \got{1}{B}\gnl
\gcmu \gcl{1} \gnl
\gcl{1} \glmptb \gnot{\hspace{-0,34cm}\Delta_C'} \grmptb \gnl
\glmpt \gnot{\hspace{-0,34cm}\Epsilon_C'} \grmpt \gcl{1} \gnl
\gob{5}{B} \gnl
\gend
\end{equation}
\end{tabular}
\end{center} \vspace{-0,5cm}
$$ \textnormal{\hspace{-0,7cm}  \footnotesize comonad law for $\Delta_C'$}  \hspace{6,5cm}  \textnormal{\footnotesize comonad law for $\Epsilon_C'$} $$

\bigskip

\begin{rem}
Observe that we have the following: 
\begin{equation} \eqlabel{trivial psi-phi}
\gbeg{2}{4}
\got{1}{F} \got{1}{B} \gnl
\glmptb \gnot{\hspace{-0,34cm}\phi'} \grmptb \gnl
\gcu{1} \gcl{1} \gnl
\gob{3}{F} 
\gend\stackrel{\equref{2-cell Epsilon_C'}}{=}
\gbeg{2}{4}
\got{1}{F} \got{1}{B} \gnl
\gcl{2} \gcu{1} \gnl
\gob{1}{F} 
\gend,
\hspace{1,2cm} 
\gbeg{2}{4}
\got{3}{F} \gnl
\gu{1} \gcl{1} \gnl
\glmptb \gnot{\hspace{-0,34cm}\psi} \grmptb \gnl
\gob{1}{F} \gob{1}{B} 
\gend\stackrel{\equref{psi laws for B}}{=}
\gbeg{2}{4}
\got{1}{F} \gnl
\gcl{1} \gu{1} \gnl
\gcl{1} \gcl{1} \gnl
\gob{1}{F} \gob{1}{B} 
\gend,
\hspace{1,2cm} 
\gbeg{2}{4}
\got{1}{B} \gnl
\gcl{1} \gu{1} \gnl
\glmptb \gnot{\hspace{-0,34cm}\psi} \grmptb \gnl
\gob{1}{F} \gob{1}{B} 
\gend\stackrel{\equref{2-cell eta_M}}{=}
\gbeg{2}{4}
\got{3}{B} \gnl
\gu{1} \gcl{2} \gnl
\gcl{1} \gnl
\gob{1}{F} \gob{1}{B}
\gend,
\hspace{1,2cm} 
\gbeg{2}{4}
\got{1}{F} \got{1}{B} \gnl
\glmptb \gnot{\hspace{-0,34cm}\phi'} \grmptb \gnl
\gcl{1} \gcu{1} \gnl
\gob{1}{B}  
\gend\stackrel{\equref{phi' laws for B}}{=}
\gbeg{2}{4}
\got{1}{F} \got{1}{B} \gnl
\gcu{1} \gcl{2} \gnl
\gob{3}{B}  
\gend
\end{equation} 
where in the first and the third identity we assumed that the 2-cells $\eta_M:\Id_{\A}\to FB $ and $\Epsilon_C': FB\to\Id_{\A}$ are canonical. 
Then if we set: 
$$
\gbeg{2}{3}
\got{1}{F} \got{1}{B} \gnl
\grmf \gcn{1}{1}{-3}{-3} \gnl 
\gob{1}{F}
\gend:=
\gbeg{2}{4}
\got{1}{F} \got{1}{B} \gnl
\glmptb \gnot{\hspace{-0,34cm}\phi'} \grmptb \gnl
\gcu{1} \gcl{1} \gnl
\gob{3}{F} 
\gend
\hspace{1,5cm} 
\gbeg{2}{3}
\got{1}{F} \gnl
\gcl{1} \hspace{-0,42cm} \glmf \gnl  
\gvac{1} \gob{1}{F} \gob{1}{B}
\gend:=
\gbeg{2}{4}
\got{3}{F} \gnl
\gu{1} \gcl{1} \gnl
\glmptb \gnot{\hspace{-0,34cm}\psi} \grmptb \gnl
\gob{1}{F} \gob{1}{B} 
\gend
\hspace{1,5cm} 
\gbeg{2}{3}
\got{3}{B} \gnl
\grmf \gcn{1}{1}{-1}{-1} \gnl
\gob{1}{F} \gob{1}{B}
\gend:=
\gbeg{2}{4}
\got{1}{B} \gnl
\gcl{1} \gu{1} \gnl
\glmptb \gnot{\hspace{-0,34cm}\psi} \grmptb \gnl
\gob{1}{F} \gob{1}{B} 
\gend
\hspace{1,5cm} 
\gbeg{2}{3}
\got{1}{F} \got{1}{B} \gnl
\glmf \gcn{1}{1}{-1}{-1} \gnl
\gob{3}{B}
\gend:=
\gbeg{2}{4}
\got{1}{F} \got{1}{B} \gnl
\glmptb \gnot{\hspace{-0,34cm}\phi'} \grmptb \gnl
\gcl{1} \gcu{1} \gnl
\gob{1}{B}  
\gend
$$
all these (co)module actions turn out to be trivial. 
\end{rem}

\subsection{Structures inside a paired wreath}

Similarly as we did in \cite{Femic5}, we will apply 
$\gbeg{2}{2}
\got{1}{B}\gnl
\gcu{1} \gnl
\gend$ to the monadic axioms and 
$\gbeg{2}{2}
\gu{1} \gnl
\gob{1}{F} \gnl
\gend$ to the comonadic axioms in the above list and then use the pre-counit-multiplication and the pre-unit-comultiplication compatibility relations \equref{(co)unital}, respectively. 
In this way we obtain the following identities that hold in a paired wreath: 
\begin{center} \hspace{-1,4cm}
\begin{tabular} {p{6.4cm}p{2cm}p{4cm}} 
\begin{equation} \eqlabel{mod alg}
\gbeg{3}{5}
\got{1}{B} \got{1}{F} \got{1}{F} \gnl
\glmptb \gnot{\hspace{-0,34cm}\psi} \grmptb \gcl{1} \gnl
\gcl{1} \glm \gnl
\gwmu{3} \gnl
\gob{3}{F}
\gend=
\gbeg{3}{5}
\got{1}{B} \got{1}{F} \got{3}{F}\gnl
\gcl{1} \gwmu{3} \gnl
\gcn{1}{1}{1}{3} \gvac{1} \gcl{2} \gnl
\gvac{1} \glm \gnl
\gob{5}{F}
\gend
\end{equation} & &  \vspace{-0,9cm}
\begin{equation} \eqlabel{mod alg unity}
\gbeg{3}{5}
\got{1}{B} \gnl
\gcl{1} \gu{1} \gnl
\glm \gnl
\gvac{1} \gcl{1} \gnl
\gob{3}{F}
\gend=
\gbeg{2}{6}
\got{1}{B} \gnl
\gcl{1} \gnl
\gcu{1} \gnl
\gu{1} \gnl
\gcl{1} \gnl
\gob{1}{F}
\gend
\end{equation}
\end{tabular}
\end{center} \vspace{-0,2cm}
$$ \textnormal{ \footnotesize module monad}  \hspace{4,5cm}  \textnormal{\footnotesize module monad unity} $$ \vspace{-0,7cm}
$$ \textnormal{\footnotesize from 2-cell cond. of $\mu_M$ \equref{2-cell mu_M}} \hspace{4cm} \textnormal{\footnotesize from 2-cell cond. of $\eta_M$ \equref{2-cell eta_M}} $$

\begin{center} \hspace{-1,4cm}
\begin{tabular} {p{6cm}p{2cm}p{4cm}} 
\begin{equation}\eqlabel{comod coalg}
\gbeg{3}{6}
\got{3}{B} \gnl
\gvac{1} \gcl{1} \gnl
\gwcm{3} \gnl
\grcm \gcl{1} \gnl
\gcl{1} \glmptb \gnot{\hspace{-0,34cm}\phi'} \grmptb \gnl
\gob{1}{B} \gob{1}{B} \gob{1}{F}
\gend=
\gbeg{3}{6}
\got{3}{B}\gnl
\gvac{1} \gcl{1} \gnl
\gvac{1} \grcm \gnl
\gvac{1} \gcl{1} \gcn{1}{1}{1}{3} \gnl
\gwcm{3} \gcl{1} \gnl
\gob{1}{B} \gvac{1} \gob{1}{B} \gob{1}{F}
\gend
\end{equation} & & \vspace{0,1cm}
\begin{equation}\eqlabel{comod coalg counity}
\gbeg{2}{4}
\got{1}{B} \gnl
\grcm \gnl
\gcu{1} \gcl{1} \gnl
\gob{3}{F}
\gend=
\gbeg{1}{4}
\got{1}{B} \gnl
\gcu{1} \gnl
\gu{1} \gnl
\gob{1}{F}
\gend
\end{equation}
\end{tabular}
\end{center}
\vspace{-0,2cm}
$$ \textnormal{ \footnotesize \hspace{1cm} comodule comonad}  \hspace{4,5cm}  \textnormal{\footnotesize comodule comonad counity} $$ \vspace{-0,7cm}
$$ \textnormal{\footnotesize from 2-cell cond. of $\Delta_C'$ \equref{2-cell Delta_C'}} \hspace{4cm} \textnormal{\footnotesize from 2-cell cond. of $\Epsilon_C'$ \equref{2-cell Epsilon_C'}} $$


\begin{center} \hspace{-1,4cm}
\begin{tabular} {p{7.2cm}p{0cm}p{6.8cm}}
\begin{equation}\eqlabel{weak assoc. mu_M}
\gbeg{4}{4}
\got{1}{F} \got{1}{F} \got{3}{F} \gnl
\gcl{1} \gwmu{3} \gnl
\gwmu{3} \gnl
\gob{3}{F}
\gend=
\gbeg{3}{5}
\got{1}{F} \got{1}{F} \got{1}{F} \gnl
\glmptb \gnot{\hspace{-0,34cm}\mu_M} \grmptb \gcl{1} \gnl
\gcl{1} \glm \gnl
\gwmu{3} \gnl
\gob{3}{F}
\gend
\end{equation} & &
\begin{equation}\eqlabel{weak unity eta_M}
\gbeg{3}{5}
\got{1}{} \got{3}{F}\gnl
\glmpb \gnot{\hspace{-0,34cm}\eta_M} \grmpb \gcl{1} \gnl
\gcl{1} \glm \gnl
\gwmu{3} \gnl
\gob{3}{F}
\gend=
\gbeg{1}{4}
\got{1}{F}\gnl
\gcl{2} \gnl
\gob{1}{F}
\gend=
\gbeg{2}{4}
\got{1}{F}\gnl
\gcl{1} \gu{1} \gnl
\gmu \gnl
\gob{2}{F}
\gend
\end{equation}
\end{tabular}
\end{center}
\vspace{-0,5cm}
$$ \textnormal{ \footnotesize \hspace{-1,4cm} weak associativity of $\mu_M$}  \hspace{4,5cm}  \textnormal{\footnotesize weak unity $\eta_M$} $$ \vspace{-0,7cm}
$$ \textnormal{\footnotesize from monad law for $\mu_M$ \equref{monad law mu_M}} \hspace{4cm} \textnormal{\footnotesize from monad law for $\eta_M$ \equref{monad law eta_M}} $$

\pagebreak

\begin{center} \hspace{-1,4cm}
\begin{tabular}{p{7cm}p{0cm}p{6.8cm}}
\begin{equation} \eqlabel{weak coass. Delta_C'}
\gbeg{4}{4}
\got{5}{B} \gnl
\gvac{1} \gwcm{3} \gnl
\gwcm{3} \gcl{1} \gnl
\gob{1}{B} \gvac{1} \gob{1}{B} \gob{1}{B} \gnl
\gend=
\gbeg{3}{5}
\got{3}{B} \gnl
\gwcm{3} \gnl
\grcm \gcl{1} \gnl
\gcl{1} \glmptb \gnot{\hspace{-0,34cm}\Delta_C'} \grmptb \gnl
\gob{1}{B} \gob{1}{B} \gob{1}{B} \gnl
\gend
\end{equation} &  &
\begin{equation}\eqlabel{weak counit Epsilon_C'}
\gbeg{3}{5}
\got{3}{B} \gnl
\gwcm{3} \gnl
\grcm \gcl{1} \gnl
\gcl{1} \glmpt \gnot{\hspace{-0,34cm}\Epsilon_C'} \grmpt \gnl
\gob{1}{B}
\gend=
\gbeg{1}{4}
\got{1}{B} \gnl
\gcl{2} \gnl
\gob{1}{B}
\gend=
\gbeg{2}{4}
\got{2}{B} \gnl
\gcmu \gnl
\gvac{1} \hspace{-0,42cm} \gcu{1} \gcl{1} \gnl
\gob{5}{B} \gnl
\gend
\end{equation}
\end{tabular}
\end{center}
\vspace{-0,5cm}
$$ \textnormal{ \footnotesize \hspace{-1,6cm} weak coassociativity of $\Delta_C'$}  \hspace{5cm}  \textnormal{\footnotesize weak counity $\Epsilon_C'$} $$ \vspace{-0,7cm}
$$ \textnormal{\footnotesize from comonad law for $\Delta_C'$ \equref{comonad law Delta_C'}} \hspace{4cm} \textnormal{\footnotesize from comonad law for $\Epsilon_C'$ \equref{comonad law Epsilon_C'}} $$


Next, we apply $\gbeg{2}{2}
\got{1}{F}\gnl
\gcu{1} \gnl
\gend$ to the $\psi$ axioms and the monadic axioms and $\gbeg{2}{2}
\gu{1} \gnl
\gob{1}{B} \gnl
\gend$ to the $\phi'$ axioms and the comonadic axioms in the list of identities \equref{psi laws for B} to \equref{comonad law Epsilon_C'}. We obtain the following laws 
valid in a paired wreath: 
 \begin{center} \hspace{-1,5cm} 
\begin{tabular}{p{7.2cm}p{0cm}p{6cm}}
\begin{equation}\eqlabel{F mod alg}
\gbeg{3}{5}
\got{1}{B}\got{1}{B}\got{1}{F}\gnl
\gcl{1} \glmpt \gnot{\hspace{-0,34cm}\psi} \grmptb \gnl
\grmo \gcl{1} \gvac{1} \gcl{1} \gnl
\gwmu{3} \gnl
\gob{3}{B}
\gend=
\gbeg{3}{5}
\got{1}{B}\got{1}{B}\got{1}{F}\gnl
\gmu \gcn{1}{1}{1}{0} \gnl
\gvac{1} \hspace{-0,34cm} \grmo \gcl{1} \gnl
\gvac{1} \gcl{1} \gnl
\gvac{1} \gob{1}{B}
\gend
\end{equation} & &
\begin{equation}\eqlabel{F mod alg unit}
\gbeg{2}{5}
\got{3}{F} \gnl
\gu{1} \gcl{1} \gnl
\grmo \gcl{1} \gnl
\gcl{1} \gnl
\gob{1}{B}
\gend=
\gbeg{2}{4}
\got{1}{F} \gnl
\gcu{1} \gnl
\gu{1} \gnl
\gob{1}{B}
\gend
\end{equation}

\\  {\hspace{2cm} \footnotesize module monad} & &  { \hspace{1cm} \footnotesize module monad unity} \\
\multicolumn{2}{c}{{ \hspace{5cm} \footnotesize \quad\qquad from $\psi$ axioms \equref{psi laws for B}}} 
\end{tabular}
\end{center}

\begin{center} \hspace{-1,5cm} 
\begin{tabular}{p{7.2cm}p{0cm}p{6cm}}
\begin{equation}\eqlabel{B comod coalg}
\gbeg{3}{5}
\got{3}{F}\gnl
\gwcm{3} \gnl
\gcl{1} \grmo \gvac{1} \gcl{1} \gnl
\glmptb \gnot{\hspace{-0,34cm}\phi'} \grmptb \gcl{1} \gnl
\gob{1}{B} \gob{1}{F} \gob{1}{F} 
\gend=
\gbeg{3}{4}
\got{3}{F}\gnl
\grmo \gvac{1} \gcl{1} \gnl
\gcn{1}{1}{1}{0} \hspace{-0,2cm} \gcmu \gnl
\gob{1}{B} \gob{1}{F} \gob{1}{F} 
\gend
\end{equation} & & \vspace{0,2cm}
\begin{equation}\eqlabel{B comod coalg counit}
\gbeg{2}{4}
\got{3}{F} \gnl
\grmo \gvac{1} \gcl{1}\gnl
\gcl{1} \gcu{1} \gnl
\gob{1}{B} 
\gend=
\gbeg{2}{4}
\got{1}{F} \gnl
\gcu{1} \gnl
\gu{1} \gnl
\gob{1}{B} 
\gend
\end{equation}

\\  {\hspace{2cm} \footnotesize comodule comonad} & &  { \hspace{1cm} \footnotesize comodule comonad counity} \\
\multicolumn{2}{c}{{ \hspace{5cm} \footnotesize \quad\qquad from $\phi'$ axioms \equref{phi' laws for B}}} 
\end{tabular}
\end{center} 

\begin{center} 
\begin{tabular} {p{6cm}p{1cm}p{6cm}} 
\begin{equation} \eqlabel{weak action} 
\gbeg{3}{6}
\got{1}{B} \got{1}{F} \got{1}{F} \gnl
\glmptb \gnot{\hspace{-0,34cm}\psi} \grmptb \gcl{1} \gnl
\gcl{1} \glmptb \gnot{\hspace{-0,34cm}\psi} \grmptb \gnl
\glmptb \gnot{\hspace{-0,34cm}\sigma} \grmpt \gcl{1} \gnl
\gwmu{3} \gnl
\gob{3}{B}
\gend=
\gbeg{3}{5}
\got{1}{B} \got{1}{F} \got{1}{F}\gnl
\gcl{1} \glmptb \gnot{\hspace{-0,34cm}\mu_M} \grmptb \gnl
\grmo \gcl{1} \gvac{1} \gcl{1} \gnl
\gwmu{3} \gnl
\gob{3}{B}
\gend
\end{equation} & & 
\begin{equation} \eqlabel{weak action unity} 
\gbeg{3}{4}
\got{1}{} \got{3}{B}\gnl
\glmpb \gnot{\hspace{-0,34cm}\eta_M} \grmpb \gcl{1} \gnl
\gcu{1} \gmu \gnl
\gob{4}{B}
\gend=
\gbeg{3}{5}
\got{1}{B}\gnl
\gcl{1} \glmpb \gnot{\hspace{-0,34cm}\eta_M} \grmpb \gnl
\grmo \gcl{1} \gvac{1} \gcl{1} \gnl
\gwmu{3} \gnl
\gob{3}{B}
\gend
\end{equation} 
\end{tabular}
\end{center}
\vspace{-0,6cm}
$$ \textnormal{ \footnotesize twisted action}  \hspace{5,5cm}  \textnormal{\footnotesize twisted action unity} $$ \vspace{-0,7cm}
$$ \textnormal{\footnotesize from 2-cell cond. of $\mu_M$ \equref{2-cell mu_M}} \hspace{4cm} \textnormal{\footnotesize from 2-cell cond. of $\eta_M$ \equref{2-cell eta_M}} $$

\begin{center} 
\begin{tabular} {p{6cm}p{1.5cm}p{6cm}} 
\begin{equation}\eqlabel{weak coaction}
\gbeg{3}{5}
\got{3}{F} \gnl
\gwcm{3} \gnl
\gcl{1} \grmo \gvac{1} \gcl{1}\gnl 
\glmptb \gnot{\hspace{-0,34cm}\Delta_C'} \grmptb \gcl{1} \gnl
\gob{1}{B} \gob{1}{B} \gob{1}{F} 
\gend=
\gbeg{3}{6}
\got{2}{F} \gnl
\gcmu \gnl
\gcl{1} \glmptb \gnot{\hspace{-0,34cm}\rho'} \grmpb \gnl
\glmptb \gnot{\hspace{-0,34cm}\phi'} \grmptb \gcl{1} \gnl
\gcl{1} \glmptb \gnot{\hspace{-0,34cm}\phi'} \grmptb \gnl
\gob{1}{B} \gob{1}{B} \gob{1}{F} 
\gend
\end{equation} & & \vspace{0,1cm}
\begin{equation}\eqlabel{weak coaction counity}
\gbeg{3}{5}
\got{3}{F}\gnl
\gwcm{3} \gnl
\gcl{1} \grmo \gvac{1} \gcl{1}\gnl
\glmpt \gnot{\hspace{-0,34cm}\Epsilon_C'} \grmpt \gcl{1} \gnl
\gob{5}{F}
\gend=
\gbeg{3}{5}
\got{2}{F}\gnl
\gcmu \gu{1} \gnl
\gcl{2} \glmpt \gnot{\hspace{-0,34cm}\Epsilon_C'} \grmpt \gnl
\gob{1}{F}
\gend
\end{equation} 
\end{tabular}
\end{center} \vspace{-0,5cm}
$$ \textnormal{ \footnotesize twisted coaction}  \hspace{5,5cm}  \textnormal{\footnotesize twisted coaction counity} $$ \vspace{-0,7cm}
$$ \textnormal{\footnotesize from 2-cell cond. of $\Delta_C'$ \equref{2-cell Delta_C'}} \hspace{4cm} \textnormal{\footnotesize from 2-cell cond. of $\Epsilon_C'$ \equref{2-cell Epsilon_C'}} $$

\begin{center} \hspace{4cm} 
\begin{tabular} {p{6cm}p{1cm}p{6.8cm}} 
\begin{equation}\eqlabel{2-cocycle condition}
\gbeg{3}{5}
\got{1}{F} \got{1}{F} \got{1}{F}\gnl
\gcl{1} \glmptb \gnot{\hspace{-0,34cm}\mu_M} \grmptb \gnl
\glmpt \gnot{\hspace{-0,34cm}\sigma} \grmptb \gcl{1} \gnl
\gvac{1} \gmu \gnl
\gob{4}{B}
\gend=
\gbeg{3}{6}
\got{1}{F} \got{1}{F} \got{1}{F} \gnl
\glmptb \gnot{\hspace{-0,34cm}\mu_M} \grmptb \gcl{1} \gnl
\gcl{1} \glmptb \gnot{\hspace{-0,34cm}\psi} \grmptb \gnl
\glmpt \gnot{\hspace{-0,34cm}\sigma} \grmptb \gcl{1} \gnl
\gvac{1} \gmu \gnl
\gob{4}{B}
\gend
\end{equation} & & 
\begin{equation}\eqlabel{normalized 2-cocycle}
\gbeg{3}{6}
\got{1}{} \got{3}{F}\gnl
\glmpb \gnot{\hspace{-0,34cm}\eta_M} \grmpb \gcl{1} \gnl
\gcl{1} \glmptb \gnot{\hspace{-0,34cm}\psi} \grmptb \gnl
\glmpt \gnot{\hspace{-0,34cm}\sigma} \grmptb \gcl{1} \gnl
\gvac{1} \gmu \gnl
\gob{4}{B}
\gend=
\gbeg{1}{4}
\got{1}{F}\gnl
\gcu{1} \gnl
\gu{1} \gnl
\gob{1}{B}
\gend=
\gbeg{3}{5}
\got{1}{F}\gnl
\gcl{1} \glmpb \gnot{\hspace{-0,34cm}\eta_M} \grmpb \gnl
\glmpt \gnot{\hspace{-0,34cm}\sigma} \grmptb \gcl{1} \gnl
\gvac{1} \gmu \gnl
\gob{4}{B}
\gend
\end{equation} 
\end{tabular}
\end{center} \vspace{-0,5cm}
$$ \textnormal{ \footnotesize 2-cocycle condition}  \hspace{5,5cm}  \textnormal{\footnotesize normalized 2-cocycle} $$ \vspace{-0,7cm}
$$ \textnormal{\footnotesize from monad law for $\mu_M$ \equref{monad law mu_M}} \hspace{4cm} \textnormal{\footnotesize from monad law for $\eta_M$ \equref{monad law eta_M}} $$

\begin{center} \hspace{3.5cm} 
\begin{tabular}{p{6cm}p{1.5cm}p{6.8cm}}
\begin{equation} \eqlabel{2-cycle ro'} 
\gbeg{3}{6}
\got{2}{F}\gnl
\gcmu \gvac{1} \gnl
\gcl{1} \glmptb \gnot{\hspace{-0,34cm}\rho'} \grmpb \gnl
\glmptb \gnot{\hspace{-0,34cm}\phi'} \grmptb \gcl{1} \gnl
\gcl{1} \glmptb \gnot{\hspace{-0,34cm}\Delta_C'} \grmptb \gnl
\gob{1}{B} \gob{1}{B} \gob{1}{B} 
\gend=
\gbeg{3}{5}
\got{2}{F}\gnl
\gcmu \gvac{1} \gnl
\gcl{1} \glmptb \gnot{\hspace{-0,34cm}\rho'} \grmpb \gnl
\glmptb \gnot{\hspace{-0,34cm}\Delta_C'} \grmptb \gcl{1} \gnl
\gob{1}{B} \gob{1}{B} \gob{1}{B} 
\gend
\end{equation} &  &
\begin{equation}\eqlabel{normalized 2-cycle ro'}
\gbeg{3}{6}
\got{2}{F}\gnl
\gcmu \gnl
\gcl{1} \glmptb \gnot{\hspace{-0,34cm}\rho'} \grmpb \gnl
\glmptb \gnot{\hspace{-0,34cm}\phi'} \grmptb \gcl{1} \gnl
\gcl{1} \glmpt \gnot{\hspace{-0,34cm}\Epsilon_C'} \grmpt \gnl
\gob{1}{B} 
\gend=
\gbeg{1}{4}
\got{1}{F} \gnl
\gcu{1} \gnl
\gu{1} \gnl
\gob{1}{B} 
\gend=
\gbeg{3}{5}
\got{2}{F}\gnl
\gcmu \gnl
\gcl{1} \glmptb \gnot{\hspace{-0,34cm}\rho'} \grmpb \gnl
\glmpt \gnot{\hspace{-0,34cm}\Epsilon_C'} \grmpt \gcl{1} \gnl
\gob{5}{B} \gnl
\gend
\end{equation}
\end{tabular}
\end{center} \vspace{-0,5cm}
$$ \textnormal{ \footnotesize 2-cycle condition for $\rho'$}  \hspace{5,5cm}  \textnormal{\footnotesize normalized 2-cycle $\rho'$} $$ \vspace{-0,7cm}
$$ \textnormal{\footnotesize from comonad law for $\Delta_C'$ \equref{comonad law Delta_C'}} \hspace{4cm} \textnormal{\footnotesize from comonad law for $\Epsilon_C'$ \equref{comonad law Epsilon_C'}} $$


\bigskip
 
Having seen the above properties in a paired wreath, we may combine them to get the following ones. 

\medskip

When we apply $\Epsilon_B$ to the 2-cocycle condition \equref{2-cocycle condition}, we get: 
\begin{equation} \eqlabel{epsilon to sigma} 
\gbeg{2}{5}
\got{1}{F} \got{1}{\hspace{-0,34cm}F} \got{1}{\hspace{-0,34cm}F} \gnl
\gcl{1} \hspace{-0,22cm} \gmu \gnl
\gvac{1} \hspace{-0,22cm} \glmpt \gnot{\hspace{-0,34cm}\sigma} \grmptb \gnl
\gvac{2} \gcu{1} \gnl
\gob{3}{}
\gend=
\gbeg{2}{5}
\got{1}{F} \got{1}{F} \got{1}{F} \gnl
\glmptb \gnot{\hspace{-0,34cm}\mu_M} \grmptb \gcl{1} \gnl
\gcl{1} \glm \gnl
\glmpt \gnot{\sigma} \gcmp \grmptb \gnl
\gvac{2} \gcu{1} \gnl
\gend
\qquad
\stackrel{FF
\gbeg{2}{2}
\gu{1} \gnl
\gob{1}{F}
\gend
}{\Rightarrow}
\gbeg{2}{6}
\got{1}{F} \got{1}{\hspace{-0,34cm}F}  \gnl
\gcl{2} \gcn{1}{1}{0}{0} \hspace{-0,22cm} \gu{1} \gnl
\gvac{1} \gmu \gnl
\gvac{1} \hspace{-0,22cm} \glmpt \gnot{\hspace{-0,34cm}\sigma} \grmptb \gnl
\gvac{2} \gcu{1} \gnl
\gob{3}{}
\gend=
\gbeg{2}{5}
\got{1}{F} \got{1}{F} \gnl
\glmptb \gnot{\hspace{-0,34cm}\mu_M} \grmptb \gu{1} \gnl
\gcl{1} \glm \gnl
\glmpt \gnot{\sigma} \gcmp \grmptb \gnl
\gvac{2} \gcu{1} \gnl
\gend\hspace{1cm}{\stackrel{\equref{weak unity eta_M}, \equref{mod alg unity}}{\Longleftrightarrow}}\hspace{0,7cm}
\gbeg{2}{3}
\got{1}{F} \got{1}{F} \gnl
\glmpt \gnot{\hspace{-0,34cm}\sigma} \grmptb \gnl
\gvac{1} \gcu{1} \gnl
\gend=
\gbeg{2}{7}
\got{1}{F} \got{1}{F} \gnl
\glmptb \gnot{\hspace{-0,34cm}\mu_M} \grmptb \gnl
\gcl{1} \gcu{1} \gnl
\gcl{1} \gu{1} \gnl
\glmpt \gnot{\hspace{-0,34cm}\sigma} \grmptb \gnl
\gvac{1} \gcu{1} \gnl
\gend\stackrel{*}{\stackrel{\equref{normalized 2-cocycle}}{=}}
\gbeg{2}{5}
\got{1}{F} \got{1}{F} \gnl
\gmu \gnl
\gvac{1} \hspace{-0,22cm} \gcu{1} \gnl
\gvac{1} \gu{1} \gnl
\gvac{1} \gcu{1} \gnl
\gend=
\gbeg{2}{3}
\got{1}{F} \got{1}{F} \gnl
\gcl{1} \gcl{1} \gnl
\gcu{1} \gcu{1} \gnl
\gend
\end{equation}
where at the place * the equality holds since we assume $\eta_M$ to be canonical, and in the last equality we applied the extreme right identities in \equref{(co)unital} and \equref{(co)unital mixed}. 
Dually to the above, from the 2-cycle condition \equref{2-cycle ro'} and assuming that $\Epsilon_C'$ is canonical we get: 
\begin{equation} \eqlabel{unit to ro'}
\gbeg{2}{3}
\gvac{1} \gu{1} \gnl
\glmpb \gnot{\hspace{-0,34cm}\rho'} \grmptb \gnl
\gob{1}{B} \gob{1}{B} \gnl
\gend=
\gbeg{2}{3}
\gu{1} \gu{1} \gnl
\gcl{1} \gcl{1} \gnl
\gob{1}{B} \gob{1}{B} \gnl
\gend
\end{equation}

\medskip

We say that a 2-cocycle $\sigma$ and a 2-cycle $\rho'$, respectively, is {\em trivial}, if the following corresponding identity in:
$$\sigma=
\gbeg{2}{4}
\got{1}{F} \got{1}{F} \gnl
\gcu{1} \gcu{1} \gnl
\gu{1} \gnl
\gob{1}{B}
\gend \hspace{0,1cm}, \quad
\rho'=
\gbeg{2}{4}
\got{1}{F} \gnl
\gcu{1} \gnl
\gu{1} \gu{1} \gnl
\gob{1}{B} \gob{1}{B}
\gend 
$$
holds.

\medskip

We now draw some consequences about the 2-cells lambda, from the point 5. of \deref{paired bl object}. When we apply $\Epsilon_B BB$ to \equref{proj B} and 
$FF\eta_F$ to \equref{proj F}, we get: 
\vspace{-0,7cm} 
\begin{center} \hspace{-1,3cm} 
\begin{tabular}{p{5cm}p{0cm}p{5cm}}
\begin{equation} \eqlabel{lambda B}
\lambda_B=
\gbeg{4}{7}
\got{1}{B} \got{3}{B} \gnl
\gcl{1} \gwcm{3} \gnl
\gcl{1} \grcm \gcl{3} \gnl
\glmptb \gnot{\hspace{-0,34cm}\tau_{B,B}} \grmptb \gcl{1} \gnl
\gcl{1} \grmo \gcl{1} \gnl
\gcl{1} \gwmu{3} \gnl
\gob{1}{B} \gob{3}{B}
\gend
\end{equation} & & 
\begin{equation} \eqlabel{lambda F}
\lambda_F=
\gbeg{4}{7}
\got{1}{} \got{1}{F} \got{3}{F} \gnl
\gwcm{3} \gcl{2} \gnl
\gcl{1} \grmo \gvac{1} \gcl{1} \gnl
\gcl{1} \gcl{1} \glmptb \gnot{\hspace{-0,34cm}\tau_{F,F}} \grmptb \gnl
\gcl{1} \glm \gcl{2} \gnl
\gwmu{3} \gnl
\gvac{1} \gob{1}{F} \gob{3}{F}
\gend
\end{equation} 
\end{tabular}
\end{center} \vspace{-0,7cm}
respectively. 
In the first equality we used \equref{B comod coalg counit}, \equref{weak counit Epsilon_C'}, left unital distributive law for $\tau_{F,B}$ and 
\equref{normalized 2-cocycle}. The second identity follows by $\pi$-symmetry. 
Having this, after applying (co)units of $B$ and $F$, respectively, at appropriate places, and by the same properties as above, we further obtain:
$$
\gbeg{3}{4}
\got{1}{B} \got{1}{B} \gnl
\glmptb \gnot{\hspace{-0,34cm}\lambda_B} \grmptb \gnl
\gcu{1} \gcl{1} \gnl
\gob{3}{B}
\gend=
\gbeg{2}{3}
\got{1}{B} \got{1}{B} \gnl
\gmu \gnl
\gob{2}{B}
\gend, 
\hspace{1cm}
\gbeg{2}{4}
\got{3}{B} \gnl
\gu{1} \gcl{1} \gnl
\glmptb \gnot{\hspace{-0,34cm}\lambda_B} \grmptb \gnl
\gob{1}{B} \gob{1}{B}
\gend=
\gbeg{2}{3}
\got{2}{B} \gnl
\gcmu \gnl
\gob{1}{B} \gob{1}{B}
\gend;
\hspace{1,5cm}
\gbeg{3}{4}
\got{1}{F} \got{1}{F} \gnl
\glmptb \gnot{\hspace{-0,34cm}\lambda_F} \grmptb \gnl
\gcl{1} \gcu{1} \gnl
\gob{1}{F}
\gend=
\gbeg{2}{3}
\got{1}{F} \got{1}{F} \gnl
\gmu \gnl
\gob{2}{F}
\gend,
\hspace{1cm}
\gbeg{2}{4}
\got{1}{F} \gnl
\gcl{1}\gu{1}  \gnl
\glmptb \gnot{\hspace{-0,34cm}\lambda_F} \grmptb \gnl
\gob{1}{F} \gob{1}{F}
\gend=
\gbeg{2}{3}
\got{2}{F} \gnl
\gcmu \gnl
\gob{1}{F} \gob{1}{F}
\gend
$$
and having in mind \equref{(co)mod (co)alg (co)unital}--\equref{right F-mod counits}: 
$$
\gbeg{2}{4}
\got{1}{B} \gnl
\gcl{1} \gu{1} \gnl
\glmptb \gnot{\hspace{-0,34cm}\lambda_B} \grmptb \gnl
\gob{1}{B} \gob{1}{B}
\gend=
\gbeg{2}{4}
\got{3}{B} \gnl
\gu{1} \gcl{1} \gnl
\gcl{1} \gcl{1} \gnl
\gob{1}{B} \gob{1}{B}
\gend,
\hspace{1cm}
\gbeg{2}{4}
\got{1}{B} \got{1}{B} \gnl
\glmptb \gnot{\hspace{-0,34cm}\lambda_B} \grmptb \gnl
\gcl{1} \gcu{1} \gnl
\gob{1}{B}
\gend=
\gbeg{2}{4}
\got{1}{B} \got{1}{B} \gnl
\gcl{1} \gcl{2} \gnl
\gcu{1} \gnl
\gob{3}{B}
\gend;
\hspace{1,5cm}
\gbeg{2}{4}
\got{3}{F} \gnl
\gu{1} \gcl{1} \gnl
\glmptb \gnot{\hspace{-0,34cm}\lambda_F} \grmptb \gnl
\gob{1}{F} \gob{1}{F}
\gend=
\gbeg{2}{4}
\got{1}{F} \gnl
\gcl{1} \gu{1} \gnl
\gcl{1} \gcl{1} \gnl
\gob{1}{F} \gob{1}{F}
\gend,
\hspace{1cm}
\gbeg{2}{4}
\got{1}{F} \got{1}{F} \gnl
\glmptb \gnot{\hspace{-0,34cm}\lambda_F} \grmptb \gnl
\gcu{1} \gcl{1} \gnl
\gob{3}{F}
\gend=
\gbeg{2}{4}
\got{1}{F} \got{1}{F} \gnl
\gcl{2}  \gcu{1} \gnl
\gob{1}{F}
\gend
$$
We also may write now $\mu_M$ and $\Delta_C'$ as follows: 
\begin{equation} \eqlabel{mu-delta-final}
\mu_M=
\gbeg{5}{7}
\got{1}{} \got{1}{F} \got{4}{F} \gnl
\gwcm{3} \gcmu \gnl
\gcl{1} \grmo \gvac{1} \gcl{1} \gcl{1} \gcl{3} \gnl
\gcl{1} \gcl{1} \glmptb \gnot{\hspace{-0,34cm}\tau_{F,F}} \grmptb \gnl
\gcl{1} \glm \gcl{1} \gnl
\gwmu{3} \glmpt \gnot{\hspace{-0,34cm}\sigma} \grmptb \gnl
\gvac{1} \gob{1}{F} \gob{5}{B}
\gend
\hspace{2cm}
\Delta_C'=
\gbeg{5}{7}
\got{1}{F} \got{5}{B} \gnl
\glmptb \gnot{\hspace{-0,34cm}\rho'} \grmpb \gwcm{3} \gnl
\gcl{1} \gcl{1} \grcm \gcl{3} \gnl
\gcl{1} \glmptb \gnot{\hspace{-0,34cm}\tau_{B,B}} \grmptb \gcl{1} \gnl
\gcl{1} \gcl{1} \grmo \gcl{1} \gnl
\gmu \gwmu{3} \gnl
\gob{2}{B}  \gob{3}{B}
\gend
\end{equation}

\bigskip

Finally, let us see what we get from knowing that in a paired wreath $FB$ is a $\tau$-bimonad. Recall that the multiplication $\nabla_{FB}$ and comultiplication 
$\Delta_{FB}$ of $FB$ are given by \equref{wreath (co)product}. The first three identities in \deref{tau-bim} applied to $FB$ are already fulfilled by the first four points 
from the definition of a paired wreath, concretely, by \equref{(co)unital}, \equref{psi unital add}, \equref{epsilon to sigma} and \equref{unit to ro'}. 
If we apply 
$\eta_F BF\eta_B, F\eta_B F\eta_B$ and $\eta_F B\eta_F B$ 
to $\nabla_{FB}$, we obtain: $\psi, \mu_M$ 
and $\gbeg{3}{3}
\got{1}{} \got{1}{B} \got{1}{B} \gnl
\gu{1} \gmu \gnl
\gob{1}{F} \gob{2}{B}
\gend $, respectively. By $\pi$-symmetry, when we apply 
$\Epsilon_F BF\Epsilon_F, \Epsilon_F B\Epsilon_FB$ and $F\Epsilon_B F\Epsilon_B$ 
to $\Delta_{FB}$, we obtain: $\phi', \Delta_C$ 
and $\gbeg{3}{3}
\got{2}{F} \got{1}{B} \gnl
\gcmu \gcu{1} \gnl
\gob{1}{F} \gob{1}{F}
\gend$, respectively. We apply the $3\times 3$ combinations of these operations to: 
\begin{equation}  \eqlabel{paired biwreath biproduct} \hspace{-4cm} 
\gbeg{7}{8}
\gvac{2} \got{1}{F} \got{1}{B} \got{1}{F} \got{3}{B}  \gnl
\gvac{2} \gcl{1}  \glmptb \gnot{\hspace{-0,34cm}\psi} \grmptb \gvac{1} \gcl{1} \gnl
\gvac{2} \glmptb \gnot{\hspace{-0,34cm}\mu_M} \grmptb \gwmu{3} \gnl
\gvac{2} \gcl{1} \gwmu{3} \gnl
\gvac{1} \gwcm{3} \gcl{1} \gnl
\gwcm{3} \glmptb \gnot{\hspace{-0,34cm}\Delta_C'} \grmptb \gnl
\gcl{1} \gvac{1} \glmptb \gnot{\hspace{-0,34cm}\phi'} \grmptb \gcl{1} \gnl
\gob{1}{F} \gvac{1} \gob{1}{B} \got{1}{F} \got{1}{B}
\gend=
\gbeg{3}{11}
\gvac{2} \got{1}{F} \gvac{1} \got{1}{B} \gvac{2} \got{1}{F} \got{3}{B}  \gnl
\gvac{1} \gwcm{3} \gcl{1} \gvac{1} \gwcm{3} \gcl{1} \gnl
\gwcm{3} \glmptb \gnot{\hspace{-0,34cm}\Delta_C'} \grmptb \gwcm{3} \glmptb \gnot{\hspace{-0,34cm}\Delta_C'} \grmptb \gnl
\gcl{1} \gvac{1} \glmptb \gnot{\hspace{-0,34cm}\phi'} \grmptb \glmptb \gnot{\hspace{-0,34cm}\tau_{B,F}} \grmptb \gvac{1} \glmptb \gnot{\hspace{-0,34cm}\phi'} \grmptb \gcl{5} \gnl
\gcl{1} \gvac{1} \gcn{1}{2}{1}{-1} \gcl{1} \gcl{1} \gcl{1} \gvac{1} \gcn{1}{1}{1}{-1} \gcl{1}  \gnl 
\gcl{1} \gvac{2} \glmptb \gnot{\hspace{-0,34cm}\tau_{F,F}} \grmptb  \glmptb \gnot{\hspace{-0,34cm}\tau_{B,B}} \grmptb \gvac{1} \gcn{1}{2}{1}{-1} \gnl %
\gcl{1} \gcl{1} \gcn{2}{1}{3}{1}  \gcl{1} \gcl{1} \gcl{1} \gnl 
\gcl{1} \glmptb \gnot{\hspace{-0,34cm}\psi} \grmptb \gvac{1} \glmptb \gnot{\hspace{-0,34cm}\tau_{F,B}} \grmptb \glmptb \gnot{\hspace{-0,34cm}\psi} \grmptb \gnl
\glmptb \gnot{\hspace{-0,34cm}\mu_M} \grmptb \gwmu{3} \glmptb \gnot{\hspace{-0,34cm}\mu_M} \grmptb \gwmu{3} \gnl
\gcl{1} \gwmu{3} \gvac{1} \gcl{1} \gwmu{3} \gnl
\gob{1}{F} \gvac{1} \gob{1}{B} \gvac{2} \got{1}{F} \got{3}{B}
\gend 
\end{equation} 
and we obtain 8 identities (one of them yields a trivial identity). Applying respectively the operations: 
$$(\Epsilon_F BF\Epsilon_B)(-)(\eta_F BF\eta_B), \qquad (\Epsilon_F B\Epsilon_FB)(-)(\eta_F BF\eta_B), \qquad (F\Epsilon_B F\Epsilon_B)(-)(\eta_F BF\eta_B),$$ 
$$(\Epsilon_F BF\Epsilon_F)(-)(F\eta_B F\eta_B), \qquad (\Epsilon_F B\Epsilon_FB)(-)(F\eta_B F\eta_B), \qquad (\Epsilon_F BF\Epsilon_F)(-)(\eta_F B\eta_F B), $$
$$(\Epsilon_F B\Epsilon_FB)(-)(\eta_F B\eta_F B), \qquad (F\Epsilon_B F\Epsilon_B)(-)(F\eta_B F\eta_B)$$
we get: 
\begin{equation} \eqlabel{1-3}
\gbeg{2}{4}
\got{1}{B} \got{1}{F} \gnl
\glmptb \gnot{\hspace{-0,34cm}\psi} \grmptb \gnl
\glmptb \gnot{\hspace{-0,34cm}\phi'} \grmptb \gnl
\gob{1}{B} \gob{1}{F} 
\gend=
\gbeg{6}{7}
\gvac{1} \got{1}{F} \got{5}{B} \gnl
\gwcm{3} \gwcm{3} \gnl
\grcm \glmptb \gnot{\hspace{-0,34cm}\tau_{F,B}} \grmptb \grmo \gvac{1} \gcl{1} \gnl
\gcl{1} \glmptb \gnot{\hspace{-0,34cm}\tau_{B,B}} \grmptb \glmptb \gnot{\hspace{-0,34cm}\tau_{B,F}} \grmptb \gcl{1} \gnl %
\grmo \gcl{1} \gvac{1} \glmptb \gnot{\hspace{-0,34cm}\tau_{B,F}} \grmptb \glm \gnl
\gwmu{3} \gwmu{3} \gnl
\gvac{1} \gob{1}{B} \gvac{2} \gob{1}{F} 
\gend, 
\hspace{1cm}
\gbeg{2}{4}
\got{1}{B} \got{1}{F} \gnl
\glmptb \gnot{\hspace{-0,34cm}\psi} \grmptb \gnl
\glmptb \gnot{\hspace{-0,34cm}\Delta_C'} \grmptb \gnl
\gob{1}{B} \gob{1}{B} 
\gend=
\gbeg{8}{10}
\gvac{1} \got{1}{B} \gvac{3} \got{1}{F} \gnl
\gvac{1} \gcl{1} \gvac{2} \gwcm{3} \gnl
\gwcm{3} \gwcm{3} \glmptb \gnot{\hspace{-0,34cm}\ro'} \grmpb \gnl
\grcm \glmptb \gnot{\hspace{-0,34cm}\tau_{B,F}} \grmptb \gvac{1} \glmptb \gnot{\hspace{-0,34cm}\phi'} \grmptb \gcl{4} \gnl
\gcl{1} \gcl{1} \gcl{1} \gcl{1} \gvac{1} \gcn{1}{1}{1}{-1} \gcn{1}{2}{1}{-1} \gnl 
\gcl{1} \glmptb \gnot{\hspace{-0,34cm}\tau_{F,F}} \grmptb  \glmptb \gnot{\hspace{-0,34cm}\tau_{B,B}} \grmptb \gnl %
\grmo \gcl{1} \gvac{1} \glmptb \gnot{\hspace{-0,34cm}\tau_{F,B}} \grmptb \glmptb \gnot{\hspace{-0,34cm}\psi} \grmptb \gnl
\gwmu{3} \glmpt \gnot{\hspace{-0,34cm}\sigma} \grmptb \gwmu{3} \gnl
\gvac{1} \gcl{1} \gvac{2} \gwmu{3} \gnl
\gvac{1} \gob{1}{B} \gvac{2} \gob{3}{B}
\gend, 
\hspace{1cm}
\gbeg{2}{4}
\got{1}{B} \got{1}{F} \gnl
\glm \gnl
\gvac{1} \hspace{-0,22cm} \gcmu \gnl
\gvac{1} \gob{1}{F} \gob{1}{F} 
\gend=
\gbeg{4}{6}
\gvac{1} \got{1}{B} \got{4}{F} \gnl
\gwcm{3} \gcmu \gnl
\grcm \glmptb \gnot{\hspace{-0,34cm}\tau_{B,F}} \grmptb \gcl{1} \gnl
\gcl{1} \glmptb \gnot{\hspace{-0,34cm}\tau_{F,F}} \grmptb \glm \gnl 
\glm \gwmu{3} \gnl
\gvac{1} \gob{1}{F} \gvac{1} \gob{1}{F} 
\gend
\end{equation}

\begin{equation} \eqlabel{4-6}
\gbeg{2}{4}
\got{1}{F} \got{1}{F} \gnl
\glmptb \gnot{\hspace{-0,34cm}\mu_M} \grmptb \gnl
\glmptb \gnot{\hspace{-0,34cm}\phi'} \grmptb \gnl
\gob{1}{B} \gob{1}{F} 
\gend=
\gbeg{8}{10}
\gvac{2} \got{1}{F} \gvac{3}  \got{1}{F} \gnl
\gvac{1} \gwcm{3} \gvac{2} \gcl{1}  \gnl
\gwcm{3} \glmptb \gnot{\hspace{-0,34cm}\ro'} \grmpb \gwcm{3}  \gnl
\gcl{1} \gvac{1} \glmptb \gnot{\hspace{-0,34cm}\phi'} \grmptb \glmptb \gnot{\hspace{-0,34cm}\tau_{B,F}} \grmptb \grmo \gvac{1} \gcl{1} \gnl
\gcl{1} \gvac{1} \gcl{1} \glmptb \gnot{\hspace{-0,34cm}\tau_{F,F}} \grmptb  \glmptb \gnot{\hspace{-0,34cm}\tau_{B,B}} \grmptb \gcl{1} \gnl %
\gcn{1}{1}{1}{3} \gvac{1} \glmptb \gnot{\hspace{-0,34cm}\psi} \grmptb \glmptb \gnot{\hspace{-0,34cm}\tau_{F,B}} \grmptb \glm \gnl
\gvac{1} \glmpt \gnot{\hspace{-0,34cm}\sigma} \grmptb \gmu \gwmu{3} \gnl
\gvac{2} \gcl{1} \gcn{1}{1}{2}{1} \gvac{2} \gcl{2} \gnl
\gvac{2}  \gmu \gnl
\gvac{2}  \gob{2}{B} \gvac{2}  \gob{1}{F} 
\gend, 
\hspace{0,5cm}
\gbeg{2}{4}
\got{1}{F} \got{1}{F} \gnl
\glmptb \gnot{\hspace{-0,34cm}\mu_M} \grmptb \gnl
\glmptb \gnot{\hspace{-0,34cm}\Delta_C'} \grmptb \gnl
\gob{1}{B} \gob{1}{B} 
\gend=
\gbeg{10}{11}
\gvac{2} \got{1}{F} \gvac{4}  \got{1}{F} \gnl
\gvac{1} \gwcm{3} \gvac{2} \gwcm{3} \gnl
\gwcm{3} \glmptb \gnot{\hspace{-0,34cm}\ro'} \grmpb \gwcm{3} \glmptb \gnot{\hspace{-0,34cm}\ro'} \grmpb \gnl
\gcl{1} \gvac{1} \glmptb \gnot{\hspace{-0,34cm}\phi'} \grmptb \glmptb \gnot{\hspace{-0,34cm}\tau_{B,F}} \grmptb \gvac{1} \glmptb \gnot{\hspace{-0,34cm}\phi'} \grmptb \gcl{4} \gnl
\gcl{1} \gvac{1} \gcl{2} \gcl{1} \gcl{1} \gcl{1} \gvac{1} \gcn{1}{1}{1}{-1} \gcn{1}{2}{1}{-1}  \gnl 
\gcl{1} \gvac{2} \glmptb \gnot{\hspace{-0,34cm}\tau_{F,F}} \grmptb  \glmptb \gnot{\hspace{-0,34cm}\tau_{B,B}} \grmptb \gnl %
\gcn{1}{1}{1}{3} \gvac{1} \glmptb \gnot{\hspace{-0,34cm}\psi} \grmptb \glmptb \gnot{\hspace{-0,34cm}\tau_{F,B}} \grmptb \glmptb \gnot{\hspace{-0,34cm}\psi} \grmptb \gnl
\gvac{1} \glmpt \gnot{\hspace{-0,34cm}\sigma} \grmptb \gmu \glmpt \gnot{\hspace{-0,34cm}\sigma} \grmptb \gwmu{3} \gnl
\gvac{2} \gcl{1} \gcn{1}{1}{2}{1} \gvac{2} \gwmu{3} \gnl
\gvac{2}  \gmu \gvac{3} \gcl{1} \gnl
\gvac{2}  \gob{2}{B} \gvac{3}  \gob{1}{B} 
\gend, 
\hspace{0,2cm}
\gbeg{2}{4}
\got{1}{B} \got{1}{B} \gnl
\gmu \gnl
\gvac{1} \hspace{-0,22cm} \grcm \gnl
\gvac{1} \gob{1}{B} \gob{1}{F} 
\gend=
\gbeg{4}{6}
\gvac{1} \got{1}{B} \got{3}{B} \gnl
\gwcm{3} \grcm \gnl
\grcm \glmptb \gnot{\hspace{-0,34cm}\tau_{B,B}} \grmptb \gcl{1} \gnl
\gcl{1} \glmptb \gnot{\hspace{-0,34cm}\tau_{F,B}} \grmptb \glm \gnl 
\gmu \gwmu{3} \gnl
\gob{2}{F} \gvac{1} \gob{1}{F} 
\gend
\end{equation}

\begin{equation} \eqlabel{proj B-bialg}
\gbeg{3}{5}
\got{1}{B} \got{3}{B} \gnl
\gwmu{3} \gnl
\gvac{1} \gcl{1} \gnl
\gwcm{3} \gnl
\gob{1}{B} \gob{3}{B} 
\gend=
\gbeg{6}{9}
\gvac{1} \got{1}{B} \got{5}{B} \gnl
\gwcm{3} \gwcm{3} \gnl
\grcm \gcl{1} \grcm \gcl{3} \gnl
\gcl{1} \gcl{1} \glmptb \gnot{\hspace{-0,34cm}\tau_{B,B}} \grmptb \gcl{1} \gnl
\gcl{1} \glmptb \gnot{\hspace{-0,34cm}\tau_{F,B}} \grmptb \glmptb \gnot{\hspace{-0,34cm}\psi} \grmptb \gnl 
\gmu \glmpt \gnot{\hspace{-0,34cm}\sigma} \grmptb \gmu \gnl
\gcn{1}{2}{2}{2} \gvac{2} \gcl{1} \gcn{1}{1}{2}{1} \gnl
\gvac{3} \gmu \gnl
\gob{2}{F} \gob{4}{F} 
\gend
\stackrel{\equref{proj B}}{=}
\gbeg{3}{5}
\got{2}{B} \got{1}{B} \gnl
\gcmu \gcl{1} \gnl
\gcl{1} \glmptb \gnot{\hspace{-0,34cm}\lambda_B} \grmptb \gnl
\gmu \gcl{1} \gnl
\gob{2}{B}  \gob{1}{B}
\gend\stackrel{\equref{lambda B}}{=}
\gbeg{4}{7}
\gvac{1} \got{1}{B} \got{5}{B} \gnl
\gwcm{3} \gwcm{3} \gnl
\gcl{1} \gvac{1} \gcl{1} \grcm \gcl{3} \gnl
\gcl{1} \gvac{1} \glmptb \gnot{\hspace{-0,34cm}\tau_{B,B}} \grmptb \gcl{1} \gnl
\gcl{1} \gvac{1} \gcl{1} \grmo \gcl{1} \gvac{1} \gnl
\gwmu{3} \gwmu{3} \gnl
\gvac{1} \gob{1}{B}  \gob{5}{B}
\gend
\end{equation}

\begin{equation} \eqlabel{proj F-bialg}
\gbeg{3}{5}
\got{1}{F} \got{3}{F} \gnl
\gwmu{3} \gnl
\gvac{1} \gcl{1} \gnl
\gwcm{3} \gnl
\gob{1}{F} \gob{3}{F} 
\gend=
\gbeg{6}{9}
\gvac{1} \got{2}{F} \gvac{1} \got{2}{F} \gnl
\gvac{1} \gcmu \gvac{1} \gcn{1}{1}{2}{2} \gnl
\gcn{2}{1}{3}{2} \gcl{1} \gvac{1} \gcmu \gnl
\gcmu \glmptb \gnot{\hspace{-0,34cm}\ro'} \grmpb \gcl{1} \gcl{1} \gnl
\gcl{3} \glmptb \gnot{\hspace{-0,34cm}\phi'} \grmptb \glmptb \gnot{\hspace{-0,34cm}\tau_{B,F}} \grmptb \gcl{1}  \gnl %
\gvac{1} \gcl{1} \glmptb \gnot{\hspace{-0,34cm}\tau_{F,F}} \grmptb \glm \gnl
\gvac{1} \glm \gwmu{3} \gnl
\gwmu{3} \gcn{1}{1}{3}{3} \gnl
\gvac{1} \gob{1}{F} \gvac{2} \gob{1}{F} 
\gend 
\stackrel{\equref{proj F}}{=}
\gbeg{3}{5}
\got{1}{F} \got{2}{F} \gnl
\gcl{1} \gcmu \gnl
\glmptb \gnot{\hspace{-0,34cm}\lambda_F} \grmptb  \gcl{1} \gnl
\gcl{1} \gmu \gnl
\gob{1}{F}  \gob{2}{F}
\gend\stackrel{\equref{lambda F}}{=}
\gbeg{4}{7}
\got{1}{} \got{1}{F} \got{5}{F} \gnl
\gwcm{3} \gwcm{3} \gnl
\gcl{1} \grmo \gvac{1} \gcl{1} \gcl{1} \gvac{1} \gcl{3} \gnl
\gcl{1} \gcl{1} \glmptb \gnot{\hspace{-0,34cm}\tau_{F,F}} \grmptb \gnl
\gcl{1} \glm \gcl{1} \gnl
\gwmu{3} \gwmu{3} \gnl
\gvac{1} \gob{1}{F} \gob{5}{F}
\gend
\end{equation}

\section{Hopf data}

We are going to collect the following properties of a paired wreath: 
points 1. and 2. of \deref{paired bl object}, 
the identities: \equref{mod alg} -- \equref{unit to ro'} and \equref{1-3} -- \equref{proj F-bialg} with the corresponding 2-cells, and the thesis of \coref{(co)actions},
and give a name to that collection of data.

\begin{defn}
A Hopf datum in $\K$ consists of the following data: 
\begin{enumerate}
\item a monad $(B, 
\gbeg{2}{3}
\got{1}{B} \got{1}{B} \gnl
\gmu \gnl
\gob{2}{B}
\gend, 
\gbeg{1}{2}
\gu{1} \gnl
\gob{1}{B}
\gend)$, a comonad $(F, 
\gbeg{2}{3}
\got{2}{F} \gnl
\gcmu \gnl
\gob{1}{F} \gob{1}{F}
\gend, 
\gbeg{1}{2}
\got{1}{F} \gnl
\gcu{1} \gnl
\gend)$ both over the same 0-cell $\A$ in $\K$, with 2-cells 
$\gbeg{2}{3}
\got{2}{B} \gnl
\gcmu \gnl
\gob{1}{B} \gob{1}{B}
\gend, 
\gbeg{1}{2}
\got{1}{B} \gnl
\gcu{1} \gnl
\gend$ and 
$\gbeg{2}{3}
\got{1}{F} \got{1}{F} \gnl
\gmu \gnl
\gob{2}{F}
\gend,  
\gbeg{1}{2}
\gu{1} \gnl
\gob{1}{F} \gnl
\gend$ in $\K$, satisfying the compatibility conditions:
$$
\gbeg{2}{3}
\got{1}{B} \got{1}{B} \gnl
\gmu \gnl
\gvac{1} \hspace{-0,22cm} \gcu{1} \gnl
\gend=
\gbeg{2}{3}
\got{1}{B} \got{1}{B} \gnl
\gcl{1} \gcl{1} \gnl
\gcu{1} \gcu{1} \gnl
\gend, \hspace{0,5cm}
\gbeg{1}{3}
 \gu{1} \gnl
\hspace{-0,34cm} \gcmu \gnl
\gob{1}{B} \gob{1}{B} \gnl
\gend=
\gbeg{2}{3}
\gu{1} \gu{1} \gnl
\gcl{1} \gcl{1} \gnl
\gob{1}{B} \gob{1}{B} \gnl
\gend, \hspace{0,2cm}
\gbeg{1}{2}
\gu{1} \gnl
\gcu{1} \gnl
\gob{2}{} \gnl
\gend B=
\Id_{id_{\A}},
\hspace{0,7cm}
\gbeg{1}{3}
 \gu{1} \gnl
\hspace{-0,34cm} \gcmu \gnl
\gob{1}{F} \gob{1}{F} \gnl
\gend=
\gbeg{2}{3}
\gu{1} \gu{1} \gnl
\gcl{1} \gcl{1} \gnl
\gob{1}{F} \gob{1}{F}
\gend, \hspace{0,2cm}
\gbeg{2}{3}
\got{1}{F} \got{1}{F} \gnl
\gmu \gnl
\gvac{1} \hspace{-0,22cm} \gcu{1} \gnl
\gend=
\gbeg{2}{3}
\got{1}{F} \got{1}{F} \gnl
\gcl{1} \gcl{1} \gnl
\gcu{1} \gcu{1} \gnl
\gend, \hspace{0,2cm}
\gbeg{1}{2}
\gu{1} \gnl
\gcu{1} \gnl
\gob{2}{} \gnl
\gend F=
\Id_{id_{\A}};
$$
\item further 2-cells: 
\begin{equation} \eqlabel{(co)module cells}
\gbeg{2}{3}
\got{1}{B} \got{1}{F} \gnl
\glm \gnl
\gob{3}{F}
\gend, \hspace{0,6cm}
\gbeg{2}{3}
\got{1}{B} \gnl
\grcm \gnl
\gob{1}{B} \gob{1}{F}
\gend, \hspace{0,6cm}
\gbeg{2}{3}
\got{1}{B} \got{1}{F} \gnl
\grmo \gcl{1} \gnl 
\gob{1}{B}
\gend, \hspace{0,6cm}
\gbeg{2}{3}
\got{3}{F} \gnl
\grmo \gvac{1} \gcl{1} \gnl
\gob{1}{B} \gob{1}{F}
\gend
\end{equation}
so that 
$\gbeg{2}{1}
\glm \gnl
\gend$ makes $F$ a proper left $B$-module and 
$\gbeg{2}{1}
\grcm \gnl
\gend$ makes $B$ a proper right $F$-comodule, and the following relations are fulfilled: 
$$
\gbeg{2}{3}
\got{1}{B} \got{1}{F} \gnl
\glm \gnl
\gvac{1} \gcu{1}
\gend=
\gbeg{2}{3}
\got{1}{B} \got{1}{F} \gnl
\gcl{1} \gcl{1} \gnl
\gcu{1} \gcu{1} \gnl
\gend,
\hspace{0,8cm}
\gbeg{2}{3}
\gu{1} \gnl
\grcm \gnl
\gob{1}{B} \gob{1}{F}
\gend=
\gbeg{2}{3}
\gu{1} \gu{1} \gnl
\gcl{1} \gcl{1} \gnl
\gob{1}{B} \gob{1}{F}
\gend,
\hspace{1,4cm}
\gbeg{2}{3}
\got{1}{B} \got{1}{F} \gnl
\grmo \gcl{1} \gnl 
\gcu{1}
\gend=
\gbeg{2}{3}
\got{1}{B} \got{1}{F} \gnl
\gcl{1} \gcl{1} \gnl
\gcu{1} \gcu{1} \gnl
\gend,
\hspace{0,8cm}
\gbeg{2}{3}
\gvac{1} \gu{1} \gnl
\grmo \gvac{1} \gcl{1} \gnl
\gob{1}{B} \gob{1}{F}
\gend=
\gbeg{2}{3}
\gu{1} \gu{1} \gnl
\gcl{1} \gcl{1} \gnl
\gob{1}{B} \gob{1}{F}
\gend
$$
\item 2-cells: 
$$
\gbeg{2}{3}
\got{1}{F} \got{1}{F} \gnl
\glmpt \gnot{\hspace{-0,34cm}\sigma} \grmptb \gnl
\gob{3}{B}
\gend \hspace{0,2cm}, \hspace{0,6cm}
\gbeg{3}{3}
\got{3}{F} \gnl
\glmpb \gnot{\hspace{-0,34cm}\rho'} \grmptb \gnl
\gob{1}{B}\gob{1}{B}
\gend
$$
and 
$$\tau_{B,F}:BF\to FB, \tau_{F,B}:FB\to BF, \tau_{B,B}:BB\to BB, \tau_{F,F}:FF\to FF$$
where $\tau_{B,F}, \tau_{F,F}, \tau_{F,F}, \tau_{F,B}$ are left and right monadic and comonadic distributive laws, where the adjectives ``comonadic'' corresponding to $B$ 
and ``monadic'' corresponding to $F$ are meant with respect to the 2-cells from the point 1.;  
\item so that setting: 
$$
\psi=
\gbeg{4}{5}
\got{2}{B} \got{2}{F} \gnl
\gcmu \gcmu \gnl
\gcl{1} \glmptb \gnot{\hspace{-0,34cm}\tau_{B,F}} \grmptb \gcl{1} \gnl
\glm \grmo \gcl{1} \gnl
\gob{1}{} \gob{1}{F} \gob{1}{B} 
\gend, 
\hspace{0,7cm}
\phi'= 
\gbeg{4}{5}
\got{1}{} \got{1}{F} \got{1}{B} \gnl
\grmo \gvac{1} \gcl{1} \grcm \gnl
\gcl{1} \glmptb \gnot{\hspace{-0,34cm}\tau_{F,B}} \grmptb \gcl{1} \gnl
\gmu \gmu \gnl
\gob{2}{B} \gob{2}{F} 
\gend,
\hspace{0,7cm}
\mu_M=
\gbeg{5}{7}
\got{1}{} \got{1}{F} \got{4}{F} \gnl
\gwcm{3} \gcmu \gnl
\gcl{1} \glcm \gcl{1} \gcl{3} \gnl
\gcl{1} \gcl{1} \glmptb \gnot{\hspace{-0,34cm}\tau_{F,F}} \grmptb \gnl
\gcl{1} \glm \gcl{1} \gnl
\gwmu{3} \glmpt \gnot{\hspace{-0,34cm}\sigma} \grmptb \gnl
\gvac{1} \gob{1}{F} \gob{5}{B}
\gend,
\hspace{1cm}
\Delta_C'=
\gbeg{5}{7}
\got{1}{F} \got{5}{B} \gnl
\glmptb \gnot{\hspace{-0,34cm}\rho'} \grmpb \gwcm{3} \gnl
\gcl{1} \gcl{1} \grcm \gcl{3} \gnl
\gcl{1} \glmptb \gnot{\hspace{-0,34cm}\tau_{B,B}} \grmptb \gcl{1} \gnl
\gcl{1} \gcl{1} \grm \gnl
\gmu \gwmu{3} \gnl
\gob{2}{B}  \gob{3}{B}
\gend
$$
the identities \equref{mod alg} -- \equref{unit to ro'} and \equref{1-3} -- \equref{proj F-bialg} hold true. 
\end{enumerate}
\end{defn}

\medskip

\begin{cor} \colabel{Hd conseq}
In a Hopf datum the following identities hold true:
$$
\gbeg{2}{4}
\got{1}{B} \got{1}{F} \gnl
\glmptb \gnot{\hspace{-0,34cm}\psi} \grmptb \gnl
\gcl{1} \gcu{1} \gnl
\gob{1}{F}
\gend=
\gbeg{2}{3}
\got{1}{B} \got{1}{F} \gnl
\glm \gnl
\gob{3}{F}
\gend
\hspace{1,5cm}
\gbeg{2}{4}
\got{1}{B} \got{1}{F} \gnl
\glmptb \gnot{\hspace{-0,34cm}\psi} \grmptb \gnl
\gcu{1} \gcl{1} \gnl
\gob{3}{B} 
\gend=
\gbeg{2}{3}
\got{1}{B} \got{1}{F} \gnl
\grmo \gcl{1} \gnl 
\gob{1}{B}
\gend
\hspace{1,5cm}
\gbeg{2}{4}
\got{1}{F} \gnl
\gcl{1} \gu{1} \gnl
\glmptb \gnot{\hspace{-0,34cm}\phi'} \grmptb \gnl
\gob{1}{B} \gob{1}{F}
\gend=
\gbeg{2}{3}
\got{3}{F} \gnl
\grmo \gvac{1} \gcl{1} \gnl
\gob{1}{B} \gob{1}{F}
\gend
\hspace{1,5cm}
\gbeg{2}{4}
\got{3}{B} \gnl
\gu{1} \gcl{1} \gnl
\glmptb \gnot{\hspace{-0,34cm}\phi'} \grmptb \gnl
\gob{1}{B} \gob{1}{F}
\gend=
\gbeg{2}{3}
\got{1}{B} \gnl
\grcm \gnl
\gob{1}{B} \gob{1}{F}
\gend
$$

$$
\gbeg{2}{4}
\got{3}{F} \gnl
\gu{1} \gcl{1} \gnl
\glmptb \gnot{\hspace{-0,34cm}\psi} \grmptb \gnl
\gob{1}{F} \gob{1}{B}
\gend=
\gbeg{2}{4}
\got{1}{F} \gnl
\gcl{1} \gu{1} \gnl
\gcl{1} \gcl{1} \gnl
\gob{1}{F} \gob{1}{B}
\gend
\hspace{1,5cm}
\gbeg{2}{4}
\got{1}{F} \got{1}{B} \gnl
\glmptb \gnot{\hspace{-0,34cm}\phi'} \grmptb \gnl
\gcl{1} \gcu{1} \gnl
\gob{1}{B}
\gend=
\gbeg{2}{4}
\got{1}{F} \got{1}{B} \gnl
\gcu{1} \gcl{2} \gnl
\gob{3}{B}
\gend
$$
\end{cor}

\begin{proof}
The first four identities follow by the point 2 of the Definition and relations \equref{weak counit Epsilon_C'} and \equref{weak unity eta_M}, the last two follow by 
the point 1 of the Definition and relations \equref{F mod alg unit} and \equref{B comod coalg counit}.
\qed\end{proof}

\begin{cor} \colabel{Hopf data}
Every paired wreath is a Hopf datum. 
\end{cor}

\begin{rem} \rmlabel{when cocycles are trivial}
Let us record some properties of a Hopf datum. 
Observe that the identities \equref{mod alg}-\equref{comod coalg counity} are $\alpha$-symmetric to the identities \equref{F mod alg}-\equref{B comod coalg counit}, and 
that the first identity in \equref{1-3} is auto $\alpha$- and $\pi$-symmetric. Moreover, if $\sigma$ and $\rho'$ are trivial, the 2-cells $\mu_M$ and $\Delta_C'$ are canonical 
and we have: 
\begin{itemize}
\item the identities \equref{2-cocycle condition}--\equref{unit to ro'} trivially hold, 
\item the identities \equref{weak action}--\equref{weak coaction counity} say that the respective (co)actions (\equref{right F-mod}, \equref{left B-comod}) are proper,  
which is $\alpha$-symmetric to the fact that the (co)actions \equref{F left B-mod}, \equref{B right F-comod} are proper (\coref{(co)actions}),
\item the identities \equref{weak assoc. mu_M}--\equref{weak counit Epsilon_C'} say that $F$ is a monad and $B$ is a comonad 
(this is $\alpha$-symmetric to $B$ being a monad and $F$ a comonad),
\item the second identity in \equref{4-6} holds trivially, 
\item the second identity in \equref{1-3} is $\alpha$-symmetric to the third one therein, 
and the first identity in \equref{4-6} is $\alpha$-symmetric to the third one therein, 
\item the identity \equref{proj B-bialg} is $\alpha$-symmetric to \equref{proj F-bialg}.
\end{itemize}
We have already seen that the definition of a Hopf datum is auto $\pi$-symmetric, now we observe that when $\sigma$ and $\rho'$ are trivial 
a Hopf datum is also auto $\alpha$-symmetric. 
\end{rem}

\medskip

Observe that in a paired wreath and a Hopf datum some of the (co)module structures \equref{F left B-mod} -- \equref{left B-comod} can be trivial. 
The same holds for the (co)cycles $\sigma$ and $\rho'$, as we commented in the above Remark. This gives $2^6$ combinations of structures and thus $2^6$ 
types of paired wreaths and of Hopf data. Some of them have been studied classically in $\K=\hat Vec$ where $Vec$ is the category of vector spaces over a field $k$ (or modules over a commutative ring). 
The rest would yield new structures, not only of crossed products in $Vec$, but even in general $\K$.

\begin{ex}
Let $\K=\hat\C$ where $\C$ is a braided monoidal category with braiding 
$\gbeg{2}{1}
\gbr \gnl
\gend$. In this setting our Hopf datum is the Hopf datum defined in \cite[Section 4]{BD}. For trivial $\sigma$ and $\rho'$ one has the Hopf datum from \cite[Section 2]{BD1}. 
On the other hand, a normalized cross product bialgebra from \cite[Definition 3.5]{BD} differs from our paired wreath object in ``projection relations'' (3.7), which in our 
setting appear substituted by more general conditions \equref{proj B}, \equref{proj F}. Namely, one of the two symmetric projection relations is: 
$$\gbeg{4}{4}
\got{2}{B} \got{1}{F} \gnl
\gcmu \gcl{1} \gnl
\gcl{1} \grmo \gcl{1} \gnl
\gob{1}{B} \gob{1}{B}
\gend=
\gbeg{4}{6}
\got{3}{B} \got{1}{F} \gnl
\gwcm{3} \gcl{1} \gnl
\grcm \glmptb \gnot{\hspace{-0,34cm}\psi} \grmptb \gnl
\gcl{1} \glmpt \gnot{\hspace{-0,34cm}\sigma} \grmptb \gcl{1} \gnl
\gcl{1} \gvac{1} \gmu \gnl
\gob{1}{B} \gob{4}{B}
\gend
$$
Tensoring on the left by $B$ and applying braiding above and below and naturality, this implies: 
$$\gbeg{4}{5}
\got{2}{B} \got{1}{B} \got{1}{F} \gnl
\gcmu \gcl{1} \gcl{1} \gnl
\gcl{1} \gbr \gcl{1} \gnl
\gcl{1} \gcl{1} \grmo \gcl{1} \gnl
\gob{1}{B} \gob{1}{B} \gob{1}{B}
\gend=
\gbeg{4}{7}
\got{3}{B} \got{1}{B}\got{1}{F} \gnl
\gwcm{3} \gcl{1} \gcl{1} \gnl
\grcm \gbr \gcl{1} \gnl
\gcl{1} \gbr \glmptb \gnot{\hspace{-0,34cm}\psi} \grmptb \gnl
\gcl{1} \gcl{2} \glmpt \gnot{\hspace{-0,34cm}\sigma} \grmptb \gcl{1} \gnl
\gcl{1} \gvac{2} \gmu \gnl
\gob{1}{B} \gob{1}{B} \gob{4}{B}
\gend
\qquad\textnormal{which further implies:}\quad
\gbeg{5}{7}
\got{2}{B} \got{3}{B} \gnl
\gcn{2}{1}{2}{2} \gwcm{3} \gnl
\gcmu \grcm \gcl{3} \gnl
\gcl{1} \gbr \gcl{1} \gnl
\gcl{1} \gcl{1} \grmo \gcl{1} \gnl
\gcl{1} \gcl{1} \gwmu{3} \gnl
\gob{1}{B} \gob{1}{B} \gob{3}{B}
\gend=
\gbeg{5}{9}
\gvac{1} \got{1}{B} \got{5}{B} \gnl
\gcn{1}{1}{3}{3} \gvac{2} \gwcm{3} \gnl
\gwcm{3} \grcm \gcl{4} \gnl
\grcm \gbr \gcl{1} \gnl
\gcl{1} \gbr  \glmptb \gnot{\hspace{-0,34cm}\psi} \grmptb \gnl %
\gcl{3} \gcl{3}  \glmpt \gnot{\hspace{-0,34cm}\sigma} \grmptb \gcl{1} \gnl
\gvac{3} \gmu \gcn{1}{1}{1}{0} \gnl
\gcl{1} \gvac{3} \hspace{-0.32cm} \gmu \gnl
\gvac{1} \hspace{-0.24cm} \gob{1}{B} \gob{1}{B} \gvac{2} \gob{1}{B} 
\gend
$$
The latter is our \equref{proj B} with $\tau$'s corresponding to the braiding. The argument for the other identity is similar. 
\end{ex}

In \cite[Section 4]{BD} the authors show graphically a scheme of all the possible $2^6$ Hopf data in $\C$. The analogous scheme can be considered in a general $\K$.

\subsection{When a Hopf datum is a paired wreath}

For the converse of \coref{Hopf data} one should prove that given a Hopf datum, the points 2, 5 and 6 in the definition of a paired wreath hold true. 
To begin with, we will analyze which part of a Hopf datum assures that 
$(F, \psi: BF\to FB, \mu_M, \eta_M)$ is a left wreath around $B$ and $(B, \phi': FB\to BF, \Delta_C', \Epsilon_C')$ is a right cowreath around $F$, 
{\em i.e.} that the point 2 is fulfilled. We will study separate cases. 
Before we sum up our findings, we introduce:

\begin{defn} \delabel{mHd}
A {\em monad Hopf datum} in $\K$ consists of the following data: 
\begin{enumerate}
\item a monad $(B, 
\gbeg{2}{3}
\got{1}{B} \got{1}{B} \gnl
\gmu \gnl
\gob{2}{B}
\gend, 
\gbeg{1}{2}
\gu{1} \gnl
\gob{1}{B}
\gend)$, a 1-cell $F$, both over the same 0-cell $\A$ in $\K$, with 2-cells 
$\gbeg{2}{3}
\got{2}{B} \gnl
\gcmu \gnl
\gob{1}{B} \gob{1}{B}
\gend, 
\gbeg{1}{2}
\got{1}{B} \gnl
\gcu{1} \gnl
\gend$ and 
$\gbeg{2}{3}
\got{1}{F} \got{1}{F} \gnl
\gmu \gnl
\gob{2}{F}
\gend,  
\gbeg{1}{2}
\gu{1} \gnl
\gob{1}{F} \gnl
\gend, 
\gbeg{2}{3}
\got{2}{F} \gnl
\gcmu \gnl
\gob{1}{F} \gob{1}{F}
\gend, 
\gbeg{1}{2}
\got{1}{F} \gnl
\gcu{1} \gnl
\gend$ in $\K$, satisfying the compatibility conditions: 
$$
\gbeg{2}{3}
\got{1}{B} \got{1}{B} \gnl
\gmu \gnl
\gvac{1} \hspace{-0,22cm} \gcu{1} \gnl
\gend=
\gbeg{2}{3}
\got{1}{B} \got{1}{B} \gnl
\gcl{1} \gcl{1} \gnl
\gcu{1} \gcu{1} \gnl
\gend, \hspace{0,5cm}
\gbeg{1}{3}
 \gu{1} \gnl
\hspace{-0,34cm} \gcmu \gnl
\gob{1}{B} \gob{1}{B} \gnl
\gend=
\gbeg{2}{3}
\gu{1} \gu{1} \gnl
\gcl{1} \gcl{1} \gnl
\gob{1}{B} \gob{1}{B} \gnl
\gend, \hspace{0,2cm}
\gbeg{1}{2}
\gu{1} \gnl
\gcu{1} \gnl
\gob{2}{} \gnl
\gend B=
\Id_{id_{\A}},
\hspace{0,7cm}
\gbeg{1}{3}
 \gu{1} \gnl
\hspace{-0,34cm} \gcmu \gnl
\gob{1}{F} \gob{1}{F} \gnl
\gend=
\gbeg{2}{3}
\gu{1} \gu{1} \gnl
\gcl{1} \gcl{1} \gnl
\gob{1}{F} \gob{1}{F}
\gend, \hspace{0,2cm}
\gbeg{2}{3}
\got{1}{F} \got{1}{F} \gnl
\gmu \gnl
\gvac{1} \hspace{-0,22cm} \gcu{1} \gnl
\gend=
\gbeg{2}{3}
\got{1}{F} \got{1}{F} \gnl
\gcl{1} \gcl{1} \gnl
\gcu{1} \gcu{1} \gnl
\gend, \hspace{0,2cm}
\gbeg{1}{2}
\gu{1} \gnl
\gcu{1} \gnl
\gob{2}{} \gnl
\gend F=
\Id_{id_{\A}};
$$
\item further 2-cells: 
\begin{equation} \eqlabel{module cells}
\gbeg{2}{3}
\got{1}{B} \got{1}{F} \gnl
\glm \gnl
\gob{3}{F}
\gend, \hspace{0,6cm}
\gbeg{2}{3}
\got{1}{B} \got{1}{F} \gnl
\grmo \gcl{1} \gnl 
\gob{1}{B}
\gend
\end{equation}
so that 
$\gbeg{2}{1}
\glm \gnl
\gend$ makes $F$ a proper left $B$-module, and the following relations are fulfilled: 
$$
\gbeg{2}{3}
\got{1}{B} \got{1}{F} \gnl
\glm \gnl
\gvac{1} \gcu{1}
\gend=
\gbeg{2}{3}
\got{1}{B} \got{1}{F} \gnl
\gcl{1} \gcl{1} \gnl
\gcu{1} \gcu{1} \gnl
\gend,
\hspace{0,8cm}
\gbeg{2}{3}
\got{1}{B} \got{1}{F} \gnl
\grmo \gcl{1} \gnl 
\gcu{1}
\gend=
\gbeg{2}{3}
\got{1}{B} \got{1}{F} \gnl
\gcl{1} \gcl{1} \gnl
\gcu{1} \gcu{1} \gnl
\gend
$$
\item 2-cells: 
$$
\gbeg{2}{3}
\got{1}{F} \got{1}{F} \gnl
\glmpt \gnot{\hspace{-0,34cm}\sigma} \grmptb \gnl
\gob{3}{B}
\gend 
$$
and 
$$\tau_{B,F}:BF\to FB, \tau_{F,B}:FB\to BF, \tau_{B,B}:BB\to BB, \tau_{F,F}:FF\to FF$$
where $\tau_{B,F}, \tau_{F,F}, \tau_{F,F}, \tau_{F,B}$ are left and right monadic and comonadic distributive laws, where the adjectives ``monadic'' and ``monadic'' 
are meant with respect to the 2-cells from the point 1.;  
\item so that setting: 
$$
\psi=
\gbeg{4}{5}
\got{2}{B} \got{2}{F} \gnl
\gcmu \gcmu \gnl
\gcl{1} \glmptb \gnot{\hspace{-0,34cm}\tau_{B,F}} \grmptb \gcl{1} \gnl
\glm \grmo \gcl{1} \gnl
\gob{1}{} \gob{1}{F} \gob{1}{B} 
\gend, 
\hspace{0,7cm}
\mu_M=
\gbeg{5}{5}
\got{2}{F} \got{2}{F} \gnl
\gcmu \gcmu \gnl
\gcl{1} \glmptb \gnot{\hspace{-0,34cm}\tau_{F,F}} \grmptb \gcl{1} \gnl
\gmu \glmpt \gnot{\hspace{-0,34cm}\sigma} \grmptb \gnl
\gob{2}{F} \gob{3}{B}
\gend
$$
the identities \equref{mod alg} -- \equref{mod alg unity}, \equref{weak assoc. mu_M}--\equref{weak unity eta_M}, \equref{F mod alg} -- \equref{F mod alg unit}, 
\equref{weak action} -- \equref{weak action unity}, \equref{2-cocycle condition} -- \equref{normalized 2-cocycle} and \equref{epsilon to sigma} hold true. 
\end{enumerate}
\end{defn}

\begin{rem}
In the above Definition we could put 
$$
\mu_M=
\gbeg{5}{7}
\got{1}{} \got{1}{F} \got{4}{F} \gnl
\gwcm{3} \gcmu \gnl
\gcl{1} \grmo \gvac{1} \gcl{1} \gcl{1} \gcl{3} \gnl
\gcl{1} \gcl{1} \glmptb \gnot{\hspace{-0,34cm}\tau_{F,F}} \grmptb \gnl
\gcl{1} \glm \gcl{1} \gnl
\gwmu{3} \glmpt \gnot{\hspace{-0,34cm}\sigma} \grmptb \gnl
\gvac{1} \gob{1}{F} \gob{5}{B}
\gend
$$
assuming the existence of the 2-cell 
$\gbeg{2}{1}
\grmo \gvac{1} \gcl{1} 
\gend$ and the relation 
$
\gbeg{2}{4}
\got{3}{F} \gnl
\grmo \gvac{1} \gcl{1}\gnl
\gcl{1} \gcu{1} \gnl
\gob{1}{B} 
\gend=
\gbeg{1}{4}
\got{1}{F} \gnl
\gcu{1} \gnl
\gu{1} \gnl
\gob{1}{B} 
\gend
$. While on one hand, these come from the comonadic structures, on the other hand, the cases for which we will prove to be wreaths it will happen that either $\sigma$ is trivial 
(thus $\mu_M$ is canonical), or that 
$\gbeg{2}{1}
\glm \gnl 
\gend$ is trivial. In both cases we will have the form of $\mu_M$ as in the above Definition. 
\end{rem}

\medskip

Similarly as in \coref{Hd conseq}, we have:

\begin{cor} \colabel{monad Hd conseq}
In a monad Hopf datum the following identity holds true:
$$
\gbeg{2}{4}
\got{3}{F} \gnl
\gu{1} \gcl{1} \gnl
\glmptb \gnot{\hspace{-0,34cm}\psi} \grmptb \gnl
\gob{1}{F} \gob{1}{B}
\gend=
\gbeg{2}{4}
\got{1}{F} \gnl
\gcl{1} \gu{1} \gnl
\gcl{1} \gcl{1} \gnl
\gob{1}{F} \gob{1}{B}
\gend
$$
\end{cor}

A {\em comonad Hopf datum} is defined dually, using ``$\pi$-symmetry''. 

\medskip

\begin{rem} \rmlabel{wreath datum}
We will also want to name the direct consequences of a wreath. Assume that $B$ is a monad. The data in point 4 in the above Definition, excluding the identity \equref{epsilon to sigma}, 
we will call a {\em wreath datum}, assuming that all the appearing 2-cells exist. 
\end{rem}

\medskip

We will also need:

\begin{defn}
We say that $\tau_{B,F}$ is natural with respect to 
$\gbeg{2}{3}
\got{1}{B} \got{1}{F} \gnl
\grmo \gcl{1} \gnl 
\gob{1}{B} \gnl
\gend$,  
$\gbeg{2}{3}
\got{1}{B} \got{1}{F} \gnl
\glm \gnl 
\gob{3}{F} \gnl
\gend$ and $\sigma$, respectively, if it holds: 
\begin{center} 
\begin{tabular}{p{4.7cm}p{0cm}p{4.7cm}p{0cm}p{4.7cm}}
\begin{equation} \eqlabel{nat rm}
\gbeg{3}{5}
\got{1}{B} \got{1}{F} \got{1}{F} \gnl
\grmo \gcl{1} \gvac{1} \gcl{1} \gnl
\glmptb \gnot{\tau_{B,F}} \gcmp \grmptb \gnl
\gcl{1} \gvac{1} \gcl{1} \gnl
\gob{1}{F} \gob{3}{B}
\gend=
\gbeg{3}{5}
\got{1}{B} \got{1}{F} \got{1}{F} \gnl
\gcl{1} \glmptb \gnot{\hspace{-0,34cm}\tau_{F,F}} \grmptb \gnl
\glmptb \gnot{\hspace{-0,34cm}\tau_{B,F}} \grmptb \gcl{1} \gnl
\gcl{1} \grmo \gcl{1}  \gnl
\gob{1}{F} \gob{1}{B}
\gend
\end{equation} & & 
\begin{equation} \eqlabel{nat lm}
\gbeg{3}{5}
\got{1}{B} \got{1}{B} \got{1}{F} \gnl
\gcl{1} \glm \gnl
\glmptb \gnot{\tau_{B,F}} \gcmp \grmptb \gnl
\gcl{1} \gvac{1} \gcl{1} \gnl
\gob{1}{F} \gob{3}{B}
\gend=
\gbeg{3}{5}
\got{1}{B} \got{1}{F} \got{1}{F} \gnl
\glmptb \gnot{\hspace{-0,34cm}\tau_{B,B}} \grmptb \gcl{1} \gnl
\gcl{1} \glmptb \gnot{\hspace{-0,34cm}\tau_{B,F}} \grmptb \gnl
\glm \gcl{1} \gnl
\gvac{1} \gob{1}{F} \gob{1}{B}
\gend
\end{equation} & & 
\begin{equation} \eqlabel{nat sigma}
\gbeg{3}{5}
\got{1}{F} \got{1}{F} \got{1}{F} \gnl
\gcl{1} \glmptb \gnot{\hspace{-0,34cm}\tau_{F,F}} \grmptb \gnl
\glmptb \gnot{\hspace{-0,34cm}\tau_{F,F}} \grmptb \gcl{1} \gnl
\gcl{1} \glmpt \gnot{\hspace{-0,34cm}\sigma} \grmptb \gnl
\gob{1}{F} \gob{3}{B}
\gend=
\gbeg{3}{5}
\got{1}{F} \got{1}{F} \got{1}{F} \gnl
\glmpt \gnot{\hspace{-0,34cm}\sigma} \grmptb \gcl{1} \gnl
\gvac{1} \glmptb \gnot{\hspace{-0,34cm}\tau_{B,F}} \grmptb \gnl
\gvac{1} \gcl{1} \gcl{1} \gnl
\gvac{1} \gob{1}{F} \gob{1}{B}
\gend
\end{equation}
\end{tabular}
\end{center}
respectively. Substituting $\tau_{B,F}$ in the above three diagrams by $\tau_{B,B}, \tau_{F,F}$ and $\tau_{B,B}$, respectively, we obtain the notions of naturality for the pairs 
$(\tau_{B,B}, 
\gbeg{2}{1}
\grmo \gcl{1} \gnl 
\gend), 
 (\tau_{F,F}, 
\gbeg{2}{1}
\glm \gnl 
\gend)$ and $(\tau_{B,B}, \sigma)$, respectively, in the obvious way.

We say that $\tau_{F,B}$ is natural with respect to 
$\gbeg{2}{3}
\got{3}{F} \gnl
\grmo \gvac{1} \gcl{1} \gnl 
\gob{1}{B} \gob{1}{F} \gnl
\gend$,  
$\gbeg{2}{3}
\got{1}{B} \gnl
\grcm \gnl 
\gob{1}{B} \gob{1}{F} \gnl
\gend$ and $\rho'$, respectively, if it holds:  
\begin{center} 
\begin{tabular}{p{4.7cm}p{0cm}p{4.7cm}p{0cm}p{4.7cm}}
\begin{equation} \eqlabel{nat lcm}
\gbeg{3}{5}
\got{1}{F} \got{3}{B} \gnl
\gcl{1} \gvac{1} \gcl{1} \gnl
\glmptb \gnot{\tau_{F,B}} \gcmp \grmptb \gnl
\gcl{1} \grmo \gvac{1} \gcl{1} \gnl
\gob{1}{B} \gob{1}{B} \gob{1}{F} 
\gend=
\gbeg{3}{5}
\got{1}{} \got{1}{F} \got{1}{B} \gnl
\grmo \gvac{1} \gcl{1} \gcl{1}\gnl 
\gcl{1} \glmptb \gnot{\hspace{-0,34cm}\tau_{F,B}} \grmptb \gnl
\glmptb \gnot{\hspace{-0,34cm}\tau_{B,B}} \grmptb \gcl{1} \gnl
\gob{1}{B} \gob{1}{B} \gob{1}{F} 
\gend
\end{equation} & & 
\begin{equation} \eqlabel{nat rcm}
\gbeg{3}{5}
\got{1}{F} \got{3}{B} \gnl
\gcl{1} \gvac{1} \gcl{1} \gnl
\glmptb \gnot{\tau_{F,B}} \gcmp \grmptb \gnl
\grcm \gcl{1} \gnl
\gob{1}{B} \gob{1}{F} \gob{1}{F} 
\gend=
\gbeg{3}{5}
\got{1}{F} \got{1}{B} \gnl
\gcl{1} \grcm \gnl
\glmptb \gnot{\hspace{-0,34cm}\tau_{F,B}} \grmptb \gcl{1} \gnl
\gcl{1} \glmptb \gnot{\hspace{-0,34cm}\tau_{F,F}} \grmptb \gnl
\gob{1}{B} \gob{1}{F} \gob{1}{F} 
\gend
\end{equation} & & 
\begin{equation} \eqlabel{nat rho'}
\gbeg{3}{5}
\got{1}{F} \got{3}{B} \gnl
\glmptb \gnot{\hspace{-0,34cm}\rho'} \grmpb \gcl{1} \gnl
\gcl{1} \glmptb \gnot{\hspace{-0,34cm}\tau_{B,B}} \grmptb \gnl
\glmptb \gnot{\hspace{-0,34cm}\tau_{B,B}} \grmptb \gcl{1} \gnl
\gob{1}{B} \gob{1}{B} \gob{1}{B}
\gend=
\gbeg{3}{5}
\got{1}{F} \got{1}{B} \gnl
\gcl{1} \gcl{1} \gnl
\glmptb \gnot{\hspace{-0,34cm}\tau_{F, B}} \grmptb \gnl
\gcl{1} \glmptb \gnot{\hspace{-0,34cm}\rho'} \grmpb \gnl
\gob{1}{B} \gob{1}{B} \gob{1}{B}
\gend
\end{equation}
\end{tabular}
\end{center}
respectively. Substituting $\tau_{F,B}$ in the above three diagrams by $\tau_{F,F}, \tau_{B,B}$ and $\tau_{F,F}$, respectively, we obtain the notions of naturality for the pairs 
$(\tau_{F,F}, 
\gbeg{2}{1}
\grmo \gvac{1} \gcl{1}  \gnl 
\gend), 
 (\tau_{B,B}, 
\gbeg{2}{1}
\grcm \gnl 
\gend)$ and $(\tau_{F,F}, \rho')$, respectively, in the obvious way. 
\end{defn}

\begin{defn}
We say that $(\tau_{F,B}, \tau_{F,F}, \tau_{F,B}), (\tau_{B,F}, \tau_{B,B}, \tau_{F,B}), 
(\tau_{B,F}, \tau_{F,F}), (\tau_{F,B}, \tau_{B,B}), \tau_{B,B}$ and $\tau_{F,F}$ satisfy the Yang-Baxter equations if 
\begin{center} 
\begin{tabular}{p{6cm}p{0cm}p{6cm}}
\begin{equation} \eqlabel{YBE FBF}
\gbeg{3}{5}
\got{1}{F} \got{1}{B} \got{1}{F} \gnl
\gcl{1} \glmptb \gnot{\hspace{-0,34cm}\tau_{B,F}} \grmptb \gnl
\glmptb \gnot{\hspace{-0,34cm}\tau_{F,F}} \grmptb \gcl{1} \gnl
\gcl{1} \glmptb \gnot{\hspace{-0,34cm}\tau_{F,B}} \grmptb \gnl
\gob{1}{F} \gob{1}{B} \gob{1}{F}
\gend=
\gbeg{3}{5}
\got{1}{F} \got{1}{B} \got{1}{F} \gnl
\glmptb \gnot{\hspace{-0,34cm}\tau_{F,B}} \grmptb \gcl{1} \gnl
\gcl{1} \glmptb \gnot{\hspace{-0,34cm}\tau_{F,F}} \grmptb \gnl
\glmptb \gnot{\hspace{-0,34cm}\tau_{B,F}} \grmptb \gcl{1} \gnl
\gob{1}{F} \gob{1}{B} \gob{1}{F}
\gend
\end{equation} & & 
\begin{equation} \eqlabel{YBE BFB}
\gbeg{3}{5}
\got{1}{B} \got{1}{F} \got{1}{B} \gnl
\glmptb \gnot{\hspace{-0,34cm}\tau_{B,F}} \grmptb \gcl{1} \gnl
\gcl{1} \glmptb \gnot{\hspace{-0,34cm}\tau_{B,B}} \grmptb \gnl
\glmptb \gnot{\hspace{-0,34cm}\tau_{F,B}} \grmptb \gcl{1} \gnl
\gob{1}{B} \gob{1}{F} \gob{1}{B} 
\gend=
\gbeg{3}{5}
\got{1}{B} \got{1}{F} \got{1}{B} \gnl
\gcl{1} \glmptb \gnot{\hspace{-0,34cm}\tau_{F,B}} \grmptb \gnl
\glmptb \gnot{\hspace{-0,34cm}\tau_{B,B}} \grmptb \gcl{1} \gnl
\gcl{1} \glmptb \gnot{\hspace{-0,34cm}\tau_{B,F}} \grmptb \gnl
\gob{1}{B} \gob{1}{F} \gob{1}{B} 
\gend
\end{equation}
\end{tabular}
\end{center}

\begin{center} 
\begin{tabular}{p{6cm}p{0cm}p{6cm}}
\begin{equation} \eqlabel{YBE1}
\gbeg{3}{5}
\got{1}{B} \got{1}{F} \got{1}{F} \gnl
\gcl{1} \glmptb \gnot{\hspace{-0,34cm}\tau_{F,F}} \grmptb \gnl
\glmptb \gnot{\hspace{-0,34cm}\tau_{B,F}} \grmptb \gcl{1} \gnl
\gcl{1} \glmptb \gnot{\hspace{-0,34cm}\tau_{B,F}} \grmptb \gnl
\gob{1}{F} \gob{1}{F} \gob{1}{B}
\gend=
\gbeg{3}{5}
\got{1}{B} \got{1}{F} \got{1}{F} \gnl
\glmptb \gnot{\hspace{-0,34cm}\tau_{B,F}} \grmptb \gcl{1} \gnl
\gcl{1} \glmptb \gnot{\hspace{-0,34cm}\tau_{B,F}} \grmptb \gnl
\glmptb \gnot{\hspace{-0,34cm}\tau_{F,F}} \grmptb \gcl{1} \gnl
\gob{1}{F} \gob{1}{F} \gob{1}{B}
\gend
\end{equation} & & 
\begin{equation} \eqlabel{YBE2}
\gbeg{3}{5}
\got{1}{F} \got{1}{B} \got{1}{B} \gnl
\glmptb \gnot{\hspace{-0,34cm}\tau_{F,B}} \grmptb \gcl{1} \gnl
\gcl{1} \glmptb \gnot{\hspace{-0,34cm}\tau_{F,B}} \grmptb \gnl
\glmptb \gnot{\hspace{-0,34cm}\tau_{B,B}} \grmptb \gcl{1} \gnl
\gob{1}{B} \gob{1}{B} \gob{1}{F} 
\gend=
\gbeg{3}{5}
\got{1}{F} \got{1}{B} \got{1}{B} \gnl
\gcl{1} \glmptb \gnot{\hspace{-0,34cm}\tau_{B,B}} \grmptb \gnl
\glmptb \gnot{\hspace{-0,34cm}\tau_{F,B}} \grmptb \gcl{1} \gnl
\gcl{1} \glmptb \gnot{\hspace{-0,34cm}\tau_{F,B}} \grmptb \gnl
\gob{1}{B} \gob{1}{B} \gob{1}{F} 
\gend
\end{equation}
\end{tabular}
\end{center}

\begin{center} 
\begin{tabular}{p{6cm}p{0cm}p{6cm}}
\begin{equation} \eqlabel{YBE BB}
\gbeg{3}{5}
\got{1}{B} \got{1}{B} \got{1}{B} \gnl
\gcl{1} \glmptb \gnot{\hspace{-0,34cm}\tau_{B,B}} \grmptb \gnl
\glmptb \gnot{\hspace{-0,34cm}\tau_{B,B}} \grmptb \gcl{1} \gnl
\gcl{1} \glmptb \gnot{\hspace{-0,34cm}\tau_{B,B}} \grmptb \gnl
\gob{1}{B} \gob{1}{B} \gob{1}{B}
\gend=
\gbeg{3}{5}
\got{1}{B} \got{1}{B} \got{1}{B} \gnl
\glmptb \gnot{\hspace{-0,34cm}\tau_{B,B}} \grmptb \gcl{1} \gnl
\gcl{1} \glmptb \gnot{\hspace{-0,34cm}\tau_{B,B}} \grmptb \gnl
\glmptb \gnot{\hspace{-0,34cm}\tau_{B,B}} \grmptb \gcl{1} \gnl
\gob{1}{B} \gob{1}{B} \gob{1}{B}
\gend
\end{equation} & & 
\begin{equation} \eqlabel{YBE FF}
\gbeg{3}{5}
\got{1}{F} \got{1}{F} \got{1}{F} \gnl
\glmptb \gnot{\hspace{-0,34cm}\tau_{F,F}} \grmptb \gcl{1} \gnl
\gcl{1} \glmptb \gnot{\hspace{-0,34cm}\tau_{F,F}} \grmptb \gnl
\glmptb \gnot{\hspace{-0,34cm}\tau_{F,F}} \grmptb \gcl{1} \gnl
\gob{1}{F} \gob{1}{F} \gob{1}{F} 
\gend=
\gbeg{3}{5}
\got{1}{F} \got{1}{F} \got{1}{F} \gnl
\gcl{1} \glmptb \gnot{\hspace{-0,34cm}\tau_{F,F}} \grmptb \gnl
\glmptb \gnot{\hspace{-0,34cm}\tau_{F,F}} \grmptb \gcl{1} \gnl
\gcl{1} \glmptb \gnot{\hspace{-0,34cm}\tau_{F,F}} \grmptb \gnl
\gob{1}{F} \gob{1}{F} \gob{1}{F} 
\gend
\end{equation} 
\end{tabular}
\end{center}
hold, respectively. 
\end{defn}

Observe that assertions for $\phi'$ and $\rho'$ (and their proofs) are $\pi$-symmetric to those for $\psi$ and $\sigma$. We find:

\begin{prop}  \prlabel{simplest case}
The simplest monad Hopf datum, given by $\psi=\tau_{B,F}$ and trivial $\sigma$ (the actions \equref{module cells} are trivial), 
consists of monads $B$ and $F$. It determines a canonical wreath.  
\end{prop}

\begin{proof}
In this case, all the identities in the wreath datum, but the two ones claiming the monad laws for $F$, hold trivially. The proof of the rest is direct: 
every pair of monads delivers a canonical wreath.  
\qed\end{proof}

As we saw, wreaths define associative wreath products whose product $\nabla_{FB}$ and unit $\eta_{FB}$ are given via \equref{wreath (co)product - old}. 
It is directly proved that when $\mu_M$ is canonical and $\nabla_{FB}$ and $\eta_{FB}$ define a monad, then $\psi$ is a left and right monadic distributive law. 
Then for $\K=\hat\C$, where $\C$ is a monoidal category 
we recover the following well-known fact:

\begin{cor}
Given algebras $B,F$ and a morphism $\tau_{B,F}: B\ot F\to F\ot B$ in a monoidal category $\C$, 
the tensor product $F\ot B$ has an asssociative product defined by $(\nabla_B\ot\nabla_F)(B\ot\tau_{B,F}\ot F)$, where  
$\nabla_B, \nabla_F$ are multiplications of $B,F$, respectively, if and only if $\tau_{B,F}$ is a left and right monadic distributive law in $\C$. 
\end{cor}

\begin{prop} \prlabel{triv psi, nontriv sigma}
The monad Hopf datum given by $\psi=\tau_{B,F}$ and a non-trivial $\sigma$ consists of the following 2-cocycle: 
$$
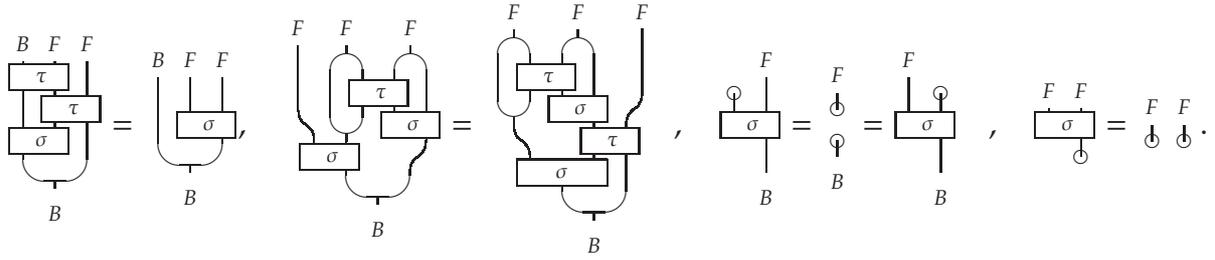
\gbeg{3}{6}
\got{1}{B} \got{1}{F} \got{1}{F} \gnl
\glmptb \gnot{\hspace{-0,34cm}\tau} \grmptb \gcl{1} \gnl
\gcl{1} \glmptb \gnot{\hspace{-0,34cm}\tau} \grmptb \gnl
\glmptb \gnot{\hspace{-0,34cm}\sigma} \grmpt \gcl{1} \gnl
\gwmu{3} \gnl
\gob{3}{B}
\gend=
\gbeg{3}{5}
\got{1}{B} \got{1}{F} \got{1}{F}\gnl
\gcl{1} \gcl{1} \gcl{1} \gnl
\gcl{1} \glmpt \gnot{\hspace{-0,34cm}\sigma} \grmptb  \gnl
\gwmu{3} \gnl
\gob{3}{B}
\gend,
\hspace{0,4cm}
\gbeg{5}{7}
\got{1}{F} \got{2}{F} \got{2}{F}\gnl
\gcl{1} \gcmu \gcmu  \gnl
\gcl{1} \gcl{1} \glmptb \gnot{\hspace{-0,34cm}\tau} \grmptb \gcl{1} \gnl
\gcn{1}{1}{1}{2} \gmu \glmpt \gnot{\hspace{-0,34cm}\sigma} \grmptb  \gnl
\gvac{1} \hspace{-0,22cm} \glmpt \gnot{\hspace{-0,34cm}\sigma} \grmptb \gvac{1} \gcn{1}{1}{2}{1} \gnl
\gvac{2} \gwmu{3} \gnl
\gob{7}{B}
\gend=
\gbeg{6}{8}
\got{2}{F} \got{2}{F} \got{2}{F} \gnl
\gcmu \gcmu \gcn{1}{2}{2}{2} \gnl
\gcl{1} \glmptb \gnot{\hspace{-0,34cm}\tau} \grmptb \gcl{1} \gnl
\gmu \glmpt \gnot{\hspace{-0,34cm}\sigma} \grmptb \gcn{1}{1}{2}{1} \gnl
\gcn{1}{1}{2}{3} \gvac{2} \glmptb \gnot{\hspace{-0,34cm}\tau} \grmptb \gnl
\gvac{1} \glmpt \gnot{\sigma} \gcmpb \grmpt \gcl{1} \gnl
\gvac{2} \gwmu{3} \gnl
\gob{7}{B}
\gend,
\hspace{0,4cm}
\gbeg{2}{5}
\got{3}{F}\gnl
\gu{1} \gcl{1} \gnl
\glmpt \gnot{\hspace{-0,34cm}\sigma} \grmptb \gnl
\gvac{1} \gcl{1} \gnl
\gob{3}{B}
\gend=
\gbeg{1}{4}
\got{1}{F}\gnl
\gcu{1} \gnl
\gu{1} \gnl
\gob{1}{B}
\gend=
\gbeg{3}{5}
\got{1}{F}\gnl
\gcl{1} \gu{1} \gnl
\glmpt \gnot{\hspace{-0,34cm}\sigma} \grmptb \gnl
\gvac{1} \gcl{1} \gnl
\gob{3}{B}
\gend,
\hspace{0,4cm}
\gbeg{2}{3}
\got{1}{F} \got{1}{F}\gnl
\glmpt \gnot{\hspace{-0,34cm}\sigma} \grmptb \gnl
\gvac{1} \gcu{1} \gnl
\gend=
\gbeg{2}{2}
\got{1}{F} \got{1}{F}\gnl
\gcu{1} \gcu{1} \gnl
\gend.
$$
If moreover $F$ is a comonad and \equref{YBE1}, \equref{YBE FF} and \equref{nat sigma} hold, then this monad Hopf datum is a wreath. 
\end{prop}

\begin{proof}
The first statement is clear the rest of the identities in a monad Hopf datum are trivial in this setting. The second assertion is a particular case of the 
analogous claim for $\psi_2$ in \prref{psi_1,2}, which will be proved below. 
\qed\end{proof}

\begin{ex} \exlabel{group coh}
When $\K=\hat\C$ where $\C$ is the monoidal (sub)category of trivial modules $M=B$ over a group $G=F$, the above recovers 2-cocycles from the group cohomology. The 2-cell $\tau$ is the flip 
in the category of abelian groups and the first condition above reflects the abelianity of the module $M=B$, as an additive group. The second one is the cocycle condition, and the 
third one states that $\sigma$ is normalized.  
\end{ex}

\begin{prop} \prlabel{psi_1,2}
\begin{enumerate}
\item Set $\psi_1=
\gbeg{3}{5}
\got{2}{B} \got{1}{F} \gnl
\gcmu \gcl{1} \gnl
\gcl{1}  \glmptb \gnot{\hspace{-0,34cm}\tau_{B,F}} \grmptb \gnl
\glm \gcl{1} \gnl
\gob{1}{} \gob{1}{F} \gob{1}{B} 
\gend$ (that is, 
$\gbeg{2}{3}
\got{1}{B} \got{1}{F} \gnl
\grmo \gcl{1} \gnl 
\gob{1}{B}
\gend$ is trivial). If $(B,F)$ is a monad Hopf datum, $\sigma$ is trivial and \equref{proj B-bialg} and \equref{nat lm} hold, 
then $(F, \psi_1, \mu_M \hspace{0,2cm}\textnormal{canonical})$ is a left wreath around $B$.

\item Set $\psi_2=
\gbeg{3}{5}
\got{1}{B} \got{2}{F} \gnl
\gcl{1} \gcmu \gnl
\glmptb \gnot{\hspace{-0,34cm}\tau_{B,F}} \grmptb \gcl{1} \gnl
\gcl{1} \grmo \gcl{1} \gnl
\gob{1}{F} \gob{1}{B} 
\gend$ (that is, $\gbeg{2}{3}
\got{1}{B} \got{1}{F} \gnl
\glm \gnl 
\gob{3}{F}
\gend$ is trivial). If $(B,F)$ is a monad Hopf datum, $F$ is a comonad, \equref{proj F-bialg}, \equref{YBE1}, \equref{nat rm}, \equref{YBE FF} and \equref{nat sigma} hold, 
then $(F, \psi_2, \mu_M)$ is a left wreath around $B$.

\item Set $\phi_1'= 
\gbeg{3}{5}
\got{1}{F} \got{1}{B} \gnl
\gcl{1} \grcm \gnl
\glmptb \gnot{\hspace{-0,34cm}\tau_{F,B}} \grmptb \gcl{1} \gnl
\gcl{1} \gmu \gnl
\gob{1}{B} \gob{2}{F} 
\gend$ (that is, $\gbeg{2}{3}
\got{3}{F} \gnl
\grmo \gvac{1} \gcl{1} \gnl 
\gob{1}{B} \gob{1}{F}
\gend$ is trivial). If $(B,F)$ is a comonad Hopf datum, $\rho'$ is trivial and \equref{proj F-bialg} and \equref{nat rcm} hold, 
then $(B, \phi'_1, \Delta_C' \hspace{0,2cm}\textnormal{canonical})$ is a right cowreath around $F$.

\item Set $\phi'_2= 
\gbeg{3}{5}
\got{1}{} \got{1}{F} \got{1}{B} \gnl
\grmo \gvac{1} \gcl{1} \gcl{1} \gnl
\gcl{1} \glmptb \gnot{\hspace{-0,34cm}\tau_{F,B}} \grmptb \gnl
\gmu \gcl{1} \gnl
\gob{2}{B} \gob{1}{F} 
\gend$ (that is, $\gbeg{2}{3}
\got{1}{B} \gnl
\grcm \gnl  
\gob{1}{B} \gob{1}{F}
\gend$ is trivial). If $(B,F)$ is a comonad Hopf datum, $B$ is a monad, \equref{proj B-bialg}, \equref{YBE2}, \equref{nat lcm}, \equref{YBE BB} and \equref{nat rho'} hold, 
then $(B, \phi'_2, \Delta_C')$ is a right cowreath around $F$. 

\end{enumerate}
\end{prop}

\begin{proof}
The parts 3 and 4 are $\pi$-symmetric to 1 and 2, respectively. The statement for $\psi_1$ is straightforward to prove, as well as that $\psi_2$ satisfies the $\psi$ axioms. We show only that 
\equref{proj F-bialg}, \equref{weak action}--\equref{weak action unity}, \equref{YBE1}, \equref{nat rm} imply \equref{2-cell mu_M}. The proof that the identities 
\equref{2-cocycle condition}, \equref{weak assoc. mu_M}, \equref{YBE FF}, \equref{nat sigma} imply \equref{monad law mu_M} is similar. The necessary identities 
for the (co)unities are proved straightforwardly. (Observe that in this setting the identities \equref{mod alg}--\equref{mod alg unity} are trivial and that 
\equref{weak assoc. mu_M}--\equref{weak unity eta_M} mean that $F$ is a monad. Consequently, in this setting $F$ is a left bimonad in $\K$.) 

We start with the identity \equref{weak action}, compose it from the left with $FF$, then from above with \equref{precompose} and from below with $\gbeg{3}{3}
\got{1}{F} \got{1}{F} \got{1}{B} \gnl
\gmu \gcl{1} \gnl
\gob{2}{F} \gob{1}{B} \gnl 
\gend$. By coassociativity of $F$ and distributive law of $\tau_{F,F}$ with respect to the comultiplication of $F$ this yields \equref{L=R 1}: \vspace{1,8cm}
\begin{center} \hspace{-1,3cm} 
\begin{tabular}{p{6cm}p{0cm}p{9cm}}
\begin{equation} \eqlabel{precompose}
\gbeg{3}{5}
\got{1}{B} \got{2}{F} \got{2}{F} \gnl
\gcl{1} \gcmu \gcmu \gnl
\glmptb \gnot{\hspace{-0,34cm}\tau_{B,F}} \grmptb \glmptb \gnot{\hspace{-0,34cm}\tau_{F,F}} \grmptb \gcl{1} \gnl
\gcl{1} \glmptb \gnot{\hspace{-0,34cm}\tau_{B,F}} \grmptb \gcl{1} \gcl{1} \gnl 
\gob{1}{F} \gob{1}{F} \gob{1}{B} \gob{1}{F} \gob{1}{F} \gnl
\gend
\end{equation} & & \vspace{-1,8cm}
\begin{equation} \eqlabel{L=R 1}
\gbeg{7}{13}
\got{1}{B} \got{4}{F} \got{2}{F} \gnl
\gcl{2} \gvac{1} \hspace{-0,34cm} \gwcm{3} \gvac{1} \hspace{-0,34cm} \gcmu \gnl
\gvac{2} \gcmu \hspace{-0,2cm} \gvac{1} \gcl{1} \gcmu \gcn{1}{6}{0}{0} \gnl %
\gvac{2} \hspace{-0,3cm} \glmptb \gnot{\hspace{-0,34cm}\tau_{B,F}} \grmptb \gcn{1}{1}{1}{2} \gvac{1} \hspace{-0,34cm} \glmptb \gnot{\hspace{-0,34cm}\tau_{F,F}} \grmptb \gcl{4} \gnl
\gvac{3} \hspace{-0,3cm} \gcl{1} \gcn{1}{1}{1}{2} \gvac{1} \hspace{-0,34cm} \glmptb \gnot{\hspace{-0,34cm}\tau_{F,F}} \grmptb \gcl{1} \gnl
\gvac{3} \gcn{2}{1}{2}{3} \glmptb \gnot{\hspace{-0,34cm}\tau_{B,F}} \grmptb \gcl{1} \gcl{1} \gnl  
\gvac{4} \gmu \glmptb \gnot{\hspace{-0,34cm}\tau_{B,F}} \grmptb \gcl{1} \gnl
\gvac{4} \gcn{2}{5}{2}{2} \gcl{1} \grmo \gcl{1} \gcn{1}{1}{3}{1} \gnl
\gvac{6} \gcl{1} \glmptb \gnot{\hspace{-0,34cm}\tau_{B,F}} \grmptb \gcn{1}{1}{2}{1} \gnl
\gvac{6} \gcl{1} \gcl{1} \grmo \gcl{1} \gnl
\gvac{6} \glmpt \gnot{\hspace{-0,34cm}\sigma} \grmptb \gcl{1} \gnl
\gvac{7} \gmu \gnl
%
\gob{10}{B} \gob{-4}{F}
\gend= 
\gbeg{7}{10}
\got{1}{B} \got{4}{F} \got{2}{F} \gnl
\gcl{2} \gvac{1} \hspace{-0,34cm} \gwcm{3} \gvac{1} \hspace{-0,34cm} \gcmu \gnl
\gvac{2} \gcmu \hspace{-0,2cm} \gvac{1} \gcl{1} \gcmu \gcn{1}{1}{0}{0} \gnl %
\gvac{2} \hspace{-0,3cm} \glmptb \gnot{\hspace{-0,34cm}\tau_{B,F}} \grmptb \gcn{1}{1}{1}{2} \gvac{1} \hspace{-0,34cm} \glmptb \gnot{\hspace{-0,34cm}\tau_{F,F}} \grmptb \gcl{1} \gcn{1}{1}{0}{1} \gnl
\gvac{3} \hspace{-0,3cm} \gcl{1} \gcn{1}{1}{1}{2} \gvac{1} \hspace{-0,34cm} \glmptb \gnot{\hspace{-0,34cm}\tau_{F,F}} \grmptb \glmptb \gnot{\hspace{-0,34cm}\tau_{F,F}} \grmptb \gcl{1}  \gnl
\gvac{3} \gcn{2}{1}{2}{3} \glmptb \gnot{\hspace{-0,34cm}\tau_{B,F}} \grmptb \gmu \glmpt \gnot{\hspace{-0,34cm}\sigma} \grmptb \gnl  
\gvac{4} \gmu \gcl{1} \gcn{1}{1}{2}{1} \gvac{2} \gcl{2} \gnl
\gvac{5} \gcn{1}{2}{0}{0} \grmo \gcl{1} \gnl
\gvac{6} \gwmu{5} \gnl
%
\gvac{4} \gob{2}{B} \gvac{2} \gob{1}{F}
\gend
\end{equation}
\end{tabular}
\end{center}
Now by the Yang-Baxter equation \equref{YBE1} the left side of \equref{L=R 1} equals: 
$$
\gbeg{7}{11}
\got{1}{B} \got{3}{F} \got{3}{F} \gnl
\gcl{2} \gvac{1} \hspace{-0,44cm} \gwcm{3} \gvac{1} \hspace{-0,34cm} \gcmu \gnl
\gvac{2} \gcmu \hspace{-0,2cm} \gvac{1} \gcl{1} \gcmu \gcn{1}{5}{0}{0} \gnl %
\gvac{2} \hspace{-0,3cm} \glmptb \gnot{\hspace{-0,34cm}\tau_{B,F}} \grmptb \gcl{1} \gvac{1} \hspace{-0,34cm} \glmptb \gnot{\hspace{-0,34cm}\tau_{F,F}} \grmptb \gcl{3} \gnl
\gvac{3} \hspace{-0,3cm} \gcl{3} \glmptb \gnot{\hspace{-0,34cm}\tau_{B,F}} \grmptb \gcn{1}{1}{2}{1} \gcn{1}{1}{2}{1} \gnl
\gvac{4} \gcl{1} \glmptb \gnot{\hspace{-0,34cm}\tau_{B,F}} \grmptb \gcl{1} \gnl  
\gvac{4} \glmptb \gnot{\hspace{-0,34cm}\tau_{F,F}} \grmptb \grmo \gcl{1} \gcn{1}{1}{4}{1} \gnl
\gvac{3} \gmu \gcl{1} \glmptb \gnot{\hspace{-0,34cm}\tau_{B,F}} \grmptb \gcn{1}{1}{3}{1} \gnl
\gvac{4} \gcn{1}{2}{0}{0} \glmpt \gnot{\hspace{-0,34cm}\sigma} \grmptb  \grmo \gcl{1} \gnl
\gvac{6} \gmu \gnl
\gob{8}{B} \gob{-2}{F}
\gend\stackrel{\tau}{=}
\gbeg{7}{11}
\gvac{1} \got{1}{B} \got{2}{F} \got{4}{F} \gnl
\gvac{1} \gcl{1} \gcmu \gvac{1}  \gcmu \gnl 
\gvac{1} \glmptb \gnot{\hspace{-0,34cm}\tau_{B,F}} \grmptb \hspace{-0,2cm} \gvac{1} \gcn{1}{1}{0}{1} \gcmu \gcn{1}{5}{0}{0} \gnl %
\gvac{1} \gcmu \gcn{1}{1}{0}{1} \glmptb \gnot{\hspace{-0,34cm}\tau_{F,F}} \grmptb \gcl{3} \gnl
\gvac{1} \gcl{1} \gcl{1} \glmptb \gnot{\hspace{-0,34cm}\tau_{B,F}} \grmptb \gcl{1} \gnl
\gvac{1} \gcl{1} \glmptb \gnot{\hspace{-0,34cm}\tau_{F,F}} \grmptb \grmo \gcl{1} \gnl  
\gvac{1} \gmu \gcl{2} \gcl{1} \gcn{1}{1}{3}{1} \gnl
\gvac{1} \gcn{1}{1}{2}{2} \gvac{2} \glmptb \gnot{\hspace{-0,34cm}\tau_{B,F}} \grmptb \gcn{1}{1}{2}{1} \gnl
\gvac{2} \gcn{1}{2}{0}{0} \glmpt \gnot{\hspace{-0,34cm}\sigma} \grmptb  \grmo \gcl{1} \gnl
\gvac{4} \gmu \gnl
\gob{4}{B} \gob{2}{F}
\gend\stackrel{\equref{nat rm}}{=}
\gbeg{7}{9}
\gvac{1} \got{1}{B} \got{2}{F} \got{3}{F} \gnl
\gvac{1} \gcl{1} \gcmu \gvac{1} \hspace{-0,22cm} \gcmu \gnl 
\gvac{2} \hspace{-0,22cm} \glmptb \gnot{\hspace{-0,34cm}\tau_{B,F}} \grmptb \gcl{1} \gcmu \gcn{1}{1}{0}{1} \gnl 
\gvac{2} \gcn{1}{1}{1}{0} \grmo \gcl{1} \gcn{1}{1}{3}{1} \gcn{1}{2}{3}{1} \gvac{1} \gcl{1} \gnl 
\gvac{1} \gcmu \glmptb \gnot{\hspace{-0,34cm}\tau_{B,F}} \grmptb \gvac{1}  \gcn{1}{2}{3}{1} \gnl
\gvac{1} \gcl{1} \glmptb \gnot{\hspace{-0,34cm}\tau_{F,F}} \grmptb  \glmptb \gnot{\hspace{-0,34cm}\tau_{B,F}} \grmptb  \gnl
\gvac{1} \gmu \glmpt \gnot{\hspace{-0,34cm}\sigma} \grmptb  \grmo \gcl{1} \gnl
\gvac{1} \gcn{1}{1}{2}{2} \gvac{2} \gmu \gnl
\gob{4}{B} \gob{2}{F}
\gend\stackrel{\tau}{=}
\gbeg{7}{12}
\gvac{2} \got{1}{B} \got{2}{F} \got{1}{F} \gnl
\gvac{2} \gcl{1} \gcmu \gcl{2} \gnl
\gvac{2} \glmptb \gnot{\hspace{-0,34cm}\tau_{B,F}} \grmptb\gcl{1} \gnl
\gvac{2} \gcl{1} \grmo \gcl{1} \gcn{1}{1}{3}{2} \gnl
\gvac{2} \gcl{1} \gcl{1} \gcmu \gnl
\gvac{2} \gcn{1}{2}{1}{-2}  \glmptb \gnot{\hspace{-0,34cm}\tau_{B,F}} \grmptb \gcl{1} \gnl
\gvac{3} \gcn{1}{1}{1}{0} \grmo \gcl{1} \gnl
\gcmu \gcmu \gcl{3} \gnl
\gcl{1} \glmptb \gnot{\hspace{-0,34cm}\tau_{F,F}} \grmptb \gcl{1} \gnl
\gmu \glmpt \gnot{\hspace{-0,34cm}\sigma} \grmptb \gnl
\gcn{1}{1}{2}{2} \gvac{2} \gmu \gnl
\gob{2}{B}  \gob{4}{F}
\gend
$$
which is the left hand-side of \equref{2-cell mu_M} in the present setting. In the first and the third equation we applied the distributive law of $\tau_{B,F}$ with respect 
to the comultiplication of $F$. The right hand-side of \equref{L=R 1}, by the distributivity law of $\tau_{B,F}$ with respect to the multiplication of $F$, equals: 
$$
\gbeg{8}{10}
\got{1}{B} \got{4}{F} \got{2}{F} \gnl
\gcl{4} \gvac{1} \hspace{-0,34cm} \gwcm{3} \gvac{1} \hspace{-0,34cm} \gcmu \gnl
\gvac{2} \gcmu \hspace{-0,2cm} \gvac{1} \gcl{1} \gcmu \gcn{1}{1}{0}{0} \gnl %
\gvac{3} \gcn{1}{1}{0}{0} \gcn{1}{1}{0}{1} \glmptb \gnot{\hspace{-0,34cm}\tau_{F,F}} \grmptb \gcl{1} \gcn{1}{1}{0}{1} \gnl
\gvac{3} \hspace{-0,3cm} \gcn{1}{1}{1}{2} \gvac{1} \hspace{-0,34cm} \glmptb \gnot{\hspace{-0,34cm}\tau_{F,F}} \grmptb \glmptb \gnot{\hspace{-0,34cm}\tau_{F,F}} \grmptb \gcl{1}  \gnl
\gvac{2} \gcn{2}{1}{2}{4} \gmu \gmu \glmpt \gnot{\hspace{-0,34cm}\sigma} \grmptb \gnl  
\gvac{4} \hspace{-0,22cm} \glmptb \gnot{\hspace{-0,34cm}\tau_{B,F}} \grmptb \gcn{1}{1}{3}{1} \gvac{2} \gcn{1}{1}{2}{2} \gnl
\gvac{4} \gcl{2} \grmo \gcl{1} \gvac{3} \gcn{1}{1}{2}{1} \gnl
\gvac{5} \gwmu{5} \gnl
\gvac{4} \gob{1}{B} \gvac{2} \gob{1}{F}
\gend\stackrel{\tau}{=}
\gbeg{8}{10}
\got{1}{B} \got{4}{F} \got{3}{F} \gnl
\gcl{4} \gvac{1} \hspace{-0,34cm} \gwcm{3} \gvac{1} \gcmu \gnl
\gvac{2} \gcl{1} \gvac{1} \glmptb \gnot{\tau_{F,F}} \gcmp \grmptb \gcl{1} \gnl
\gvac{2} \hspace{-0,3cm} \gcmu \gcmu \gvac{1} \hspace{-0,22cm} \glmpt \gnot{\hspace{-0,34cm}\sigma} \grmptb \gnl  
\gvac{3} \gcn{1}{1}{0}{0} \hspace{-0,22cm} \glmptb \gnot{\hspace{-0,34cm}\tau_{F,F}} \grmptb \gcl{1} \gvac{1} \gcn{1}{3}{2}{2} \gnl
\gvac{2} \gcn{1}{1}{1}{2} \gmu \gmu \gnl
\gvac{3} \hspace{-0,22cm} \glmptb \gnot{\hspace{-0,34cm}\tau_{B,F}} \grmptb \gcn{1}{1}{3}{1} \gnl
\gvac{3} \gcl{2} \grmo \gcl{1} \gvac{3} \gcn{1}{1}{3}{1} \gnl
\gvac{4} \gwmu{5} \gnl
\gvac{3} \gob{1}{B} \gvac{2} \gob{1}{F}
\gend \stackrel{\equref{proj F-bialg}}{=}
\gbeg{6}{9}
\got{2}{B} \got{1}{F} \got{3}{F} \gnl
\gcn{1}{4}{2}{2} \gvac{1} \hspace{-0,34cm} \gcmu \gcmu \gnl
\gvac{2} \gcl{1} \glmptb \gnot{\hspace{-0,34cm}\tau_{F,F}} \grmptb \gcl{1} \gnl
\gvac{2} \gmu \glmpt \gnot{\hspace{-0,34cm}\sigma} \grmptb \gnl  
\gvac{2} \gcmu \gvac{1} \gcl{3} \gnl
\gvac{1} \glmptb \gnot{\hspace{-0,34cm}\tau_{B,F}} \grmptb \gcl{1} \gnl
\gvac{1} \gcl{2} \grmo \gcl{1} \gnl
\gvac{2} \gwmu{4} \gnl
\gvac{1} \gob{1}{B} \gvac{1} \gob{2}{F}
\gend 
$$
and this is precisely the right hand-side of \equref{2-cell mu_M}. In the first equation we applied the distributivity law of $\tau_{F,F}$ with respect to the comultiplication of $F$. 
\qed\end{proof}

The statement in the above Proposition in the setting of $\psi_2$ can be reformulated as follows:

\begin{prop} \prlabel{Sw iff}
Suppose that $(B,F, \psi_2, \mu_M)$ is a wreath system (equivalently, that $F$ is a monad and that 
$B$ is a right $F$-module monad, $\sigma$ is a normalized 2-cocycle and the $F$-action on $B$ is twisted by $\sigma$) and that $F$ is a left bimonad. 
Then $(B,F, \psi_2, \mu_M)$ is a wreath. 

Consequently, $(B,F, \psi_2, \mu_M)$ is a wreath if and only if it is a wreath system and $F$ is a left bimonad in $\K$. 
\end{prop}

\bigskip

(In the following Proposition we will label the 2-cell $\psi$ as $\psi_4$, as the notation $\psi_3$ we reserved for a 2-cell $\psi$ appearing in a mixed biwreath-like object in \cite{Femic5}.) 

\begin{prop} \prlabel{psi_4}
Set $\psi_4=
\gbeg{4}{5}
\got{2}{B} \got{2}{F} \gnl
\gcmu \gcmu \gnl
\gcl{1} \glmptb \gnot{\hspace{-0,34cm}\tau_{B,F}} \grmptb \gcl{1} \gnl
\glm \grmo \gcl{1} \gnl
\gob{1}{} \gob{1}{F} \gob{1}{B} 
\gend$ 
and suppose we are in a setting where the following assumptions hold true: 
\begin{enumerate}
\item $(B,F)$ is a monad Hopf datum and $F$ is a comonad; 
\item there are 2-cells 
$
\gbeg{2}{3}
\got{3}{F} \gnl
\grmo \gvac{1} \gcl{1} \gnl
\gob{1}{B} \gob{1}{F}
\gend,
\hspace{0,3cm}
\gbeg{2}{3}
\got{1}{B} \gnl
\grcm \gnl
\gob{1}{B} \gob{1}{F}
\gend,
\hspace{0,3cm}
\gbeg{3}{3}
\got{3}{F} \gnl
\glmpb \gnot{\hspace{-0,34cm}\rho'} \grmptb \gnl
\gob{1}{B}\gob{1}{B}
\gend
$ such that \equref{proj F-bialg}, \equref{proj B-bialg} and the second and third equation in \equref{1-3} hold; 
\item $B$ is a comonad (or, \equref{weak coass. Delta_C'}--\equref{weak counit Epsilon_C'} hold and 
$\gbeg{2}{3}
\got{1}{B} \gnl
\grcm \gnl 
\gob{1}{B} \gob{1}{F} \gnl
\gend$ is trivial or $\rho'$ is trivial (hence $\Delta_C'$ is canonical));
\item $\sigma$ is trivial (hence $\mu_M'$ is canonical);
\item $\tau_{B,F}$ is natural with respect to 
$\gbeg{2}{3}
\got{1}{B} \got{1}{F} \gnl
\grmo \gcl{1} \gnl 
\gob{1}{B} \gnl
\gend$ and 
$\gbeg{2}{3}
\got{1}{B} \got{1}{F} \gnl
\glm \gnl 
\gob{3}{F} \gnl
\gend$ (recall \equref{nat rm}, \equref{nat lm}).
\end{enumerate}
Then $(F, \psi_4, \mu_M \hspace{0,2cm}\textnormal{canonical})$ is a left wreath around $B$. 
\end{prop}

\begin{proof}
Let us prove that $\psi_4$ satisfies the $\psi$ axioms. Indeed: 
$$\hspace{-0,2cm}
\gbeg{3}{5}
\got{1}{B} \got{1}{B} \got{1}{F} \gnl
\gmu \gcn{1}{1}{1}{0} \gnl
\gvac{1} \hspace{-0,34cm} \glmptb \gnot{\hspace{-0,34cm}\widetilde{\psi}} \grmptb \gnl
\gvac{1} \gcl{1} \gcl{1} \gnl
\gvac{1} \gob{1}{F} \gob{1}{B} 
\gend=
\scalebox{0.84}[0.84]{
\gbeg{5}{6}
\got{1}{B} \got{1}{B} \got{2}{F}\gnl
\gmu \gcn{1}{1}{2}{2} \gnl
\gcmu \gcmu \gnl
\gcl{1} \glmptb \gnot{\hspace{-0,34cm}\tau_{B,F}} \grmptb \gcl{1} \gnl
\glm \grmo \gcl{1} \gnl
\gvac{1} \gob{1}{F} \gob{1}{B} 
\gend}
\stackrel{\equref{proj B-bialg}}{=}
\scalebox{0.84}[0.84]{
\gbeg{9}{9}
\gvac{1} \got{1}{B} \gvac{2} \got{1}{B} \got{3}{F}\gnl
\gwcm{3} \gwcm{3} \gcl{3} \gnl
\gcl{1} \gvac{1} \gcl{1} \grcm \gcl{3} \gnl
\gcl{1} \gvac{1} \glmptb \gnot{\hspace{-0,34cm}\tau_{B,B}} \grmptb \gcl{1}\gnl
\gcl{1} \gvac{1} \gcl{1} \grmo \gcl{1} \gvac{1} \gcn{1}{1}{3}{2} \gnl
\gwmu{3} \gwmu{3} \hspace{-0,4cm} \gcmu \gnl
\gvac{2} \gcn{1}{1}{1}{5} \gvac{2} \glmptb \gnot{\hspace{-0,34cm}\tau_{B,F}} \grmptb \gcl{1} \gnl
\gvac{4} \glm \grmo \gcl{1} \gnl
\gvac{5} \gob{1}{F} \gob{1}{B} 
\gend}\stackrel{\tau}{=}
\scalebox{0.84}[0.84]{
\gbeg{9}{10}
\gvac{1} \got{1}{B} \gvac{2} \got{1}{B} \got{3}{F}\gnl
\gwcm{3} \gwcm{3} \gcl{3} \gnl
\gcl{1} \gvac{1} \gcl{1} \grcm \gcl{2} \gnl
\gcl{1} \gvac{1} \glmptb \gnot{\hspace{-0,34cm}\tau_{B,B}} \grmptb \gcl{1} \gnl
\gcl{1} \gvac{1} \gcl{1} \grmo \gcl{1} \gvac{1}  \gcn{1}{1}{1}{0} \hspace{-0,22cm} \gcmu  \gnl
\gvac{1} \hspace{-0,2cm} \gwmu{3} \gcn{1}{1}{1}{2} \gvac{1} \hspace{-0,34cm} \glmptb \gnot{\hspace{-0,34cm}\tau_{B,F}} \grmptb \gcl{2} \gnl
\gvac{2} \gcn{1}{1}{2}{5} \gvac{2} \glmptb \gnot{\hspace{-0,34cm}\tau_{B,F}} \grmptb \gcl{1} \gnl
\gvac{4} \glm \gmu \gcn{1}{1}{1}{0} \gnl
\gvac{5} \gcl{1} \gvac{1} \hspace{-0,34cm} \grmo \gcl{1} \gnl
\gvac{5} \gob{2}{F} \gob{1}{B} 
\gend}
$$

$$\stackrel{\equref{F mod alg}}{=}
\scalebox{0.84}[0.84]{
\gbeg{10}{13}
\gvac{1} \got{1}{B} \gvac{2} \got{1}{B} \got{3}{F}\gnl
\gwcm{3} \gwcm{3} \gcl{2} \gnl
\gcl{1} \gvac{1} \gcl{1} \grcm \gcl{2} \gnl
\gcl{1} \gvac{1} \glmptb \gnot{\hspace{-0,34cm}\tau_{B,B}} \grmptb \gcl{1} \gvac{1} \gcn{1}{1}{1}{2}  \gnl
\gcl{1} \gvac{1} \gcl{1} \grmo \gcl{1} \gvac{1} \gcn{1}{1}{1}{0} \hspace{-0,22cm} \gwcm{3}  \gnl
\gvac{1} \hspace{-0,3cm} \gwmu{3} \gcn{1}{1}{1}{2} \gvac{1} \hspace{-0,34cm} \glmptb \gnot{\hspace{-0,34cm}\tau_{B,F}} \grmptb \gvac{1} \gcl{1} \gnl
\gvac{2} \gcn{1}{1}{2}{5} \gvac{2} \glmptb \gnot{\hspace{-0,34cm}\tau_{B,F}} \grmptb \gcn{1}{1}{1}{2} \gvac{1} \gcn{1}{1}{1}{2}\gnl
\gvac{4} \glm \gcl{1} \gcmu \gcmu \gnl
\gvac{5} \gcl{4} \gcl{1} \gcl{1} \glmptb \gnot{\hspace{-0,34cm}\tau_{B,F}} \grmptb \gcl{1} \gnl
\gvac{6} \gcn{1}{1}{1}{3} \glm \grmo \gcl{1} \gnl
\gvac{7} \grmo \gcl{1} \gvac{1} \gcl{1} \gnl
\gvac{7} \gwmu{3} \gnl
\gvac{5} \gob{1}{F} \gob{5}{B} 
\gend}
\stackrel{\tau}{\stackrel{coass. F}{\stackrel{B\x mod.}{=}}}
\scalebox{0.84}[0.84]{
\gbeg{11}{12}
\got{2}{B} \gvac{1} \got{1}{B} \got{7}{F}\gnl
\gcmu \gwcm{3} \gvac{2} \gcl{2} \gnl
\gcl{1} \gcl{1} \grcm \gcl{2} \gnl
\gcl{1} \glmptb \gnot{\hspace{-0,34cm}\tau_{B,B}} \grmptb \gcl{1} \gvac{2} \gwcm{3} \gnl
\gcl{3} \gcl{3} \grmo \gcl{1} \gvac{1} \hspace{-0,34cm} \gcmu \gcmu \gcn{1}{3}{2}{2} \gnl
\gvac{2} \gcn{1}{1}{2}{2} \gvac{1} \gcl{1} \glmptb \gnot{\hspace{-0,34cm}\tau_{B,F}} \grmptb \gcl{1} \gnl
\gvac{3} \gcn{1}{1}{0}{1} \glmptb \gnot{\hspace{-0,34cm}\tau_{B,F}} \grmptb \gcl{1} \gcl{1} \gnl
\gcn{1}{2}{2}{5} \hspace{-0,4cm} \gcn{1}{1}{4}{5} \gvac{2} \glmptb \gnot{\hspace{-0,34cm}\tau_{B,F}} \grmptb \gcl{1} \glmptb \gnot{\hspace{-0,34cm}\tau_{B,F}} \grmptb \gcn{1}{1}{2}{1} \gnl
\gvac{3} \glm \gcn{1}{1}{1}{3}  \glm \grmo \gcl{1} \gnl
\gvac{3} \glm \gvac{1} \grmo \gcl{1} \gvac{1} \gcl{1} \gnl
\gvac{4} \gcl{1} \gvac{1} \gwmu{3} \gnl
\gvac{4} \gob{1}{F} \gob{5}{B} 
\gend}
\stackrel{coass. B}{=}
\scalebox{0.84}[0.84]{
\gbeg{11}{10}
\got{2}{B} \gvac{2} \got{1}{B} \got{6}{F}\gnl
\gcmu \gvac{1} \gwcm{3} \gvac{1} \hspace{-0,34cm} \gwcm{3} \gnl
\gvac{1} \hspace{-0,3cm} \gcl{1} \gcl{1} \gwcm{3} \gcl{1} \gcmu \gcn{1}{2}{2}{2} \gnl
\gvac{1} \gcl{1} \gcl{1} \grcm \gcl{1} \glmptb \gnot{\hspace{-0,34cm}\tau_{B,F}} \grmptb \gcl{1} \gnl
\gvac{1} \gcl{1} \glmptb \gnot{\hspace{-0,34cm}\tau_{B,B}} \grmptb \gcl{1} \glmptb \gnot{\hspace{-0,34cm}\tau_{B,F}} \grmptb \glmptb \gnot{\hspace{-0,34cm}\tau_{B,F}} \grmptb \gcn{1}{1}{2}{1} \gnl
\gvac{1} \gcl{2} \gcl{2} \grmo \gcl{1} \gvac{1} \gcn{1}{1}{1}{-1} \glm \grmo \gcl{1} \gnl
\gvac{3} \glmptb \gnot{\hspace{-0,34cm}\tau_{B,F}} \grmptb \gcn{1}{1}{5}{1} \gcn{1}{2}{5}{3} \gnl
\gvac{1} \gcn{1}{1}{1}{3} \glm \grmo \gcl{1} \gnl
\gvac{2} \glm \gwmu{4}  \gnl
\gvac{3} \gob{1}{F} \gob{4}{B} 
\gend}=
$$

$$
\stackrel{\equref{nat rm}}{\stackrel{\tau}{=}}
\scalebox{0.84}[0.84]{
\gbeg{11}{11}
\got{2}{B} \gvac{2} \got{2}{B} \got{4}{F}\gnl
\gcmu \gvac{1} \gwcm{4} \gcmu \gnl
\gcl{5} \gcl{1} \gwcm{3} \gvac{1} \glmptb \gnot{\hspace{-0,34cm}\tau_{B,F}} \grmptb  \gcl{1} \gnl
\gvac{1} \gcl{1} \grcm \gcn{1}{1}{1}{2} \gvac{1} \hspace{-0,32cm} \gcmu \hspace{-0,22cm} \grmo \gcl{1} \gvac{1} \gnl
\gvac{2} \glmptb \gnot{\hspace{-0,34cm}\tau_{B,B}} \grmptb \gcn{1}{1}{1}{2} \gvac{1} \hspace{-0,3cm} \glmptb \gnot{\hspace{-0,34cm}\tau_{B,F}} \grmptb \gcl{1} \gcn{1}{2}{0}{0} \gnl
\gvac{2} \gcn{1}{1}{2}{2} \gcn{2}{1}{2}{3} \glmptb \gnot{\hspace{-0,34cm}\tau_{F,F}} \grmptb \glm  \gnl
\gvac{3} \gcn{1}{1}{0}{1} \glmptb \gnot{\hspace{-0,34cm}\tau_{B,F}} \grmptb \gcl{1} \gcn{1}{2}{3}{-1} \gcn{1}{1}{2}{2}  \gnl
\gvac{2} \gcn{1}{1}{0}{3}  \glm \grmo \gcl{1} \gvac{3} \gcn{1}{2}{0}{-1} \gnl
\gvac{3} \glm \grmo \gcl{1} \gnl
\gvac{4} \gcl{1} \gwmu{4}  \gnl
\gvac{4} \gob{1}{F} \gob{4}{B} 
\gend}
\stackrel{\equref{nat lm}}{\stackrel{F\x mod.}{=}}
\scalebox{0.84}[0.84]{
\gbeg{11}{11}
\got{2}{B} \gvac{2} \got{2}{B} \got{4}{F}\gnl
\gcmu \gvac{1} \gwcm{4} \gcmu \gnl
\gcl{4} \gcl{3} \gwcm{3} \gvac{1} \glmptb \gnot{\hspace{-0,34cm}\tau_{B,F}} \grmptb  \gcl{1} \gnl
\gvac{2} \grcm \gcn{1}{1}{1}{2} \gvac{1} \hspace{-0,32cm} \gcmu \hspace{-0,22cm} \grmo \gcl{1} \gvac{1} \gnl
\gvac{3} \gcl{1} \gcn{1}{1}{1}{2} \gvac{1} \hspace{-0,3cm} \glmptb \gnot{\hspace{-0,34cm}\tau_{B,F}} \grmptb \gcl{1} \gcn{1}{4}{0}{0} \gnl
\gvac{3} \gcn{1}{1}{0}{1} \gcn{1}{1}{0}{1} \glmptb \gnot{\hspace{-0,34cm}\tau_{F,F}} \grmptb \glm  \gnl
\gvac{1} \gcn{2}{2}{2}{5} \gcn{1}{1}{1}{3} \glm \gwmu{3} \gnl
\gvac{4} \glmptb \gnot{\hspace{-0,34cm}\tau_{B,F}} \grmptb \gcn{1}{1}{3}{1} \gnl
\gvac{3} \glm \grmo \gcl{1} \gvac{3} \gcn{1}{1}{0}{-1} \gnl
\gvac{4} \gcl{1} \gwmu{4}  \gnl
\gvac{4} \gob{1}{F} \gob{4}{B} 
\gend}
\stackrel{\equref{1-3}}{=}
\scalebox{0.84}[0.84]{
\gbeg{6}{9}
\got{1}{B} \got{2}{B} \got{2}{F}\gnl
\gcl{3} \gcmu \gcmu \gnl
\gvac{1} \gcl{1} \glmptb \gnot{\hspace{-0,34cm}\tau_{B,F}} \grmptb \gcl{1} \gnl
\gvac{1} \glm \grmo \gcl{1} \gnl
\hspace{-0,34cm} \gcmu \gcmu \gcn{1}{2}{0}{0} \gnl
\gcl{1} \glmptb \gnot{\hspace{-0,34cm}\tau_{B,F}} \grmptb \gcl{1} \gnl
\glm \grmo \gcl{1} \gcn{1}{1}{2}{1} \gnl
\gvac{1} \gcl{1} \gmu \gnl
\gvac{1} \gob{1}{F} \gob{2}{B} 
\gend}=
\gbeg{4}{5}
\got{1}{B} \got{1}{B} \got{1}{F} \gnl
\gcl{1} \glmptb \gnot{\hspace{-0,34cm}\widetilde{\psi}} \grmptb \gnl
\glmptb \gnot{\hspace{-0,34cm}\widetilde{\psi}} \grmptb \gcl{1} \gnl
\gcl{1} \gmu \gnl
\gob{1}{F} \gob{2}{B} 
\gend
$$
We clarify only the brief notations in the above equalities that possibly are not clear enough: 
in the third equation we applied the monad distributive law for $\tau$, in the fifth one the comonad distributive law for $\tau$ and left $B$-module 
structure of $F$, in the seventh one the comonad distributive law for $\tau$, in the eighth one the right $F$-module structure on $B$ and in the ninth one the third equation in 
\equref{1-3}. Observe that in order for the comultiplication of $B$ to be coassociative 
we need the third of the above assumptions. 
Moreover, in order for the right $F$-action on $B$ to be proper, by \equref{weak action} it should be either trivial (in which case we are in the setting of $\psi_1$ from \prref{psi_4}), 
or the fourth assumption above should be fulfilled. 

The proof that $\mu_M$ canonical (assuming that $\sigma$ is trivial) satisfies the 2-cell condition is analogous: take left-right symmetric diagrams and interchange the r\^oles of $B$ and $F$. 
Recall that we called this {\em $\alpha$-symmetry}. Then the statement to prove, as well as the proof itself, are $\alpha$-symmetric to those from above. One part of the conditions 
necessary to prove the $\psi$ axioms is already auto $\alpha$-symmetric, from those conditions which do not contain their $\alpha$-symmetric counterpart, 
we need to add: \equref{proj F-bialg}, \equref{mod alg} and the second equation in \equref{1-3}. 

The monad law for $\mu_M$ comes down to associativity of $F$, it is fulfilled by \equref{weak assoc. mu_M}, since $\mu_M$ is canonical. The necessary relations for the units are proved 
straighforwardly. 
\qed\end{proof}

\bigskip

By $\pi$-symmetry, we have:

\begin{prop} \prlabel{phi_4}
Set $\phi'_4=
\gbeg{4}{5}
\got{1}{} \got{1}{F} \got{1}{B} \gnl
\grmo \gvac{1} \gcl{1} \grcm \gnl
\gcl{1} \glmptb \gnot{\hspace{-0,34cm}\tau_{F,B}} \grmptb \gcl{1} \gnl
\gmu \gmu \gnl
\gob{2}{B} \gob{2}{F} 
\gend$. Then $(B, \phi'_4, \Delta_C' \hspace{0,2cm}\textnormal{canonical})$ is a right cowreath around $F$ if the following assumptions hold true: 
\begin{enumerate}
\item $(B,F)$ is a comonad Hopf datum and $B$ is a monad; 
\item there are 2-cells 
$
\gbeg{2}{3}
\got{1}{B} \got{1}{F} \gnl
\glm \gnl
\gob{3}{F}
\gend,
\hspace{0,3cm}
\gbeg{2}{3}
\got{1}{B} \got{1}{F} \gnl
\grmo \gcl{1} \gnl
\gob{1}{B} 
\gend,
\hspace{0,3cm}
\gbeg{3}{3}
\got{1}{F} \got{1}{F}  \gnl
\glmpt \gnot{\hspace{-0,34cm}\sigma} \grmptb \gnl
\gob{3}{B}
\gend
$ such that \equref{proj F-bialg}, \equref{proj B-bialg} and the first and third equation in \equref{4-6} hold; 
\item $F$ is a monad (or, \equref{weak assoc. mu_M}--\equref{weak unity eta_M} hold and 
$\gbeg{2}{3}
\got{1}{B} \got{1}{F} \gnl
\glm \gnl 
\gob{3}{F} \gnl
\gend$ is trivial or $\sigma$ is trivial (hence $\mu_M$ is canonical);
\item $\rho'$ is trivial (hence $\Delta_C'$ is canonical);
\item $\tau_{F,B}$ is natural with respect to 
$\gbeg{2}{3}
\got{3}{F} \gnl
\grmo \gvac{1} \gcl{1} \gnl 
\gob{1}{B} \gob{1}{F} \gnl
\gend$ and 
$\gbeg{2}{3}
\got{1}{B} \gnl
\grcm \gnl 
\gob{1}{B} \gob{1}{F} \gnl
\gend$ (recall \equref{nat lcm}, \equref{nat rcm}).
\end{enumerate}
\end{prop}

\bigskip

To the 2-cells 
$\gbeg{2}{1}
\glm \gnl
\gend, 
\gbeg{2}{1}
\grmo \gcl{1} \gnl 
\gend, 
\gbeg{2}{1}
\glmpt \gnot{\hspace{-0,34cm}\sigma} \grmptb \gnl
\gend 
$ \hspace{0,1cm}
in a monad Hopf datum we will assign numbers 0 or 1, depending on whether the action and cocycle are trivial or not. The first two 2-cells determine $\psi$. 
Then a monad Hopf datum determined by the former three 2-cells we will denote shortly by $((i,j),k)$ with $i,j,k\in\{0,1\}$, where the first two entries correspond to the action 2-cells, 
and the third one to the cocycle. Similarly, we will assign a 3-tuple $((\crta i,\crta j),\crta k)$, with $\crta i,\crta j,\crta k\in\{0,1\}$,  to the comonad Hopf datum determined by the 
2-cells 
$\gbeg{2}{1}
\grcm \gnl
\gend, 
\gbeg{2}{1}
\grmo \gvac{1} \gcl{1} \gnl
\gend, 
\gbeg{3}{1}
\glmpb \gnot{\hspace{-0,34cm}\rho'} \grmptb \gnl
\gend$ in this same order. Neglecting the two simplest (co)monad Hopf data studied above, we present in the following table the rest of monad Hopf data which we studied in \prref{psi_1,2} and 
\prref{psi_4} and which, under respective hypotheses, determine a wreath, together with their $\pi$-symmetric companions: 
\begin{table}[h!]
\begin{center}
\begin{tabular}{ c c } 
 $(\psi, \sigma)$ & \hspace{0,2cm} ($\pi$-symmetric $\phi'$, $\rho'$) \\ [0.5ex]
\hline
$ ((1,0),0)$ & $((\crta 0,\crta 1),\crta 0)$ \\ [1ex]   
 $((0,1),0)$ & $((\crta 1,\crta 0),\crta 0)$  \\ [1ex]
 $((0,1),1)$ & $((\crta 1,\crta 0),\crta 1)$  \\ [1ex]
 $((1,1),0)$ & $((\crta 1,\crta 1),\crta 0)$  \\ [1ex]
\end{tabular}
\caption{(Co)monad Hopf data which are (co)wreaths}
\label{table:1}
\end{center}
\end{table}\\ 
\indent If we are to decide which Hopf data determine a paired wreath, we observe that only $\psi_4$ presented restrictions on its $\pi$-symmetric companion (see condition 3. in \prref{psi_4}). 
However, from the above table it is clear that $\psi_4$ may be combined will all the resting four $\phi'$'s. So, from Table \ref{table:1} we possibly have $4\times 4=16$ paired wreaths. 
It remains to study when a Hopf datum implies conditions 5 and 6 of a paired wreath. 

Let us analyze the condition 5. From \equref{lambda B} we know the form that 2-cell $\lambda_B$ should have. We only may claim, which is directly checked, 
that \equref{proj B} is fulfilled if either $\sigma$ is trivial, or $\sigma$ is non-trivial (then we are in the setting of $((0,1),1)$) and 
$\gbeg{2}{1}
\grcm \gnl
\gend$ is trivial. 
The latter case by $\pi$-symmetry yields that $((\crta 1,\crta 0),\crta 1)$ can coexist in a paired wreath only with 
$((0,1),0)$ and $((0,1),1)$. Summing up, from the above 16 combinations the 
ones in Table \ref{table:2} certainly fulfill condition 5. 
\begin{table}[h!]
\begin{center}
\begin{tabular}{ c c } 
 $(\psi, \sigma)$ & \hspace{0,2cm} ((matching cowreaths) \\ [0,5ex]
\hline
 $ ((0,1),1)$ & $((\crta 1,\crta 0),\crta 0), ((\crta 1,\crta 0),\crta 1)$  \\ [1ex]
  $((0,1),0)$ & all 4  \\ [1ex]
$((1,0),0)$ & \hspace{0,8cm}  $((\crta 0,\crta 1),\crta 0), ((\crta 1,\crta 0),\crta 0), ((\crta 1,\crta 1),\crta 0)$  \\ [1ex]   
 $((1,1),0)$ & \hspace{0,8cm}  $((\crta 0,\crta 1),\crta 0), ((\crta 1,\crta 0),\crta 0), ((\crta 1,\crta 1),\crta 0)$  \\ [1ex]
\end{tabular}
\caption{(Co)wreaths satisfying \equref{proj B} / \equref{proj F}}
\label{table:2}
\end{center}
\end{table}

As far as for condition 6, consider the case where the four (co)action 2-cells \equref{(co)module cells} are non-trivial and $\sigma$ and $\rho'$ are trivial. 
The proof in string diagrams of the fourth relation in \deref{tau-bim} in this setting, with 
$\tau_{FB,FB}$ given as in \equref{tau FB}, is pretty tedious and possibly leads to a never ending loop. We proved, however, that if one of the 
four (co)action 2-cells is trivial, then the Hopf datum is a bimonad. (Observe that if one proves the result choosing any of the four 2-cells to be trivial, then 
the result holds true if any other of the four 2-cells is trivial, because of the auto $\alpha$- and $\pi$-symmetry of a Hopf datum, \rmref{when cocycles are trivial}.) 
Instead of presenting here our proof in string diagrams, which is tedious, we make the following observation. 
In \cite[Section 2]{BD1} the authors use a remarkable tool in the context of a Hopf datum in a braided monoidal category $\C$ that can fully be taken over and used in our context of 
a Hopf datum in a general 2-category $\K$. It serves to study a Hopf datum of the form $((1,1),0)$ with $((\crta 1,\crta 1),\crta 0)$ and it helps to decide when the 
identity \equref{paired biwreath biproduct} in this setting is fulfilled. Namely, they introduce a {\em recursive Hopf datum}, they prove in \cite[Theorem 2.14]{BD1} 
that for a recursive Hopf datum $(B,F)$ the object $B\ot F$ is a bialgebra in $\C$ (the 1-cell $FB$ is a bimonad in $\K$) and in \cite[Definition and Proposition 2.15]{BD1} they prove that 
every {\em trivalent} Hopf datum is a recursive Hopf datum. A trivalent Hopf datum is a one for which $\sigma$ and $\rho'$ are trivial and at most three of the four (co)action 2-cells 
\equref{(co)module cells} are non-trivial. 
Thus we may state:

\begin{thm} \thlabel{sufficient for bimonad}
Let $(B,F)$ be a Hopf datum in $\K$ determined by $((i,j),0)$ with $((\crta k,\crta l),\crta 0)$ with $i,j,k,l\in\{0,1\}$ and so that $i+j+k+l\leq 3$. 
Suppose that the 2-cells $\tau$ satisfy the naturality conditions and Yang-Baxter equations \equref{nat rm} -- \equref{nat sigma}, \equref{YBE FBF} -- \equref{YBE FF}. 
Then $FB$ is a $\tau$-bimonad in $\K$.   

In view of the above said, all the trivalent Hopf data from Table \ref{table:2} are paired wreaths. Consequently, in the latter cases a Hopf datum is equivalent to a paired wreath.
\end{thm}

The Theorem applies also to the monad Hopf datum studied in \prref{simplest case}, which yields the following four Hopf data which simultaneously are paired wreaths: 
\begin{table}[h!]
\begin{center}
\begin{tabular}{ c  c } 
 $(\psi, \sigma)$ &  (matching cowreaths) \\ [0,5ex]
\hline
 $ ((0,0),0)$ & \hspace{0,4cm} $((\crta 0,\crta 0),\crta 0), ((\crta 1,\crta 0),\crta 0), ((\crta 0,\crta 1),\crta 0), ((\crta 1,\crta 1),\crta 0)$  \\ [1ex]
\end{tabular}
\end{center}
\end{table}

Observe that the three cases in the third line in Table \ref{table:2} are $\alpha$-symmetric to the corresponding ones in the second line. So in total there are 2+4+3+4=13 non-isomorphic 
trivalent paired wreaths. 

\subsection{Examples}

When $\K=\hat\C$ where $\C$ is a braided monoidal category, the r\^ole of the 2-cells $\tau$ is played by the braiding. If $\C$ is not braided, $\tau$ is a local braiding, 
existing between the objects $B$ and $F$ and all their combinations (recall the (co)monadic distributive law properties, Yang-Baxter equations and naturalities for $\tau$).

\begin{ex}
When $\K=\hat\C$, where $\C$ is a monoidal category, a paired wreath and a Hopf datum are written out in the string diagrams on the previous pages, where the strings (2-cells) are morphisms, 
their source and targets (1-cells) are objects, and $B$ is an algebra, $F$ a coalgebra in $\C$, satisfying the listed axioms. 
\end{ex}

\begin{ex}
The paired wreath $((1,0),0), ((\crta 1,\crta 0),\crta 0)$ is a 2-categorical version of Radford biproduct bialgebra. 
Set $\K=\hat Vec$.  
Then \cite[Theorem 2.1 and Proposition 2]{Rad1}, that characterizes when Radford biproduct $B\times F$ is a bialgebra, is a special case of \thref{sufficient for bimonad} in 
the setting $((1,0),0), ((\crta 1,\crta 0),\crta 0)$, {\em i.e.}  $(\psi_1, \phi'_2)$. This 2-categorical formulation of the Radford biproduct coincides with our 
biwreath from \cite{Femic5}. 
\end{ex}

\begin{ex} \exlabel{ex 2}
The paired wreath $((1,0),0), ((\crta 0,\crta 1),\crta 0)$ is a 2-categorical version of Majid's bicrossproduct. 
Similarly, as in the above example, in $\K=\hat Vec$ there is \cite[Theorem 3.3]{Maj5}, which characterizes when Majid's bicrossproduct is a bialgebra. 
It is a special case of \thref{sufficient for bimonad} in 
the setting $((1,0),0), ((\crta 0,\crta 1),\crta 0)$, {\em i.e.}  $(\psi_1, \phi'_1)$ (or $\alpha$-symmetrically: $((0,1),0), ((\crta 1,\crta 0),\crta 0)$, {\em i.e.}  $(\psi_2, \phi'_2)$). 
\end{ex}

Observe that since in the above two examples $\sigma$ and $\rho'$ are trivial, {\em i.e.} the 2-cels $\mu_M$ and $\Delta_C'$ are canonical, from the bialgebra condition on 
$B\times F$ one may deduce the compatibility conditions for $F$ in \equref{(co)unital} and \equref{(co)unital mixed}, as well as \equref{psi unital add} (or the compatibility 
conditions for the (co)actions in \coref{(co)actions}). In a general setting, in a paired wreath though, we are forced to require these compatibility conditions as part of the definition.

\begin{ex}
The wreath $((0,1),1)$ in general $\K$ is a 2-categorical generalization of Sweedler's cocycle twisted smash product algebra, \cite{Sw1}. Moreover, it is the 2-categorical formulation 
of the biwreath-like object we studied in \cite{Femic5} for $\K=\hat\C$, where $\C$ is a braided monoidal category. 
Its $\pi$-symmetric cowreath is the dual construction representing Sweedler's cycle twisted cosmash coproduct coalgebra. 
\end{ex}

\begin{ex}
At the end of \cite{Maj6} Majid included {\em cocycle bicrossproducts} which are bialgebras in $Vec$. They correspond to $((0,1),1), ((\crta 1,\crta 0),\crta 1)$ 
in our setting. Namely, Majid's conditions C, A, B, D correspond to our first and second identities in \equref{1-3} and \equref{4-6}, respectively. 
(To see the latter one of the four, write out our identity and apply \equref{F mod alg} and \equref{B comod coalg}.)  
In Table \ref{table:2} we see that in the 2-categorical setting it only remains to prove that $FB$ is a bimonad in order to be able to claim that $(B,F)$ is 
a paired wreath.  
\end{ex}

\begin{ex}
In \cite[Proposition 3.12]{Maj5} Majid characterized a {\em matched pair of bialgebras} $(B,F)$ in $Vec$. It turns out to be a Hopf datum for $\K=\hat Vec$ in the 
setting $((1,1),0), ((\crta 0,\crta 0),\crta 0)$, {\em i.e.}  $(\psi_4, \tau_{F,B})$. The ``fourth bialgebra compatibility'' for $B$ and $F$ are our identities 
\equref{proj B-bialg} and \equref{proj F-bialg}, left and right module coalgebra assumptions are our second and third identity in \equref{1-3}, and Majid's conditions 
A, B, C are our identities \equref{mod alg}, \equref{F mod alg} and the first identity in \equref{1-3}, respectively. The rest of our identities in a Hopf datum hold trivially, 
since the two coactions and the (co)cycles are trivial. 

Matched pairs of bialgebras in any braided monoidal category $\C$ are treated in \cite{BD1, ZC, Femic1}. In \cite[Theorem 1.4]{ZC} it is proved that given a 
matched pair of bialgebras $(B,F)$ in $\C$ the ``wreath'' product $B\times F$ (usually called {\em double cross product} in this setting) is a bialgebra in $\C$. 
Recall that the Drinfel`d double is a particular case. 

\thref{sufficient for bimonad} confirms the result from the 2-categorical viewpoint. 
\end{ex}

\begin{ex} \exlabel{ex 6}
One of the Hopf data for which we could not prove to be a paired wreath (neither a wreath) is determined by $((1,1),1), ((\crta 1,\crta 0),\crta 0)$, {\em i.e.} $(\psi_4, \phi'_2)$. 
If $\K=\hat Vec$, this Hopf datum is Schauenburg's cosmash product defined in \cite[Theorem 5.1]{Sch3}. 
\end{ex}

We end this subsection proving the following result, which is a generalization of \cite[Lemma 5.2]{Femic5}. 

\begin{lma} \lelabel{lambda d.l.}
In a paired wreath and a Hopf datum determined by $((1,0),0), ((\crta 1,\crta 0),\crta 0)$, {\em i.e.} Radford biproduct in $\K$, 
the 1-cell $B$ is a right bimonad, and $F$ is a left bimonad in $\K$. 
\end{lma}

\begin{proof}
In view of the identity \equref{proj F-bialg} it suffices to prove that the 2-cell $\lambda_F$ from \equref{lambda F} 
is left monadic and comonadic distributive law. The claim for $B$ will then hold by $\pi$-symmetry. We find: 
$$\hspace{-0,2cm}
\gbeg{3}{5}
\got{1}{F} \got{1}{F} \got{1}{F} \gnl
\gmu \gcn{1}{1}{1}{0} \gnl
\gvac{1} \hspace{-0,34cm} \glmptb \gnot{\hspace{-0,34cm}\lambda_F} \grmptb \gnl
\gvac{1} \gcl{1} \gcl{1} \gnl
\gvac{1} \gob{1}{F} \gob{1}{F} 
\gend=
\scalebox{0.84}[0.84]{
\gbeg{5}{9}
\got{1}{F} \gvac{1} \got{1}{F}\got{1}{F}\gnl
\gwmu{3} \gcl{4} \gnl
\gvac{1} \gcl{1} \gnl
\gwcm{3} \gnl
\gcl{1} \grmo \gvac{1} \gcl{1} \gnl
\gcl{1} \gcl{1} \glmpt \gnot{\hspace{-0,34cm}\tau_{F,F}} \grmptb \gnl
\gcl{1} \glm \gcl{2} \gnl
\gwmu{3} \gnl
\gvac{1} \gob{1}{F} \gvac{1} \gob{1}{F} 
\gend}
\stackrel{\equref{proj F-bialg}}{=}
\scalebox{0.84}[0.84]{
\gbeg{7}{10}
\gvac{1} \got{1}{F} \gvac{1} \got{2}{F} \got{1}{F}\gnl
\gwcm{3} \gcmu \gcl{4} \gnl
\gcl{1} \grmo \gvac{1} \gcl{1} \gcl{1} \gcl{2} \gnl
\gcl{1} \gcl{1} \glmpt \gnot{\hspace{-0,34cm}\tau_{F,F}} \grmptb \gnl
\gcl{1} \glm \gmu \gnl
\gwmu{3} \hspace{-0,22cm} \grmo \gvac{1} \gcl{1} \gcn{1}{1}{2}{1} \gnl
\gvac{2} \gcn{1}{1}{0}{1} \gcl{1} \glmpt \gnot{\hspace{-0,34cm}\tau_{F,F}} \grmptb \gnl
\gvac{2} \gcl{1} \glm \gcl{2} \gnl
\gvac{2} \gwmu{3} \gnl
\gvac{3} \gob{1}{F} \gob{3}{F} 
\gend}\stackrel{ass. F}{\stackrel{*}{=}} 
\scalebox{0.84}[0.84]{
\gbeg{8}{13}
\gvac{1} \got{1}{F} \gvac{2} \got{1}{F} \got{3}{F}\gnl
\gwcm{3} \gwcm{3} \gcl{7} \gnl
\gcl{1} \grmo \gvac{1} \gcl{1} \gcl{1} \gvac{1} \gcl{2} \gnl
\gcl{1} \gcl{1} \glmpt \gnot{\hspace{-0,34cm}\tau_{F,F}} \grmptb \gnl
\gcl{6} \glm \gcn{1}{1}{1}{2} \gcn{1}{1}{3}{4} \gnl
\gvac{2} \gcl{5} \hspace{-0,22cm} \grmo \gvac{1} \gcl{1} \grmo \gvac{1} \gcl{1} \gnl
\gvac{3} \gcl{1} \glmpt \gnot{\hspace{-0,34cm}\tau_{F,F}} \grmptb \gcl{1} \gnl
\gvac{3} \gmu \gmu \gnl
\gvac{4} \gcn{1}{1}{0}{2} \gvac{1} \hspace{-0,34cm} \glmpt \gnot{\hspace{-0,34cm}\tau_{F,F}} \grmptb \gnl
\gvac{5} \glm \gcl{3} \gnl
\gcn{1}{1}{3}{4} \gvac{2} \gwmu{4} \gnl
\gvac{2} \hspace{-0,34cm} \gwmu{4} \gnl
\gvac{3} \gob{2}{F} \gob{6}{F} 
\gend}$$

$$\stackrel{B\x comod.}{\stackrel{B\x mod.}{=}}
\scalebox{0.84}[0.84]{
\gbeg{9}{13}
\gvac{1} \got{2}{F} \gvac{2} \got{1}{F} \got{3}{F}\gnl
\gwcm{4} \gwcm{3} \gcl{4} \gnl
\gcl{1} \gvac{1} \grmo \gvac{1} \gcl{1} \gcl{1} \grmo \gvac{1} \gcl{1} \gnl
\gcl{1} \gcn{1}{1}{3}{2} \gvac{1} \glmpt \gnot{\hspace{-0,34cm}\tau_{F,F}} \grmptb \gcl{1} \gcl{2} \gnl
\gcl{7} \gcmu \gcl{1} \glmpt \gnot{\hspace{-0,34cm}\tau_{F,F}} \grmptb \gnl
\gvac{1} \gcl{1} \glmpt \gnot{\hspace{-0,34cm}\tau_{F,F}} \grmptb \gcl{1} \gmu \gcn{1}{1}{1}{0} \gnl
\gvac{1} \glm \gcl{1} \gcn{1}{1}{1}{2} \gvac{1} \hspace{-0,34cm} \glmpt \gnot{\hspace{-0,34cm}\tau_{F,F}} \grmptb \gnl
\gvac{2} \gcn{1}{2}{2}{2} \gcn{1}{2}{2}{5} \gvac{1} \glm \gcl{5} \gnl
\gvac{6} \gcl{1} \gnl
\gvac{2} \gcn{1}{1}{2}{3} \gvac{2} \glm \gcl{3} \gnl
\gvac{3} \gwmu{4} \gnl
\gvac{1} \hspace{-0,34cm} \gwmu{5} \gnl
\gvac{3} \gob{2}{F} \gob{6}{F} 
\gend}
\stackrel{\tau}{\stackrel{nat.}{\stackrel{\equref{mod alg}}{=}}}
\scalebox{0.84}[0.84]{
\gbeg{8}{10}
\gvac{1} \got{1}{F} \gvac{2} \got{1}{F} \got{3}{F}\gnl
\gwcm{3} \gwcm{3} \gcl{2} \gnl
\gcl{1} \grmo \gvac{1} \gcl{1} \gcl{1} \grmo \gvac{1} \gcl{1} \gnl
\gcl{1} \gcl{3} \gbr \gcl{1} \glmpt \gnot{\hspace{-0,34cm}\tau_{F,F}} \grmptb \gnl
\gcl{4} \gcl{1} \gcl{2} \gcn{1}{1}{1}{3} \glm \gcl{1} \gnl
\gvac{4} \glmpt \gnot{\hspace{-0,34cm}\tau_{F,F}} \grmptb \gcl{1} \gnl
\gvac{1} \gcn{1}{1}{1}{3} \gwmu{3} \gmu \gnl
\gvac{2} \glm \gcn{1}{2}{4}{4} \gnl
\gwmu{4} \gnl
\gvac{1} \gob{2}{F} \gob{6}{F} 
\gend}
\stackrel{\tau}{=}
\scalebox{0.84}[0.84]{
\gbeg{7}{12}
\gvac{1} \got{1}{F} \gvac{1} \got{1}{F} \got{3}{F}\gnl
\gvac{1} \gcl{5} \gwcm{3} \gcl{2} \gnl
\gvac{2} \gcl{1} \grmo \gvac{1} \gcl{1} \gnl
\gvac{2} \gcl{1} \gcl{1} \glmpt \gnot{\hspace{-0,34cm}\tau_{F,F}} \grmptb \gnl
\gvac{2} \gcl{1} \glm \gcl{6} \gnl
\gvac{2} \gwmu{3} \gnl
\gwcm{3} \gcl{2} \gnl
\gcl{1} \grmo \gvac{1} \gcl{1} \gnl
\gcl{1} \gcl{1} \glmpt \gnot{\hspace{-0,34cm}\tau_{F,F}} \grmptb \gnl
\gcl{1} \glm \gcl{1} \gnl
\gwmu{3} \gwmu{3} \gnl
\gvac{1} \gob{1}{F} \gob{5}{F} 
\gend}=
\gbeg{4}{5}
\got{1}{F} \got{1}{F} \got{1}{F} \gnl
\gcl{1} \glmptb \gnot{\hspace{-0,34cm}\lambda_F} \grmptb \gnl
\glmptb \gnot{\hspace{-0,34cm}\lambda_F} \grmptb \gcl{1} \gnl
\gcl{1} \gmu \gnl
\gob{1}{F} \gob{2}{F.} 
\gend
$$
At the place * we applied the first identity in \equref{4-6}. Observe that the pre-multiplication of $F$ is associative by \equref{weak assoc. mu_M} since $\sigma$ is trivial. 
The left $B$-comodule structure on $F$ is proper, since $\rho'$ is trivial, see \equref{weak coaction}. In the fifth equality we applied naturality of $\tau_{F,F}$ with respect to 
$\gbeg{2}{1}
\glm \gnl
\gend$ and the monadic distributive law of $\tau_{F,F}$, a part from \equref{mod alg}, and the same distributive law we applied in the sixth equality. The (co)unital 
distributive law property for $\lambda$ we proved before \equref{mu-delta-final}. The resting comonadic distributive law property follows by vertical symmetry: 
one uses \equref{proj F-bialg}, the coassociativity of the comultiplication of $F$, the third identity in \equref{1-3}, left $B$-(co)module structures of $F$, comonadic 
distributive law property of $\tau_{F,F}$, naturality of $\tau_{F,F}$ with respect to 
$\gbeg{2}{1}
\grmo \gvac{1} \gcl{1} \gnl 
\gend$ 
and \equref{B comod coalg}. 
\qed\end{proof}

\subsection{New definitions arising from Hopf data and concluding remarks}

The names we put in the identities \equref{mod alg}--\equref{normalized 2-cycle ro'} suggest new definitions of the respective objects, more precisely, 1-cells in 2-categories.  
We highlight some of them here. Given a monad Hopf datum $(B,F, \psi, \mu_M)$ in $\K$ (we consider $\eta_M$ canonical), the following notions are defined through the corresponding identity: 
\begin{enumerate}
\item a (normalized) 2-cocycle, 
\item action twisted by a 2-cocycle,
\item module monad (in a broader sense: the monad is weak associative, or the action on the monad is twisted by a 2-cocycle).
\end{enumerate}
A comonad Hopf datum delivers $\pi$-symmetric notions. 

\medskip

Observe that taking into account the formula \equref{mu-delta-final} for $\mu_M$, weak associativity of the pre-multiplication on $F$ (\equref{weak assoc. mu_M}), 
the 2-cocycle condition and twisted action get their specific forms: 
$$
\gbeg{4}{4}
\got{1}{F} \got{1}{F} \got{3}{F} \gnl
\gcl{1} \gwmu{3} \gnl
\gwmu{3} \gnl
\gob{3}{F}
\gend=
\gbeg{5}{8}
\got{1}{} \got{1}{F} \gvac{1} \got{2}{F} \got{1}{F} \gnl
\gwcm{3} \gcmu \gcl{5} \gnl
\gcl{1} \grmo \gvac{1} \gcl{1} \gcl{1} \gcl{2} \gnl
\gcl{1} \gcl{1} \gbr \gnl
\gcl{1} \glm \glmpt \gnot{\hspace{-0,34cm}\sigma} \grmptb  \gnl
\gwmu{3} \gvac{1} \glm \gnl
\gvac{1} \gwmu{5} \gnl
\gob{7}{F} 
\gend
$$
\begin{center}
{\footnotesize weak associativity}
\end{center}

$$
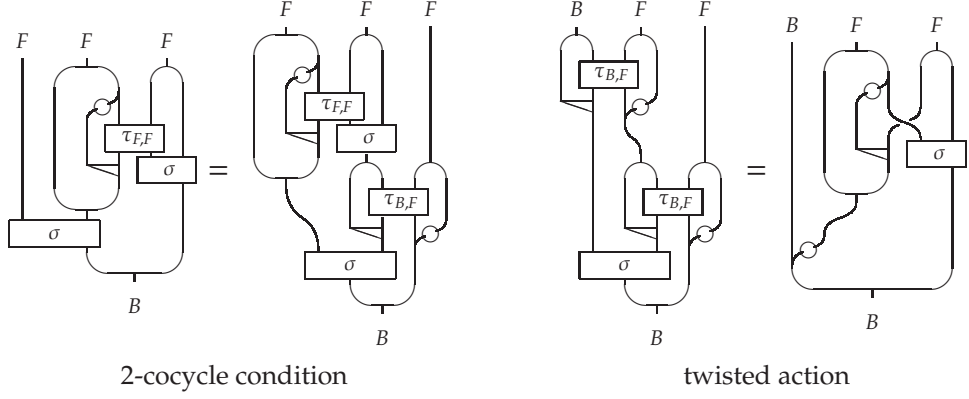
\gbeg{6}{9}
\got{1}{F} \got{3}{F} \got{2}{F}\gnl
\gcl{5} \gwcm{3} \gcmu \gnl
\gvac{1} \gcl{1} \grmo \gvac{1} \gcl{1} \gcl{1} \gcl{2} \gnl
\gvac{1} \gcl{1} \gcl{1} \glmptb \gnot{\hspace{-0,34cm}\tau_{F,F}} \grmptb \gnl
\gvac{1} \gcl{1} \glm \glmpt \gnot{\hspace{-0,34cm}\sigma} \grmptb  \gnl
\gvac{1} \gwmu{3} \gvac{1} \gcl{2} \gnl
\glmpt \gnot{\sigma} \gcmp \grmptb  \gnl
\gvac{2} \gwmu{4} \gnl
\gob{8}{B}
\gend=
\gbeg{6}{11}
\got{3}{F} \got{2}{F} \got{2}{F} \gnl
\gwcm{3} \gcmu \gcn{1}{4}{2}{2} \gnl
\gcl{1} \grmo \gvac{1} \gcl{1} \gcl{1} \gcl{2} \gnl
\gcl{1} \gcl{1} \glmptb \gnot{\hspace{-0,34cm}\tau_{F,F}} \grmptb \gnl
\gcl{1} \glm \glmpt \gnot{\hspace{-0,34cm}\sigma} \grmpt  \gnl
\gwmu{3} \gcmu \gcmu \gnl
\gvac{1} \gcn{1}{2}{1}{3} \gvac{1} \gcl{1} \glmpt \gnot{\hspace{-0,34cm}\tau_{B,F}} \grmptb \gcl{1} \gnl
\gvac{3} \glm \grmo \gcl{1} \gnl 
\gvac{2} \glmpt \gnot{\sigma} \gcmpb \grmpt \gcl{1} \gnl
\gvac{3} \gwmu{3} \gnl
\gob{9}{B}
\gend
\hspace{1,5cm}
\gbeg{6}{11}
\got{2}{B} \got{2}{F} \got{2}{F} \gnl
\gcmu \gcmu \gcn{1}{4}{2}{2} \gnl
\gcl{1} \glmptb \gnot{\hspace{-0,34cm}\tau_{B,F}} \grmptb \gcl{1} \gnl
\glm \grmo \gcl{1} \gnl
\gvac{1} \gcl{4} \gcn{1}{1}{1}{2} \gnl 
\gvac{2} \gcmu \gcmu \gnl
\gvac{2} \gcl{1} \glmptb \gnot{\hspace{-0,34cm}\tau_{B,F}} \grmptb \gcl{1}  \gnl
\gvac{2} \glm \grmo \gcl{1}  \gnl
\gvac{1} \glmpt \gnot{\sigma} \gcmpb \grmpt \gcl{1} \gnl
\gvac{2} \gwmu{3} \gnl
\gob{7}{B}
\gend=
\gbeg{3}{10}
\got{1}{B} \gvac{1} \got{1}{F} \gvac{1} \got{2}{F} \gnl
\gcl{6} \gwcm{3} \gcmu \gnl
\gvac{1} \gcl{1} \grmo \gvac{1} \gcl{1} \gcl{1} \gcl{2} \gnl
\gvac{1} \gcl{1} \gcl{1} \gbr \gnl
\gvac{1} \gcl{1} \glm \glmpt \gnot{\hspace{-0,34cm}\sigma} \grmptb  \gnl
\gvac{1} \gwmu{3} \gvac{1} \gcl{3}  \gnl
\gvac{1} \gcn{1}{1}{3}{1}  \gnl
\grmo \gcl{1} \gnl
\gwmu{6} \gnl
\gob{6}{B}
\gend
$$ \vspace{0,4cm}
\hspace{3,6cm} {\footnotesize 2-cocycle condition} \hspace{4,2cm} {\footnotesize twisted action}

\noindent This corresponds to the monad Hopf datum $((1,1),1)$. 
When $\K=\hat Vec$ this 2-cocycle condition and twisted action read: 
\begin{multline}
 \sigma(f,g_{(1)} (g_{(2)_{[-1]}}\rtr k_{(1)})) \sigma(g_{(2)_{[0]}}, k_{(2)}) = \\
\sigma(f_{(1)}(f_{(2)_{[-1]}}\rtr g_{(1)}), (\sigma(f_{(2)_{[0]}}, g_{(2)})_{(1)}\rtr k_{(1)})) (\sigma(f_{(2)_{[0]}}, g_{(2)})_{(2)}\ltr k_{(2)}) 
\end{multline}
and 
\begin{multline}
 \sigma(b_{(1)}\rtr f_{(1)}, (b_{(2)}\ltr f_{(2)})_{(1)}\rtr g_{(1)}) ((b_{(2)}\ltr f_{(2)})_{(2)}\ltr g_{(2)}) = \\
(b\ltr (f_{(1)}(f_{(2)_{[-1]}}\rtr g_{(1)}))) \sigma(f_{(2)_{[0]}}, g_{(2)}) 
\end{multline}
for $f,g,k\in F, b\in B$, respectively.
When the left $B$-action on $F$ is trivial, this (more general) definition of a 2-cocycle and twisted action recovers the twisted action by Sweedler's 2-cocycle.

\bigskip

The identities \equref{1-3} and \equref{4-6} holding in a Hopf datum (and a paired wreath) suggest some new definitions, too. Recall, as we observed in \rmref{when cocycles are trivial}, 
that when $\sigma$ and $\rho'$ are trivial, the second and the fourth of these six identities are $\alpha$-symmetric to the third and the sixth one, respectively, 
and the fifth identity is trivially fulfilled. 
Then we obtain new definitions of a Yetter-Drinfel`d condition and module comonad and comodule monad in the obvious way. Observe that these structures appeared in \exref{ex 2}. 
But we also may consider module comonads or comodule monads with non-trivial (co)cycles, as they appear in \exref{ex 6}.


\bigskip

The benefit of having defined paired wreaths and Hopf data in 2-categories is that one may study different types of crossed (co)products in much larger class of cases, taking for 
$\K$ to be any 2-category that one may pick up. From those 
whose 0-cells are certain elements from some vector space, to those whose 0-cells are proper 2-categories, just to mention some of them. 

\medskip

As we pointed out, one may use generalized definitions of 2-cocycles, twisted actions and (co)module (co)algebras 
and study in this new setting the known constructions which are done with the classical definitions. For example, the Brauer group of $H$-module algebras, for a quasi-triangular Hopf algebra 
$H$ over a field, Galois objects, etc.

\end{document}